\documentclass[11pt]{amsart}
\usepackage{fullpage}
\usepackage{amssymb,amsmath,amsthm,amscd,mathrsfs,graphicx, color,mathtools}
\usepackage[cmtip,all,matrix,arrow,tips,curve]{xy}
\usepackage[colorlinks=true]{hyperref}
\usepackage{cleveref}
\usepackage{tikz}
\usepackage{pdfpages}
\usepackage{mathpazo}

\newif\iffullversion
\fullversiontrue 

\numberwithin{equation}{section}
\newtheorem{teo}{Theorem}[section]
\newtheorem{pro}[teo]{Proposition}
\newtheorem{lem}[teo]{Lemma}
\newtheorem{cor}[teo]{Corollary}

\theoremstyle{definition}

\newtheorem{notation}[teo]{Notation}

\theoremstyle{remark}
\newtheorem{rem}[teo]{Remark}

\newcommand{\calB}{\mathcal{B}}

\newcommand{\calD}{\mathcal{D}}

\newcommand{\calO}{\mathcal{O}}
\newcommand{\calL}{\mathcal{L}}
\newcommand{\calM}{\mathcal{M}}
\newcommand{\calN}{\mathcal{N}}

\newcommand{\calX}{\mathcal{X}}

\newcommand{\GL}{\operatorname{GL}}
\newcommand{\PGL}{\operatorname{PGL}}

\newcommand{\SL}{\operatorname{SL}}
\newcommand{\SU}{\operatorname{SU}}

\newcommand{\Pic}{\operatorname{Pic}}

\newcommand{\Sym}{\operatorname{Sym}}
\newcommand{\Bl}{\operatorname{Bl}}

\newcommand{\Aut}{\operatorname{Aut}}

\newcommand{\Sing}{\operatorname{Sing}}

\newcommand{\codim}{\operatorname{codim}}
\newcommand{\diag}{\operatorname{diag}}
\newcommand{\Spec}{\operatorname{Spec}}

\newcommand{\Proj}{\operatorname{Proj}}
\newcommand{\Stab}{\operatorname{Stab}}
\renewcommand{\O}{\operatorname{O}}

\newcommand{\curv}{{\operatorname{curve}}}

\newcommand{\ZZ}{\mathbb{Z}}

\newcommand{\QQ}{\mathbb{Q}}
\newcommand{\RR}{\mathbb{R}}
\newcommand{\CC}{\mathbb{C}}
\newcommand{\TT}{\mathbb{T}}

\newcommand{\PP}{\mathbb{P}}
\newcommand{\FF}{\mathbb{F}}
\newcommand{\LL}{\mathbb{L}}

\newcommand{\SSS}{\mathbb{S}}

\newcommand{\BG}[1][\Gamma]{{\calB_4/#1}}

\newcommand{\BBG}[1][\Gamma]{{(\calB_4/#1)^*}}
\newcommand{\BBGm}[1][\Gamma_m]{{(\calB_4/#1)^*}}
\newcommand{\oBG}[1][\Gamma]{{\overline{\calB_4/#1}}}
\newcommand{\oBGm}[1][\Gamma_m]{{\overline{\calB_4/#1}}}

\newcommand{\GIT}[1][\calM]{{#1}^{\operatorname{GIT}}}
\newcommand{\MK}[1][\calM]{{#1}^{\operatorname{K}}}

\newcommand{\hzero}{{\widehat 0}}
\newcommand{\hone}{{\widehat 1}}
\newcommand{\htwo}{{\widehat 2}}

\title[Compactifications of moduli of cubic surfaces]{Non-isomorphic smooth compactifications of the moduli space of cubic surfaces}

\author[S. Casalaina-Martin]{Sebastian Casalaina-Martin}
\address{University of Colorado, Department of Mathematics, Boulder, CO 80309}
\email{casa@math.colorado.edu}

\author[S. Grushevsky]{Samuel Grushevsky}
\address{Department of Mathematics and Simons Center for Geometry and Physics, Stony Brook University, Stony Brook, NY 11794-3651}
\email{sam@math.stonybrook.edu}

\author[K. Hulek]{Klaus Hulek}
\address{Institut f\"ur Algebraische Geometrie, Leibniz Universit\"at Hannover, 30060 Hannover, Germany}
\email{hulek@math.uni-hannover.de}

\author[R. Laza]{Radu Laza}
\address{Stony Brook University, Department of Mathematics, Stony Brook, NY 11794-3651}
\email{rlaza@math.stonybrook.edu}

\date{\today}
\thanks{Research of the first author is supported in part by grant from the Simons Foundation \# 581058.
Research of the second author is supported in part by NSF grants DMS-18-02116 and DMS-21-01631.
Research of the third author is supported in part by DFG grant Hu-337/7-2.
Research of the fourth author is supported in part by NSF grant DMS-21-01640.
}

\begin{document}
\begin{abstract}
The well-studied moduli space of complex cubic surfaces has three different, but isomorphic, compact realizations: as a GIT quotient $\GIT$, as a Baily--Borel compactification of a ball quotient $\BBG$, and as a compactified $K$-moduli space. From all three perspectives, there is a unique boundary point corresponding to non-stable surfaces. From the GIT point of view, to deal with this point, it is natural to consider the Kirwan blowup $\MK\to \GIT$, while from the ball quotient point of view it is natural to consider the toroidal compactification $\oBG\to \BBG$. The spaces $\MK$ and $\oBG$ have the same cohomology and
and it is therefore natural to ask whether they are isomorphic.
Here we show that this is in fact {\em not} the case.
Indeed, we show the more refined statement that $\MK$ and $\oBG$ are equivalent in the Grothendieck ring, but not $K$-equivalent. Along the way, we establish a number of results and techniques for dealing with singularities and canonical classes of Kirwan blowups and toroidal compactifications of ball quotients.
\end{abstract}

\maketitle
\section{Introduction}\label{sec:intro}
The four-dimensional moduli space $\calM$ of smooth complex cubic surfaces is one of the gems of classical algebraic geometry. As a moduli space of hypersurfaces, it comes with a GIT model $\GIT$, well-studied via the classical invariant theory (see \cite[Ch.9]{DoCAG}, \cite{DvG}).   Through an auxiliary  construction with cubic threefolds, there is a Hodge-theoretic ball quotient model $\BG$ of the moduli space $\calM$,  and its Baily--Borel compactification~$\BBG$ \cite{ACTsurf} (see also \cite{DvGK} for an alternative construction of the ball quotient model, involving $K3$ surfaces instead, and  \cite{KudlaRapoport, ZhiweiZheng} for a construction via abelian varieties of Picard type, which also gives information on the field of definition of the period map).
Finally, as cubic surfaces are Fano, there is a $K$-stable compactification \cite{OSS}.  These three models of $\calM$, each of a totally different nature, turn out to be isomorphic; in particular $\GIT\cong \BBG$ by \cite[Thm. 3.17]{ACTsurf}.

The moduli space $\calM$ is open in $\GIT$, and the complement $\GIT-\calM$, which we refer to as the boundary or, equivalently, the discriminant, is an irreducible divisor. All of the points in the boundary, with the exception of one,
parameterize cubic surfaces with $A_1$ singularities, which are GIT stable. The remaining point is the unique GIT boundary point $\Delta_{3A_2}\in \GIT$ which, from the GIT  (and also $K$-moduli) point of view, corresponds to the unique
strictly polystable orbit of cubic surfaces with $3A_2$ singularities. From the ball quotient perspective, under the identification $\GIT\cong \BBG$, $\Delta_{3A_2}$ is the unique cusp $c_{3A_2}$ of $\BBG$.

For the general GIT setup, in the case where one is taking a GIT quotient of a smooth projective variety, Kirwan has introduced a procedure, usually called the Kirwan blowup of the GIT moduli space, which we denote $\MK\to \GIT$ in our case. The crucial point is that in general a Kirwan blowup is always a desingularization in the sense that it only has finite quotient singularities; alternatively, in the language of stacks, the natural quotient stack structure for the Kirwan blowup gives a smooth Deligne--Mumford stack. Similarly, to improve the singularities of the Baily--Borel compactification, one considers toroidal compactifications. Additionally, in some cases, one can attach modular meaning to some toroidal compactifications (e.g., \cite{Alexeev-A}). In general, a toroidal compactification depends on some choice of a fan, but for the case of ball quotients, the choice is unique, and we denote the toroidal compactification in our particular case by $\oBG\to \BBG$. In our particular situation,  it is a coincidence that $\GIT\cong \BBG$ has only finite quotient singularities. However, it still makes sense to consider the Kirwan blow-up $\MK$ and the toroidal compactification $\oBG$ respectively, in the sense that they have a better chance of having modular interpretations (see e.g., \cite{ZhangCubic} for some related work) that would give rise to natural Deligne--Mumford moduli stacks.

In our case, both $\MK$ and $\oBG$ are blowups of the same point $\Delta_{3A_2}(=c_{3A_2})$ on the same space $\GIT\cong \BBG$.  In \cite{cohcubics}, we studied the cohomology of various compactifications for moduli of cubic threefolds and cubic surfaces. In particular, we showed that $\MK$ and $\oBG$ have the same cohomology, and we asked if $\MK$ and $\oBG$ were in fact isomorphic.
This seems very plausible especially in the context of the recent work of Gallardo--Kerr--Schaffler \cite{GaKeSch} on the moduli of marked cubic surfaces. Namely, recall that the moduli space of cubic surfaces has a natural $W(E_6)$-cover $\calM_m$ obtained by labeling the $27$ lines on the cubic.  Naruki \cite{naruki} proved that this marked moduli space $\calM_m$ admits a smooth normal crossing compactification $\overline{\calN}$ (see also \cite{HKT09} for further modular interpretations attached to $\overline{\calN}$). On the other hand, Allcock--Carlson--Toledo \cite{ACTsurf} noticed that the ball quotient model $\BBG$ comes with a natural marked cover $\BBGm$ (with $\Gamma_m\unlhd \Gamma$ and $\Gamma/\Gamma_m\cong W(E_6) \times \{ \pm 1 \}$)
and that the associated toroidal compactification $\oBGm$ behaves similarly to $\overline{\calN}$. This was further clarified by \cite{GaKeSch} (see also Remark~\ref{rem-gks} for a quick alternative proof), who showed that in fact the Naruki and the marked toroidal compactifications agree: $\overline{\calN}\cong \oBGm$.
Returning to the unmarked case, the main result of our paper is that, surprisingly,
$\MK$ and $\oBG$ are {\em not} isomorphic. Furthermore, we investigate and explain the geometry underlying this phenomenon.

\begin{teo}\label{T:mainNonIso}
Neither the birational map $f:\MK\dashrightarrow\oBG$ from the Kirwan blowup of the GIT moduli space of cubic surfaces to the toroidal compactification of the ball quotient model, which is the identity on the moduli space $\calM$ of smooth cubic surfaces, nor its inverse $f^{-1}$, extend to a morphism of the compactifications.
\end{teo}

We provide a quick proof of this theorem by showing that the strict transform in $\oBG$ of the discriminant divisor (the closure of the locus of cubics with an $A_1$ singularity) in $\BBG=\GIT$ meets the exceptional divisor in $\oBG$ generically transversally, whereas this is not the case in $\MK$. This shows {\em a priori} that the period map~$f$ does not extend to an {\em isomorphism}. Using straightforward arguments from higher-dimensional birational geometry, using that both $\MK$ and $\oBG$ are $\QQ$-factorial, one can deduce from this that neither the period map~$f$, nor its inverse $f^{-1}$, extends to a morphism.

While \Cref{T:mainNonIso} shows that the Kirwan and toroidal compactifications are not {\em naturally} isomorphic,
one could still wonder if the agreement of the Betti numbers, observed in \cite{cohcubics}, could be explained by an abstract isomorphism
as complex projective varieties. To this end we strengthen \Cref{T:mainNonIso} by showing that $\MK$ and $\oBG$ are not abstractly isomorphic, and in fact not even $K$-equivalent.

Recall that two birational normal projective $\QQ$-Gorenstein varieties $X$ and $Y$ are said to be  $K$-equivalent if there exists a smooth variety
$Z$ dominating birationally both of them
\begin{equation}\label{E:Kequiv}
\xymatrix{
&Z\ar[ld]_g \ar[rd]^h&\\
X\ar@{<-->}[rr]&&Y
}
\end{equation}
such that $g^*K_X \sim_{\QQ} h^*K_Y$. Two facts about $K$-equivalent varieties that motivate our interest in this question are the following.
First, it is known that smooth $K$-equivalent varieties have the same Betti numbers \cite[Thm.~1.1]{batyrev96} (see for instance \cite[Ch.~4, Thm.~3.1]{popaNote} for the case where $X$ and $Y$ are not assumed to be Calabi--Yau). Second, birational normal projective $\QQ$-Gorenstein varieties with canonical singularities, such that their canonical bundles are nef, are known to be  $K$-equivalent \cite[Lem.~4.4]{kawamata_02}.
We are, however, in neither of these situations: first, the Kirwan and toroidal compactifications are singular (with finite quotient singularities), and second, the canonical bundles of the compactifications are far from being nef, as the spaces are birational to the $\QQ$-Fano variety $\GIT$. As it turns out, these two compactifications are {\em not} $K$-equivalent.
\begin{teo}\label{T:mainNonK}
The Kirwan compactification $\MK$ and the toroidal compactification $\oBG$ of the moduli space of smooth cubic surfaces $\calM$ are {\em not} $K$-equivalent.
\end{teo}

We prove \Cref{T:mainNonK} by showing that the top self-intersection numbers of the canonical bundles of the two compactifications are different.
Clearly, \Cref{T:mainNonK} implies \Cref{T:mainNonIso} (since it implies that  $f$ does not extend to an isomorphism).
However, we prefer first to give an independent proof of \Cref{T:mainNonIso} in  \Cref{sec:ballquoitient}, as this also helps elucidate the geometry of the compactifications. This approach can also be used in other contexts such as cubic threefolds or other Deligne--Mostow spaces.
A further major ingredient of our proof is an explicit geometric recipe for computing the canonical class of a Kirwan blowup of a GIT quotient. This works in full generality, so long as the strictly semi-stable locus has codimension at least $2$, which we explain in~\Cref{sec:KKirwan}, and which may be of independent interest; see also  \Cref{R:S-BGLM} regarding the singularities of these spaces.

In the absence of odd cohomology (as is the case here), another possible explanation for the equality of Betti numbers of $\MK$ and $\BBG$ could be their $\LL$-equivalence in the Grothendieck ring of varieties. We prove that in fact the classes of these compactifications are equal in the Grothendieck ring of varieties.
\begin{teo}\label{T:mainL}
The Kirwan compactification $\MK$ and the toroidal compactification $\oBG$ of the moduli space of smooth cubic surfaces $\calM$ are equivalent in the Grothendieck ring of varieties.
\end{teo}

The structure of the paper is as follows. We start in~\Cref{sec:example} by discussing a simple explicit example, for motivation. This example is of two birational surfaces, which are equivalent in the Grothendieck ring, but not $K$-equivalent. These are obtained as different blowups of~$\PP^2$ supported at one point. While this is not a GIT problem, the example demonstrates some of the features similar to our case of the moduli of cubic surfaces, and serves as motivation. In~\Cref{sec:calM} we recall the constructions and the geometry of the compactifications of the moduli space of cubic surfaces, and give the formulas for the discriminant and boundary divisors in local coordinates, as well as the crucial computation of the finite stabilizers of points on them. Some of the proofs are by detailed explicit computations, which are given in \Cref{sec:app}. In particular, in~\Cref{sec:calM} we obtain the resulting non-transversality results for the Kirwan blowup needed for proving~\Cref{T:mainNonIso}.
In  \Cref{sec:ballquoitient} we discuss the geometry of the ball quotient model in detail,  obtain the resulting transversality results for the toroidal compactification,   and thereby give a direct proof of \Cref{T:mainNonIso}.  In \Cref{sec:canonicalbundle} we provide details on the canonical bundles of the  ball quotient models.

Going further, in~\Cref{sec:KKirwan} we develop the general machinery for relating the canonical bundle of a GIT compactification and its Kirwan desingularization. We then specialize to our case, and compute the canonical bundle of $\MK$.
Finally, we prove our main results in~\Cref{sec:proofnonKequiv}, where we establish the non-$K$-equivalence of the Kirwan blowup $\MK$ and the toroidal compactification $\oBG$, and in~\Cref{sec:proofLequiv}, where we show their equivalence in the Grothendieck ring. The first requires a discussion of the top self-intersection numbers of canonical bundles on these varieties, the second relies on a concrete description of the boundary of the Kirwan blowup as a quotient of a toric variety by a finite group. In \Cref{sec:app} we collect a number of results whose proofs are based on explicit computations in the Luna slice.

\begin{rem}

One of our motivations for this paper is our systematic investigation of the moduli space of cubic threefolds and its compactifications (e.g., \cite{cohcubics} and references within). In that case as well, there is a toroidal compactification of a ball quotient model and a Kirwan blowup of the natural GIT quotient. Furthermore, as here, the two models have the same cohomology. Methods similar to those used in the proof of \Cref{T:mainNonIso} show that for moduli of cubic threefolds the natural period map does not extend to an isomorphism, either. However, the more refined results (\Cref{T:mainNonK} and \Cref{T:mainL}) seem at the moment out of reach for cubic threefolds.

We also note that there are other related moduli spaces where these techniques apply, such as the moduli space of cubic surfaces with a line, see \Cref{rem:cubics with a line}. Again, we will not pursue this here.
\end{rem}

\subsection*{Acknowledgements}
We thank Dan Abramovich for asking about the $K$-equivalence of the moduli spaces.
We are grateful to Damiano Fulghesu for a conversation on equivariant Chow rings. We thank John Christian Ottem for asking about the $\LL$-equivalence of the compactifications $\MK$ and $\oBG$.  We are grateful to David Rydh for discussions and explanations regarding the Kirwan desingularization as a blowup of an Artin stack, which helped to clarify our thinking and to
Daniel Allcock for answering questions on the ball quotient picture. Our thanks also go to Bert van Geemen who helped us clear up a question concerning the ring of invariants. We thank  Mathieu Dutour Sikiric for a computation with polytopes. This
work was partially supported by the Swedish Research Council under grant no.~2016-06596 while S.~Grushevsky, K.~Hulek, and R.~Laza visited the Institut Mittag-Leffler in Djursholm, Sweden during the fall of 2021. S.~Grushevsky thanks the Weizmann Institute of Science for its hospitality in spring 2022 when this paper was finished.

\section{A motivating example}\label{sec:example}
Before discussing the case of the moduli of cubic surfaces, we present an elementary example that captures some of the essential aspects of our arguments. While we are not aware of a global moduli interpretation of the example discussed here, the motivation and the construction of the example comes from the local study of semi-stable reduction for curves with an $A_2$ singularity  (see esp.~\cite{cml2}), and it is at least morally related to the moduli of cubic surfaces discussed here (e.g., see \cite[\S 1]{cubics}).

Namely, consider the pair $(M,D)$ consisting of $M\cong \PP^2$ together with a cuspidal cubic $D$. Let $o\in D$ be the cusp point. Define $M'$ to be the (standard) blowup of $M\cong \PP^2$ at $o$, with the exceptional divisor $E'\subseteq M'$, and let~$\widehat M$ be the standard log resolution of the cusp, obtained by $2$ further blowups of $M'$, with exceptional divisors $E_1,E_2,E_3$ over $M$. We label the exceptional divisors on $\widehat M$ in such a way that $(E_i)^2=i-4$; i.e., $E_1$ is the strict transform of $E'$. We contract $E_1$ and $E_2$ on $\widehat M$ and obtain $\overline M$, which will have two quotient singularities of types $\frac{1}{2}(1,1)$ (or equivalently $A_1$) and $\frac{1}{3}(1,1)$ (e.g., simply note that $E_1^2=-3$, $E_2^2=-2$, and that $E_1$ and $E_2$ do not intersect) and we denote by $\overline E\subseteq \overline M$ the exceptional divisor of the blowdown~$\overline M\to M$, so that $\overline E$ is the image of $E_3\subseteq \widehat M$. We obtain the following diagram:
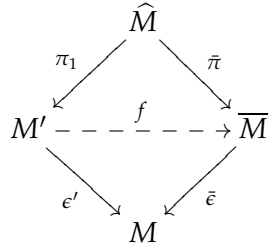
\begin{equation*}
\xymatrix{
&\widehat M \ar_{\pi_1}[dl]\ar@{->}^{\bar\pi}[dr]&\\M' \ar[dr]_{\epsilon'}\ar@{-->}[rr]^{f}&&\overline M\ar[dl]^{\bar\epsilon}\\ &M
}
\end{equation*}
Clearly, $\overline M$ is a blowup of $M$ at $o$, and in fact it is a single weighted blowup of $M$ at $o$. In conclusion, both $M'$ and $\overline M$ are blowups of $M$ at $o$.
Both are $\QQ$-factorial with klt singularities, but they are non-isomorphic (e.g., $M'$ is smooth, while $\overline M$ is singular).
More in line with our arguments, we note that the birational map $f:M'\dashrightarrow \overline M$ does not extend to an isomorphism, since  the exceptional divisors $E'\subseteq M'$ and $\overline E\subseteq \overline M$, and the strict transforms of $D$ (denote them $D'\subseteq M'$ and $\overline D\subseteq \overline M$) are non-transversal in $M'$, but are transversal in $\overline M$. In fact, note that $(\overline M, \overline D+\overline E)$ is dlt, while $(M', D'+E')$ is not. This shows that even though $\overline M$ is singular, when accounting for the ``discriminant'' $D$, the ``correct'' resolution of $M$ is in fact $\overline M$ and not $M'$ (which is smooth!). Clearly, $M'$ and $\overline M$ have the same Betti numbers, and are  equal in the Grothendieck ring of varieties. On the other hand, an elementary computation shows that $M'$ and $\overline M$ are not $K$-equivalent. Namely, $K_{M'}^2=8$ (as $M'$ is a single regular blowup of $\PP^2$), while $K_{\overline M}^2=6+\frac{1}{3}$, since $K_{\widehat M}=\bar\pi^*K_{\overline M}-\frac{1}{3} E_1$ (as the discrepancy of a quotient singularity of type $\frac{1}{n}(1,1)$ is $\frac{2-n}{n}$).

We note that the weighted blowup $\overline M\to M$ is motivated by \cite{cml2}, which discusses the simultaneous semi-stable reduction for curves with certain singularities. Indeed, locally near the cusp $o$, the pair $(M,D)$ can be identified with the versal deformation for~$A_2$ singularities, with~$D$ being the discriminant. From this perspective, it is natural to consider (locally near $o$) the $W(A_2)$ cover $\widetilde M$ of~$M$ branched over~$D$. On this cover, the discriminant~$D$ pulls back to the~$A_2$  hyperplane arrangement. The standard blowup $\widehat{\widetilde M}$ at $\widetilde o\in \widetilde M$ (the preimage of $o$) leads to a normal crossing discriminant (see \cite{cml2} for the general construction and discussion). Clearly $W(A_2)$ acts on $\widehat{\widetilde M}$, and~$\overline M$ can be recovered as the quotient $\widehat{\widetilde M}/W(A_2)$ (again, the discussion is meant locally near $o$). The main point of this construction is that it takes into account the monodromy around the discriminant divisor, and thus it leads to a ``modular'' resolution $\overline M$ of the pair $(M,D)$ (see \cite{cml2} for precise statements).

Similarly, returning to our setup in the current paper, the toroidal compactification of the moduli space of cubic surfaces is a ``modular'' blowup of the Baily--Borel compactification, which takes into account the monodromy around the discriminant divisor.  In contrast, the Kirwan resolution of the GIT quotient does not see the monodromy, and consequently the Kirwan blowup behaves more like the standard blowup $M'\to M$. What we see is that while the GIT compactification and the Baily--Borel compactification for the moduli of cubic surfaces are isomorphic as projective varieties, the natural stack structure (even at certain stable points) is different. Namely, at the generic point of the discriminant divisor, corresponding to a cubic with an $A_1$ singularity, on the GIT side there are no extra  automorphisms (i.e., the stabilizer of the generic point of the discriminant is the diagonal $\mu_4$ in $\SL(4,\CC)$, which induces the trivial automorphism of the cubic surface), while on the Hodge theoretic side there is an extra automorphism given by the Picard--Lefschetz transformation corresponding to the nodal degenerations  (compare \cite[p.125]{olsson}).

\section{The GIT models for the moduli of cubic surfaces}\label{sec:calM} As is the case in general for hypersurfaces, the moduli of cubic surfaces has a natural compact model: {\em the GIT compactification $\GIT$}. For cubic surfaces, $\GIT$ is well understood via classical invariant theory; in particular,  $\GIT\cong \PP(1,2,3,4,5)$. Here, we are interested in {\em the Kirwan blowup $\MK\to \GIT$}, which is obtained by blowing up the unique GIT
strictly polystable
boundary point $\Delta\in \GIT$  according to a general procedure due to Kirwan \cite{kirwanblowup}. After reviewing $\GIT$ and $\MK$,  we discuss the local structure of the Kirwan blowup $\MK$ along the exceptional divisor $D_{3A_2}$. We conclude in \Cref{P:discMK} that the exceptional divisor $D_{3A_2}$ and the strict transform $\widetilde D_{A_1}$ of the discriminant divisor  do not meet, even generically, transversally in $\GIT$.

\subsection{Preliminaries on moduli of cubic surfaces}
\label{subsection:prelimmoduli}
We denote by $\calM \coloneqq \PP H^0(\PP^3,\calO_{\PP^3}(3))^\circ/\SL(4,\CC)$ the four-dimensional moduli space of smooth (complex) cubic surfaces, where $\PP H^0(\PP^3,\calO_{\PP^3}(3))^\circ$ denotes the locus of smooth cubic surfaces embedded in $\PP^3$. We denote by
$$
 \GIT\coloneqq \PP H^0(\PP^3,\calO_{\PP^3}(3))/\!\!/_{\calO(1)} \SL(4,\CC)
$$
the GIT compactification, and by
\begin{equation}\label{E:KBl}
 \pi:\MK\longrightarrow \GIT
\end{equation}
the Kirwan resolution of $\GIT$. The GIT stability for cubic surfaces has been completely described (see e.g.,~\cite[\S 7.2(b)]{mukai}). For a cubic surface $S\subseteq \PP^3$:
\begin{itemize}
\item $S$ is stable if and only if it has at worst $A_1$ singularities,
\item $S$ is semi-stable if and only if it is stable, or has at worst $A_2$ singularities, and does not contain the axes of the $A_2$ singularities,
\item $S$ is strictly polystable if and only if it is projectively equivalent to the so-called $3A_2$ cubic surface
 $$
 S_{3A_2}\coloneqq \lbrace x_0x_1x_2+x_3^3=0\rbrace\,,
 $$
 which has exactly $3$ singular points, each of which is an $A_2$ singularity.
\end{itemize}
It is a classical result (see~\cite[(2.4)]{DvGK}) that the GIT compactification
$$
 \GIT\cong \PP(1,2,3,4,5)
$$
is a weighted projective space. We will denote by
$$
 D_{A_1}\subseteq \GIT
$$
the so-called {\em discriminant divisor}, that is the closure of the locus of (stable) cubics with an $A_1$ singularity. This divisor is irreducible, as the locus of corresponding cubics in $\PP H^0(\PP^3,\calO_{\PP^3}(3))$ is irreducible; the general point is given by the orbit of the locus of cubics of the form $x_0q(x_1,x_2,x_3)+f(x_1,x_2,x_3)$ with $q(x_1,x_2,x_3)$ a smooth quadric, and $f(x_1,x_2,x_3)$ a cubic.

We denote by
$$
 R\subseteq \GIT
$$
the so-called {\em Eckardt divisor}, which generically parameterizes smooth cubic surfaces having an Eckardt point --- a point that lies on three lines contained in the cubic surface. This divisor is also irreducible, as the corresponding locus in $\PP H^0(\PP^3,\calO_{\PP^3}(3))$ is irreducible; the locus of smooth Eckardt cubics is the orbit of the locus of cubics of the form $x_0^2\ell(x_1,x_2,x_3)+f(x_1,x_2,x_3)$ with $\ell(x_1,x_2,x_3)$ a linear form and $f(x_1,x_2,x_3)$ a cubic, with the line and cubic meeting transversally. In the coordinates of the previous equation, the point $(1:0:0:0)$ is then an Eckardt point, and the three lines through the Eckardt point are the ones determined by the Eckardt point and the intersection $\{\ell =f=0\}$ in the hyperplane at infinity. Such an Eckardt cubic in these coordinates has an involution given by $x_0\mapsto -x_0$. Note that a general smooth cubic surface with an Eckardt point has a unique Eckardt point, and has automorphism group $\ZZ_2$, and moreover, the Eckardt divisor $R$ contains the locus of {\em all} smooth cubic surfaces that have any non-trivial automorphism (see e.g.,~\cite[Table 1 and Fig.~1]{DoDu}).

From the description of stability given above, it follows that $\GIT$ contains a unique strictly polystable point $\Delta_{3A_2}\in\GIT$ corresponding to the orbit of the $3A_2$ cubic $S_{3A_2}$, and consequently the Kirwan resolution~$\MK$ of~$\GIT$ is a blowup with center supported at~$\Delta_{3A_2}$ (see e.g., \cite{kirwanhyp} and~\cite{ZhangCubic}); we denote by
$$
 D_{3A_2}\subseteq \MK
$$
the exceptional divisor, which is irreducible (as is the case for exceptional divisors in Kirwan blowups, being the quotient of an open subset of the blowup of a smooth irreducible subvariety of a smooth irreducible variety). Note that $\MK$ and $D_{3A_2}$ are, by construction, smooth up to finite quotient singularities. We will denote by $\widetilde D_{A_1}\subseteq \MK$ the strict transform of the discriminant divisor $D_{A_1}\subseteq \GIT$, and by $\widetilde R\subseteq\MK$ the strict transform of the Eckardt divisor $R\subseteq \GIT$.

\smallskip

\subsection{Geometry of $\GIT$ as a weighted projective space}
In full generality, consider a weighted projective space $\PP(q_0,\dots, q_n)=\Proj\CC[x_0,\dots,x_n]=(\CC^{n+1}-\{0\})/\CC^*$, where the weight of $x_i$ (resp.~the weight of the action of $\CC^*$ on $x_i$) is $q_i$ for $i=0,\dots,n$, where, without loss of generality, we assume that $\gcd(q_0,\dots, q_n)=1$ \cite[Prop.~p.37]{dolgachevWP}. Denoting $G_{\vec q}\coloneqq \mu_{q_0} \times \dots \times\mu_{q_n}$, where $\mu_{\ell}$ is the multiplicative group of roots of unity of order $\ell$, and letting the group~$G_{\vec q}$ act diagonally on~$\PP^n$, we can express the weighted projective space as the quotient
\begin{equation}\label{E:wPn}
\PP(q_0,\dots, q_n) = \PP^n / G_{\vec q}\,,
\end{equation}
with quotient map
\begin{equation}\label{E:wPn-h}
h:\PP^n \longrightarrow \PP(q_0,\dots, q_n)\,.
\end{equation}
This is the map of spaces associated to the identification of the graded ring $\CC[x_0,\dots,x_n]$ as the subring $\CC[y_0^{q_1},\dots,y_n^{q_n}]\subseteq \CC[y_0,\dots,y_n]$ of the standard graded polynomial ring, which can be viewed as the invariant ring for the group $G_{\vec q}$ acting diagonally.

While $G_{\vec q}$ need not be cyclic, the weighted projective space is locally a cyclic quotient. It is covered by the open sets $U_i=D_+(x_i)\coloneqq \{\mathfrak p\in \PP(q_0,\dots,q_n): x_i\notin \mathfrak p\}$ (i.e.,~$U_i$ is the image of $\{x_i \neq 0\}\subseteq \CC^{n+1}-\{0\}$ under the quotient map~$h$), and one has
$$
U_i \cong \CC^n/\mu_{q_i}=\CC^n/\langle (\zeta_{q_i}^{q_0}, \dots ,\zeta_{q_i}^{q_{i-1}}, \zeta_{q_i}^{q_{i+1}}, \dots ,\zeta_{q_i}^{q_{n}}) \rangle\,,
$$
where $\zeta_{q_i}$ is a primitive root of unity of order $q_i$. This identification comes from the identification $U_i=\Spec(\CC[x_0,\dots,x_n]_{(x_i)})$, where $\CC[x_0,\dots,x_n]_{(x_i)}$ is the subring of elements of degree $0$ in the localized ring $\CC[x_0,\dots,x_n]_{x_i}$, and the identification of the ring $\CC[x_0,\dots,x_n]_{(x_i)}$ with the subring of $\CC[z_0,\dots,z_{i-1},z_{i+1},\dots,z_n]^{\mu_{q_i}} \subseteq \CC[z_0,\dots,z_{i-1},z_{i+1},\dots,z_n]$ of invariants for the diagonal action of the cyclic group $\mu_{q_i}=\langle (\zeta_{q_i}^{q_0}, \dots ,\zeta_{q_i}^{q_{i-1}}, \zeta_{q_i}^{q_{i+1}}, \dots ,\zeta_{q_i}^{q_{n}}) \rangle$.

It is now straightforward to apply the Reid--Shepherd-Barron--Tai criterion to $\GIT$, by local computations in these charts.
\begin{lem}
The space $\GIT \cong \PP(1,2,3,4,5)$ has canonical singularities.
\end{lem}
\begin{proof}
We first note that $U_0 \cong \CC^4$. To explain the other charts we first treat the chart
$$
U_4= \CC^4/\langle g_4 \rangle\,,
$$
where
$$
g_4\coloneqq (\zeta_{5}, \zeta_{5}^2, \zeta_{5}^3, \zeta_{5}^4)\,.
$$
There is only one fixed point of this finite group action, namely the origin. Here the RSBT criterion tells us that we have to check for all non-trivial powers $g_4^k$ that the inequality
$$
\lfloor \tfrac{k}{5} \rfloor + \lfloor \tfrac{2k}{5} \rfloor + \lfloor \tfrac{3k}{5} \rfloor + \lfloor \tfrac{4k}{5} \rfloor \geq 1
$$
holds. Altogether we find one singularity in this chart, namely the point $P_4=(0:0:0:0:1)$; i.e.,~the image of the point $(0,0,0,0,1)\in \CC^{n+1}-\{0\}$ under the quotient map~$h$.
 The singularity at this point is canonical. The other open sets $U_i$ can be treated similarly. The situation for $U_2$ is completely analogous, and we find one further canonical singularity, namely $P_2=(0:0:1:0:0)$.
For $U_3$ we have to consider
$$
U_3= \CC^4 / \langle g_4 \rangle = \CC^4 / \langle i,-1,-i,i \rangle\,.
$$
Once again, we find one canonical singular point, namely $P_3=(0:0:0:1:0)$.
Finally in the chart
$$
U_1= \CC^4 / \langle g_2 \rangle = \CC^4 / \langle -1,1,-1,-1 \rangle
$$
we have a $1$-dimensional fixed locus, namely the line
$$
L\coloneqq\lbrace x_0=x_2=x_4=0\rbrace\,.
$$
We also note that $g_4^2=g_2$ and that $P_3 \in L$. Outside $P_3$ we have a transversal
singularity along $L$ of type $\CC^3/\langle (-1,-1,-1) \rangle$.
\end{proof}
In the proof of the Lemma we have also verified that
$$
\Sing \GIT= \{P_2\} \cup \{P_4\} \cup L\,,
$$
as already stated in~\cite[\S 6.9]{DvGK}.

We note that the space $\GIT=\PP(1,2,3,4,5)$ is $\QQ$-factorial (since it only has finite quotient singularities), and is thus $\QQ$-Gorenstein, but it is {\em not} Gorenstein.
Indeed, the line bundle $\calO_{\PP^n}(1)$ descends under the cover $h$ \eqref{E:wPn-h} to a $\QQ$-Cartier divisor on $\PP(q_0,\dots,q_n)$, which by abuse of notation we denote by $\calO_{\PP(q_0,\dots,q_n)}(1)$. The lowest multiple of $\calO_{\PP(q_0,\dots,q_n)}(1)$ which is Cartier is then $\calO_{\PP(q_0,\dots,q_n)}\left(\operatorname{lcm}(q_0, \dots ,q_n)\right)$ (e.g., \cite[Prop.~p.63]{fultonToric} or \cite[Ex.~4.1.5 and 4.2.11]{CLS}). For the canonical bundle,
using the fact that the weighted projective space is a toric variety, and that the covering map $h$ is ramified along the toric divisors, one obtains the standard formula (e.g., \cite[Thm.~3.3.4]{dolgachevWP})
\begin{equation}\label{E:KWP}
K_{\PP(q_0,\dots,q_n)}=\left(-\sum q_i \right)\calO_{\PP(q_0,\dots,q_n)}(1)\,.
\end{equation}
Thus in our case
\begin{equation}\label{E:can}
K_{\PP(1,2,3,4,5)}=-15\calO_{\PP(1,2,3,4,5)}(1)\,,
\end{equation}
and its smallest multiple that is Cartier is $4K_{\PP(1,2,3,4,5)}$.

Classical invariant theory for cubic surfaces explicitly identifies the geometric divisors $D_{A_1}$ (the nodal or discriminant) and $R$ (the Eckardt) divisors inside $\GIT\cong \PP(1,2,3,4,5)$. For further use, we review this. By~\cite[1.3 and 6.4]{DvG} the discriminant $D_{A_1}$ is given by the equation
\begin{equation}\label{E:DinIs}
 (I_8^2-2^6 I_{16})^2= 2^{14}(I_{32} + 2^{-3}I_8I_{24})\,,
\end{equation}
where $I_8,I_{16}, I_{24}, I_{32}, I_{40},I_{100}$ are the standard generators of the ring of invariants of the action of $\SL(4,\ZZ)$ on the space of cubics, and the subscripts denote their degrees. As $I_{100}^2$ is a polynomial in the other invariants listed, these degrees show that we are working with $\PP(8,16,24,32,40)\cong\PP(1,2,3,4,5)$. Moreover, the Eckardt divisor is given by $I_{100}^2=0$ (e.g., \cite[\S 6.5]{DvG}).

Pulling back the defining equation \eqref{E:DinIs} of $D_{A_1}$ to $\PP^4$ under the quotient map $h$ given by~\eqref{E:wPn-h}, we obtain
\begin{equation}\label{E:DA1pullback}
 h^*D_{A_1}=\left\lbrace (y_0^2-2^6 {y_1}^2)^2= 2^{14}(y_3^4 + 2^{-3}y_0y_2^3)\right\rbrace\,,
\end{equation}
where $y_0, \dots, y_4$ are homogeneous coordinates on $\PP^4$. Furthermore, \cite{DvGK} gives the coordinates of the point $\Delta_{3A_2}\in\PP(1,2,3,4,5)$ as
$$
\Delta_{3A_2}=(8:1:0:0:0) \in D_{A_1}\,,
$$
which in particular is a smooth point of $\GIT$. It is also a smooth point of $D_{A_1}\subseteq \PP(1,2,3,4,5)$, as in the local coordinates on the open chart $U_0\cong\CC^4$ the defining equation~\eqref{E:DinIs} of the discriminant divisor becomes $(1-2^6z_1)^2 = 2^{14}(z_3 +2^{-3} z_2)$, and $\Delta_{3A_2}$ corresponds to the point $(1/8^2,0,0,0)$, so that taking partial derivatives of this equation at the point $\Delta_{3A_2}$ gives smoothness of $D_{A_1}$ at $\Delta_{3A_2}$.

We can perform a similar analysis for the Eckardt divisor. As mentioned above, the Eckardt divisor $R$ is defined by $I_{100}^2$, which is an irreducible polynomial in $I_8,\dots,I_{40}$. The exact expression due to Salmon for $I_{100}^2$ in terms of  $I_8,\dots,I_{40}$ no longer seems to be easily accessible  in the literature.  Dardanelli and van Geemen recently rederived it for their paper \cite{DvG}, and provided us with the expression, which we have omitted to save space; for reference, it is now available on the first author's website.
From that description, one can easily see that
\begin{equation}\label{E:Rpullback}
h^*R=\calO_{\PP^4}(25).
\end{equation}

As already noted, $\GIT$ has Picard number $1$. For further reference, we collect here the classes of various Weil divisors on $\GIT=\PP(1,2,3,4,5)$:
\begin{equation}\label{equation:invarianttheory}
\begin{aligned}
K_{\GIT}&=\calO_{\PP(1,2,3,4,5)}(-15) & \\
D_{A_1}&=\calO_{\PP(1,2,3,4,5)}(4)& \text{(Discriminant divisor)}\\
R& =\calO_{\PP(1,2,3,4,5)}(25)& \text{(Eckardt divisor)}
\end{aligned}
\end{equation}
These come from \eqref{E:can}, \eqref{E:DA1pullback}, and \eqref{E:Rpullback}, respectively.
In particular, we note that the following relation holds in $\Pic(\GIT)_\QQ$:
\begin{equation} \label{E:rat2}
K_{\PP(1,2,3,4,5)}= - \frac{15}{4} D_{A_1}\,.
\end{equation}

\begin{rem}\label{rem:Eckardt}
There is an important subtle point that we emphasize here.  The divisor $\calO_{\PP^{19}}(1)$ on $(\PP^{19})^{ss}$ descends as a $\QQ$-divisor to the divisor $\frac{1}{8}\calO_{\PP(1,2,3,4,5)}(1)$.  The discriminant in $(\PP^{19})^{ss}$ has degree $32$ (the discriminant has degree $(n+2)(d-1)^{n+1}$ for degree $d$ hypersurfaces in ~$\PP^{n+1}$) and descends to the Weil divisor $D_{A_1}=\calO_{\PP(1,2,3,4,5)}(4)$ on $\PP(1,2,3,4,5)$.
Similarly, the Eckardt divisor on $(\PP^{19})^{ss}$ has degree $100$ (see, e.g., \cite[\S 6.5]{DvG}, or \cite[Thm.~4.1]{CPS15} for an approach that works for more general Eckardt loci) and descends as a $\QQ$-divisor to the divisor $\frac{1}{2}\calO_{\PP(1,2,3,4,5)}(25)$.  However, this is {\em not} the Eckardt divisor $R$ on $\PP(1,2,3,4,5)$, which, as explained above, has class $\calO_{\PP(1,2,3,4,5)}(25)$.  In other words, it is twice the Eckardt divisor on $(\PP^{19})^{ss}$ that descends to the Eckardt divisor $R$ on $\PP(1,2,3,4,5)$.  One can view this as a reflection of the fact that generic Eckardt cubic surfaces have an extra automorphism, as opposed to the case of generic cubic surfaces with an $A_1$ singularity, which do not have any extra automorphisms.

\end{rem}

\subsection{Local structure of the Kirwan blowup along the exceptional divisor}
We recall some relevant computations from \cite[App.~C]{cohcubics}.
First, to employ the Luna Slice Theorem, we will want to understand the stabilizer of the $3A_2$ cubic surface~$S_{3A_2}$, as well as its action on a Luna slice.

To begin, given a cubic surface $S\subseteq \PP^3$, we denote by $\Aut(S)$ the automorphisms of $S$ (which are automorphisms of $S$ as a subvariety of $\PP^3$, and therefore we naturally have $\Aut(S)\subseteq \PGL(4,\CC)$). From our GIT setup, we are also interested in $\Stab(S)\subseteq \SL(4,\CC)$, the stabilizer subgroup, and, since it is sometimes easier to work with, we will also consider the stabilizer $\GL(S)\subseteq \GL(4,\CC)$. We recall from~\cite[App.~C]{cohcubics} that the former two of these stabilizer groups for~$S_{3A_2}$ are 2-dimensional, but we also want to work out the finite parts explicitly. To this end, we denote
$$
D\coloneqq \{\diag(\lambda _0,\lambda _1,\lambda_2,\lambda_3): \lambda_0\lambda_1\lambda_2=\lambda_3^3\}\subseteq \GL(4,\CC)
$$
an auxiliary group, and observe that there is an isomorphism $\TT^3\cong D$ given by $(\lambda_1,\lambda_2,\lambda_3)\mapsto \diag(\lambda_1^{-1}\lambda_2^{-1}\lambda_3^3,\lambda_1,\lambda_2,\lambda_3)$.
We also want
$$
D'\coloneqq D\cap \SL(4,\CC) = \{\diag(\lambda _0,\lambda _1,\lambda_2,\lambda_3): \lambda_0\lambda_1\lambda_2=\lambda_3^3, \ \lambda_0\lambda_1\lambda_2\lambda_3=1 \}\subseteq \SL(4,\CC)\,,
$$
and note the isomorphism $\TT^2\times \mu_4\cong D'$ given by $(\lambda_1,\lambda_2,i^j)\mapsto \diag(\lambda_1^{-1}\lambda_2^{-1}i^{3j},\lambda_1,\lambda_2,i^j)$.
To see that this map is an isomorphism, note that this certainly gives an inclusion $\TT^2\times \mu_4\hookrightarrow D'$, and given $\diag(\lambda _0,\lambda _1,\lambda_2,\lambda_3)\in D'$, the two equations together give $\lambda_3^4=1$, so that $\lambda_3=i^j$ for some $j$. Finally, we denote
$$
D''\coloneqq \{\diag(\lambda _0,\lambda _1,\lambda_2,1): \ \lambda_0\lambda_1\lambda_2=1 \}\subseteq \SL(4,\CC)\,,
$$
and note the isomorphism $\TT^2\cong D''$ given by $(\lambda_1,\lambda_2)\mapsto \diag(\lambda_1^{-1}\lambda_2^{-1},\lambda_1,\lambda_2,1)$.

We use the notation $\SSS_3$ for the group of matrices obtained from the group of invertible diagonal $3\times 3$ complex matrices by applying all possible permutations of the columns. The determination of the relevant stabilizer groups is parallel to the case of the $3D_4$ cubic threefold, treated in~\cite[Prop.B.6]{cohcubics}, and proceeds by an explicit computation, which we put in~\Cref{S:R3A2Norm2}.

\begin{lem}[$3A_2$-stabilizer]
\label{L:R3A2Norm2}
We have:
\begin{enumerate}
\item The group $\Stab(S_{3A_2}$ is equal to
\begin{equation}\label{E:App-GF3D4}
\Stab(S_{3A_2})=
\left\{\left(
\begin{array}{c|c}
\SSS_3&\\ \hline
&\CC^*\\
\end{array}
\right)\in \SL(4,\CC): \lambda_0\lambda_1\lambda_2=\lambda_3^3\right\},
\end{equation}
where the $\lambda_i$ is the unique non-zero element in the $i$-th row, and
the group $\GL(S_{3A_2})$ is equal to
\begin{equation}\label{E:GL-GF3D4}
\GL(S_{3A_2}))=
\left\{\left(
\begin{array}{c|c}
\SSS_3&\\ \hline
&\CC^*\\
\end{array}
\right)\in \GL(4,\CC): \lambda_0\lambda_1\lambda_2=\lambda_3^3\right\}.
\end{equation}

\item There are central extensions
\begin{equation}\label{E:St3D4Ct}
1\to \mu_4\to \Stab(S_{3A_2}) \to \Aut(S_{3A_2})\to 1\,,
\end{equation}
\begin{equation}\label{E:GLSt3D4Ct}
1\to \CC^*\to \GL(S_{3A_2})\to \Aut(S_{3A_2})\to 1\,.
\end{equation}
\item There are short exact sequences
\begin{equation}\label{E:GLStDs3s2-1}
\xymatrix{
1 \ar[r]& D' \ar[r] \ar@{^(->}[d]& \Stab(S_{3A_2}) \ar@{^(->}[d] \ar[r] & S_3\ar[r] \ar@{=}[d]& 1\\
1 \ar[r]& D \ar[r]& \GL(S_{3A_2}) \ar[r] & S_3\ar[r]& 1\\
}
\end{equation}
with the second being split, so that there
is an isomorphism
\begin{equation}\label{E:GLStDs3s2}
\GL(S_{3A_2})\cong D\rtimes S_3\,,
\end{equation}
where the action of $S_3$ on $D$ is to permute the first three entries $\lambda_0,\lambda_1,\lambda_2$.

\item The connected components of the groups above are $\Stab(S_{3A_2})^\circ = D''\cong \TT^2$, $\GL(S_{3A_2})^\circ= D\cong \TT^3$, and $\Aut(S_{3A_2})^\circ \cong \TT^2$. There is a short exact sequence
\begin{equation}\label{E:G3A2L}
1\to \mu_4\to \Stab(S_{3A_2})/\Stab(S_{3A_2})^\circ \to S_3 \to 1\,,
\end{equation}
and we have $\GL(S_{3A_2})/\GL(S_{3A_2})^\circ \cong \Aut(S_{3A_2})/\Aut(S_{3A_2})^\circ \cong S_3$. \qed
\end{enumerate}
\end{lem}

We now describe the action of the stabilizer of the $S_{3A_2}$ cubic on the Luna slice. The description of the Luna slice appears in \cite[p.52]{ZhangCubic}, but for clarity, particularly for the action of the stabilizer, we include a discussion here as well.

\begin{lem}[$3A_2$-Luna slice]\label{L:3A2slice}
A Luna slice for $S_{3A_2}$, normal to the orbit $\SL(4,\CC)\cdot [S_{3A_2}] \subseteq \PP^{19}$, is isomorphic to $\CC^6$, spanned by the $6$ monomials
$$
x_0^3,\ x_1^3,\ x_2^3,\ x_0^2x_3,\ x_1^2x_3, \ x_2^2x_3
$$
in the tangent space $H^0(\PP^3,\calO_{\PP^3}(3))$.

The Luna slice can be projectively completed to give a $\PP^6$:
$$
\PP^6 = \{\alpha_0x_0^3 +\alpha_1x_1^3 +\alpha_2x_2^3+\alpha_\hzero x_0^2x_3+\alpha_\hone x_1^2x_3 +\alpha_\htwo x_2^2x_3 +\alpha_{3A_2}(x_0x_1x_2+x_3^3) \} $$
$$
\subseteq \PP H^0(\PP^3,\calO_{\PP^3}(3))=\PP^{19}\,.
$$
The action of $\Stab(S_{3A_2})$ and $\GL(S_{3A2})$ on the projectively completed Luna slice is given by their inclusion into the groups $\SL(4,\CC)$ and $GL(4,\CC)$, respectively, with the given actions of those groups on $H^0(\PP^3,\calO_{\PP^3}(3))$ from the GIT setup.

The action of $\Stab(S_{3A_2})$ and $\GL(S_{3A2})$ on the Luna slice is the natural induced action. In terms of \Cref{L:R3A2Norm2}(3) and the description in~\eqref{E:GLStDs3s2-1}, the action of an element $\diag(\lambda_0,\lambda_1,\lambda_2,\lambda_3) $ in $D$ or $D'$ is given by
\begin{equation}\label{E:LunaAction}
\begin{aligned}
(\lambda_0,\lambda_1,\lambda_2,\lambda_3)&\cdot (\alpha_0,\alpha_1,\alpha_2,\alpha_\hzero ,\alpha_\hone ,\alpha_\htwo )=\\ &=\left(\left(\frac{\lambda_0}{\lambda_3}\right)^3\alpha_0,\left(\frac{\lambda_1}{\lambda_3}\right)^3\alpha_1,\left(\frac{\lambda_2}{\lambda_3}\right)^3\alpha_2,
\left(\frac{\lambda_0}{\lambda_3}\right)^2\alpha_\hzero , \left(\frac{\lambda_1}{\lambda_3}\right)^2\alpha_\hone , \left(\frac{\lambda_2}{\lambda_3}\right)^2\alpha_\htwo \right)\,,
\end{aligned}
\end{equation}
and the action of $\sigma\in S_3\subseteq \GL(S_{3A_2})$ is given by
\begin{equation}\label{E:3A2-S3-act}
\sigma\cdot (\alpha_0,\alpha_1,\alpha_2,\alpha_\hzero ,\alpha_\hone ,\alpha_\htwo )=
(\alpha_{\sigma(0)},\alpha_{\sigma(1)},\alpha_{\sigma(2)},\alpha_{\widehat {\sigma(0)}},\alpha_{\widehat{\sigma(1)}},\alpha_{\widehat{\sigma(2)}})\,.
\end{equation}
\end{lem}

\begin{proof}
For $S_{3A_2}=\lbrace F_{3A_2}=0\rbrace$, the matrix $DF_{3A_2}$, whose entries span the tangent space to the orbit of the $3A_2$ cubic, is given by (see \cite[\S 4.2.3]{cohcubics} for similar computations for the $3D_4$ cubic threefold)
$$
 DF_{3A_2}=\begin{pmatrix}
 x_0x_1x_2&x_1^2x_2&x_1x_2^2&x_1x_2x_3\\
 x_0^2x_2&x_0x_1x_2&x_0x_2^2&x_0x_2x_3\\
 x_0^2x_1&x_0x_1^2&x_0x_1x_2&x_0x_1x_3\\
 3x_0x_3^2&3x_1x_3^2&3x_2x_3^2&3x_3^3\\
 \end{pmatrix}\,.
$$
Since all entries of this matrix are monomial, the only possible linear relations are pairwise equalities, up to a constant factor. One sees that the only monomial that repeats more than once is $x_0x_1x_2$, and thus all linear relations satisfied by the entries of $DF_{3A_2}$ are
$$
 (DF_{3A_2})_{00}=(DF_{3A_2})_{11}=(DF_{3A_2})_{22}\,.
$$

This means that the normal space to the orbit ($\dim \PP^{19}-\dim \text{ orbit } = 19-(16-3)=6$) is spanned by the $6$ monomials
$$
x_0^3,\ x_1^3,\ x_2^3,\ x_0^2x_3,\ x_1^2x_3, \ x_2^2x_3\,.
$$

The given action of the stabilizer on the Luna slice can be seen for instance from taking the Luna slice to be the affine space
$$
\alpha_0x_0^3 +\alpha_1x_1^3 +\alpha_2x_2^3+\alpha_\hzero x_0^2x_3+\alpha_\hone x_1^2x_3 +\alpha_\htwo x_2^2x_3 +(x_0x_1x_2+x_3^3)\subseteq H^0(\PP^3,\calO_{\PP^3}(3))\,,
$$
then using the fact that $\lambda_0\lambda_1\lambda_2=\lambda_3^3$, and dividing through the natural action by $\lambda_3^3$, to fix the cubic form $x_0x_1x_2+x_3^3$ in the affine space.
\end{proof}
We now turn to the Eckardt divisor whose generic point parameterizes  smooth cubics with a non-trivial automorphism, and determine the multiplicity with which it contains~$S_{3A_2}$. The proof is by an elaborate lengthy explicit computation using the explicit form of the action on the Luna slice, and is given in~\Cref{S:Eck-no-3A2}.
\begin{lem}\label{L:Eck-no-3A2}
The Eckardt divisor $R\subseteq \PP H^0(\PP^3,\calO_{\PP^3}(3))=\PP^{19}$ contains the $3A_2$ orbit $\SL(4,\CC) \cdot [S_{3A_2}]$ with multiplicity $\mu=15$; i.e.,~if $\tilde \pi: \Bl_{\SL(4,\CC) \cdot [S_{3A_2}]}(\PP^{19})^{ss}\to (\PP^{19})^{ss}$ is the blowup of the $3A_2$ orbit, with exceptional divisor $D_{3A_2}$, then $\tilde \pi^*R= \widetilde R+15D_{3A_2}$, where $\widetilde R\subset\oBG$ is the strict transform of~$R$. \qed
\end{lem}

One can see from the above that for the action of $\Stab(S_{3A_2})$ on the normal space, the stabilizer of a general line will be trivial. In fact, when we blow up $(\PP^{19})^{ss}$ along the orbit $\SL(4,\CC) \cdot [S_{3A_2}]$ in the Kirwan blowup process, then in the Luna slice we are blowing up the origin in $\CC^6$, with exceptional divisor $\PP^5$. The Lemma above gives the action of the stabilizer on this $\PP^5$.

For future use, we first describe the semi-stable locus for this action of the stabilizer on $\PP^5$.
\begin{lem}\label{L:LunaStability}
Denoting by $(T_0:T_1:T_2:T_\hzero:T_\hone:T_\htwo)$ the homogeneous coordinates on the exceptional divisor~$\PP^5$ of the Kirwan blowup $\Bl_0\CC^6\to\CC^6$ of the Luna slice described above, the {\em unstable} locus of the action of $\GL(S_{3A_2})$ is the union of the three codimension two loci $\lbrace T_0=T_\hzero=0\rbrace$, $\lbrace T_1=T_\hone=0\rbrace$, $\lbrace T_2=T_\htwo=0\rbrace$ in $\PP^5$.
\end{lem}
\begin{proof}
The action on the Luna slice given by~\eqref{E:LunaAction} gives the following action of $\TT^2\simeq D''=\Aut(S_{3A_2})^\circ$ on the $\CC^6$ with coordinates $T$, of which the exceptional divisor is the projectivization:
\begin{equation}\label{E:Thomogeneous}
\begin{aligned}
 (T_0,T_1,T_2,T_\hzero,T_\hone,T_\htwo)&\mapsto
 (\lambda_0^3T_0,\lambda_1^3 T_1,\lambda_2^3 T_2, \lambda_0^2T_\hzero,\lambda_1^2T_\hone,\lambda_2^2T_\htwo)\\
 &=(\lambda_1^{-3}\lambda_2^{-3} T_0,\lambda_1^3 T_1,\lambda_2^3 T_2, \lambda_1^{-2}\lambda_2^{-2}T_\hzero,\lambda_1^2T_\hone,\lambda_2^2T_\htwo)
\end{aligned}
\end{equation}
(here we are acting by $\diag(\lambda_0,\lambda_1,\lambda_2,1)$ with $\lambda_0\lambda_1\lambda_2=1$, and thus expressing $\lambda_0=\lambda_1^{-1}\lambda_2^{-1}$). The action is by multiplying each coordinate by a monomial in $\lambda_1,\lambda_2$, and thus $\CC^6$ is decomposed into a direct sum of 6 one-dimensional torus representations. Plotting the weights of each monomial in~$\RR^2$, a point in $\CC^6$ is stable if and only if the convex hull of the set of weights corresponding to non-zero coordinates contains the origin in~$\RR^2$ (e.g., \cite[Lem.~3.10]{allcock} or \cite[Thm.~12.2]{D03_LOIT}). 
The weight diagram consists of 2 points on each of 3 rays from the origin. Thus a convex hull of some subset of these 6 weights contains the origin if and only if this subset contains at least one weight from each ray. This is to say, a point is stable if and only if at least one of its two coordinates $T_0$ and $T_\hzero$ is non-zero, etc. Thus the unstable points in $\CC^4\subseteq \CC^6$ are precisely those given by a pair of equations $T_i=T_{\widehat i}=0$ for some~$i$.
\end{proof}

We note that, as is the case for any Kirwan desingularization, there are no strictly semi-stable points on this exceptional divisor.
We now describe the finite stabilizers along the exceptional divisor. This is another detailed explicit computation using the stabilizer computed in~\Cref{L:R3A2Norm2}, and we give it in~\Cref{S:stab--ex}. We note that the proof actually allows us to determine all possible stabilizers, but we will not need this information.
\begin{pro}[Stabilizers along the exceptional divisor]\label{P:stab--ex}
Let $x\in D_{3A_2}\subseteq \MK$ be a point in the exceptional divisor, and let $S_x\subseteq \Stab(S_{3A_2})\subseteq \SL(4,\CC)$ be its stabilizer; i.e., the stabilizer of a point in the exceptional divisor of  $\Bl_{\SL(4,\CC) \cdot [S_{3A_2}]}(\PP^{19})^{ss}$ with orbit corresponding to $x$.
\begin{enumerate}
\item For $x\in D_{3A_2}$ general, $S_x=\mu_4$ (the diagonal subgroup of $\SL(4,\CC)$).
\item For {\em any} $x\in D_{3A_2}$, the order of $S_x$ is {\em not} divisible by $5$. \qed
\end{enumerate}
\end{pro}
\begin{rem}\label{rem:stabilizertrivial}
The proof of \Cref{P:discMK}, given in \Cref{S:discMK}, will provide another proof of the claim (1) of \Cref{P:stab--ex} above. There, we will even show that this assertion holds for a general point of the intersection of the strict transform $\widetilde D_{A_1}$ of the discriminant with the exceptional divisor $D_{3A_2}$. We note also that part (2) above is what will enable us to argue that the top self-intersection numbers of the canonical class on $\MK$ and $\oBG$ are different.
\end{rem}

We conclude the section with the following non-transversality result:
\begin{pro}\label{P:discMK}
At a generic point of the intersection $\widetilde D_{A_1}\cap D_{3A_2}\subseteq \MK$, these two divisors do not meet transversally.
\end{pro}
\begin{proof}[Outline of the proof]
The proof of this is by a detailed computation in local coordinates in charts of the blowup. To help the reader and the flow of the paper, we only summarize the key steps of the arguments here, postponing further details until \Cref{S:discMK}, where the proof will also benefit from building upon the explicit setup developed in the Appendix prior to that proof.

We first observe that the three $A_2$ singularities can be deformed independently. In the Luna slice the deformation space of each of the $A_2$ singularities is $\CC^2$, within which the discriminant divisor $D_{A_1}$ is a cuspidal curve. Thus altogether in the $\CC^6$ Luna slice near the $S_{3A_2}$ cubic surface, the discriminant divisor is the product of the three equations of cubics, in three disjoint pairs of coordinates,  one of which has the form $27\alpha_0^2+4\alpha_\hzero ^3=0$.

To determine the local structure of the Kirwan blowup, one considers the blowup $\Bl_0\CC^6$ of the origin in the Luna slice, and then studies the action of $\TT^2$ (the connected component of the stabilizer of $S_{3A_2}$ on this blowup). Identifying the explicit 4-dimensional Luna slice for this action, one writes down the equation of the discriminant divisor in this Luna slice explicitly, in charts on the projective space. In a suitable chart this discriminant divisor (that is, of $\widetilde D_{A_1}\subseteq \MK$) is locally a union of a number of hypersurfaces, one of which has the form $27t_0^2+4\alpha_\hzero$, where $\alpha_\hzero=0$ is the local equation of the exceptional divisor of the blowup, that is of $D_{3A_2}$. This intersection is manifestly non-transverse, except that  extra care is needed to take care of the finite part of the stabilizer. Indeed, in principle a quotient of a non-transverse intersection under a finite group may become transverse, and thus we need to ensure that the finite part of the stabilizer of $S_{3A_2}$ does not influence this (this is a local computation weaker than what is needed for the proof of \Cref{P:stab--ex}).
\end{proof}

\section{The ball quotient model and the first proof of \Cref{T:mainNonIso}}\label{sec:ballquoitient}
Allcock--Carlson--Toledo \cite{ACTsurf} have constructed a ball quotient model $\BG$ for the moduli of cubic surfaces. They proved that the Baily--Borel compactification $\BBG$ is isomorphic to the GIT model $\GIT$ and that under this identification the unique cusp of $\BBG$ corresponds to the GIT boundary point $\Delta_{3A_2}\in \GIT$. By the general theory, one has a toroidal compactification $\oBG$ of $\BG$, unique in this situation, which can be described as a blowup of $\BBG$ at the unique cusp. Similarly to the previous section, we study the intersection of the exceptional divisor $T_{3A_2}$ of the blowup $\oBG\to \BBG$ with the strict transform~$\widetilde D_n\subseteq \oBG$ of  the discriminant (Heegner) divisor $D_n\subseteq \BBG$. Here, in contrast with the Kirwan blowup where \Cref{P:discMK} gives non-transversality, we show in \Cref{P:discTor} that $T_{3A_2}$ and $\widetilde D_n$ meet generically transversally.
We thus obtain a first proof of \Cref{T:mainNonIso} that the isomorphism $ \GIT-\lbrace\Delta_{3A_2}\rbrace\cong \BG$ does not extend to an isomorphism of the compactifications $\MK$ and $\oBG$, despite both spaces being the blowup of the same point in $\GIT\cong\BBG$.

One key differentiating aspect of the ball quotient model (vs.~the GIT model) is the functorial behavior with respect to marking all the lines on the cubic surfaces (i.e., with respect to the natural $W(E_6)$ cover $\calM_m\to\calM$ ). This allows us to use Naruki's compactification $\overline{\calN}$ \cite{naruki}, which is a smooth normal crossing model for the marked moduli space, in order to understand the structure of $\oBG$, by applying the isomorphism $\overline{\calN}\cong \overline{\BG_m}$ previously established by \cite{GaKeSch}.
\subsection{Preliminaries on the ball quotient model}
We will now describe the compactifications of the ball quotient model of the moduli space of cubic surfaces. Before delving into the specifics for cubic surfaces, we first recall the compactification of ball quotients in general, referring to~\cite{AMRT} for the general details of the constructions. Let
$$
 \calB_n\coloneqq \{z\in \CC^{n}: \sum |z_i|^2 < 1\}
$$
be an $n$-dimensional ball. Alternatively, we can realize $\calB_n$ as follows. Let $\calO$ be the ring of integers of an imaginary quadratic field $\QQ(\sqrt{d})$, and let $\Lambda$ be a free $\calO$-module equipped with a hermitian form $h$ of signature $(1,n)$. Then
$$
\calB_n=\left\{ [z] \in \PP(\Lambda \otimes \CC): h(z) >0 \right\}\,.
$$

For an arithmetic subgroup $\Gamma\subseteq \SU(1,n)$, there is a quotient quasi-projective analytic space $\calB_n/\Gamma$, which by construction has at worst finite quotient singularities. This quotient admits a projective Baily--Borel compactification $(\calB_n/\Gamma)^*$ defined as the $\Proj$ of the ring of automorphic forms with respect to $\Gamma$. Geometrically, the boundary $(\calB_n/\Gamma)^*-\calB_n/\Gamma=c_{F_1}\sqcup \dots \sqcup c_{F_r}$ consists of a finite number of points, called cusps. These are in $1$-to-$1$ correspondence with the $\Gamma$-orbits of isotropic lines in $\Lambda_{\QQ(\sqrt{d})}$. There are no higher dimensional cusps since the signature is $(1,n)$, implying that no other isotropic
subspaces exist.

In the case of ball quotients there exists a unique toroidal compactification $\overline{\calB_n/\Gamma}$. Uniqueness follows since all tori involved have rank $1$. More precisely, after dividing by the unipotent radical of the parabolic subgroup that stabilizes a given cusp, the quotient locally looks like an open set in $D^* \times \CC^{n-1} \subseteq \CC^* \times \CC^{n-1}$ that contains $\{0\}\times \CC^{n-1}$ in its closure, where here $D^*$ is the punctured unit disk.
The toroidal compactification is then simply obtained by adding the divisor $\{0\}\times \CC^{n-1}$. As a result, the boundary $\overline {\calB_n/\Gamma}-\calB_n/\Gamma=T_{F_1}\sqcup \dots \sqcup T_{F_r}$ consists of a finite disjoint union of smooth (up to finite quotient singularities) irreducible divisors, each of which is in fact a finite quotient of an abelian variety. The natural map $p:\overline{\calB_n/\Gamma}\to(\calB_n/\Gamma)^*$ simply contracts each divisor~$T_{F_i}$ to the cusp~$c_{F_i}$ (see
 \Cref{P:MarkedTor} below for a detailed discussion of the case relevant in this paper).

\medskip
With this setup, we return to the case of cubic surfaces, and describe the period map to the ball quotient. Specifically, considering the triple cover of $\PP^3$ branched along a cubic surface, one obtains a cubic threefold, and via the period map for cubic threefolds, taking the $\ZZ_3$ action into account, one obtains a period map to a $4$-dimensional ball quotient, $\calM\to \BG$ (see~\cite{ACTsurf}). This is an open embedding, and the complement of the image is the Heegner divisor $D_n=\calD_n/\Gamma\subseteq \BG$. It turns out that in this case the rational period map $\GIT\dashrightarrow \BBG$ to the Baily--Borel compactification extends to an isomorphism, taking the discriminant divisor $D_{A_1}\subseteq \GIT$ to the (closure of the) Heegner divisor $D_n\subseteq \BBG$ (which is denoted this way for ``nodal"). Under this isomorphism, the unique strictly polystable point $\Delta_{3A_2}\in \GIT$ corresponding to the $3A_2$ cubic is identified with the sole cusp $c_{3A_2}=\partial\BBG$. The natural map $p:\oBG\to\BBG$ contracts the irreducible boundary divisor $T_{3A_2}$ to $c_{3A_2}$. From now on, we will write $D_{A_1}=D_n$ for the discriminant divisor, where we use $D_{A_1}$ when we are thinking of it from the GIT point of view, and $D_n$ when thinking of the ball quotient --- to emphasize the context we are in.

In summary, for the case of cubic surfaces we have a diagram
\begin{equation}\label{diag_unmarked}
\xymatrix{
\MK \ar_{\pi}[d]\ar@{-->}^f[r]&\oBG \ar^p[d]\\ \GIT \ar[r]^{\sim}&\BBG
}
\end{equation}
where $f$ is a birational map that restricts to an isomorphism
\begin{equation}\label{E:BoundCIso}
f:\MK -D_{3A_2} \cong \oBG - T_{3A_2}\,.
\end{equation}
In~\cite{kirwanhyp} and~\cite{ZhangCubic} the (intersection) Betti numbers of the spaces $\GIT\cong \BBG$ and $\MK$ were computed. In \cite[\S C.2]{cohcubics}, the Betti numbers of $\oBG$ were computed, and they turned out to be the same as for $\MK$, which served as motivation for our query as to whether these two compactifications are isomorphic, which is the main subject of the current paper.

\subsection{The toroidal compactification via marked cubic surfaces}\label{sec:marked}
An indispensable tool in the study of cubic surfaces is the group $W(E_6)$, which is the automorphism group for the configuration of the $27$ lines on a smooth cubic. The moduli space of cubic surfaces $\calM$ admits a natural~$W(E_6)$ cover~$\calM_m$ parameterizing marked smooth cubic surfaces, i.e., cubics together with a labeling of the $27$ lines.
Since the automorphism group of a smooth cubic surface acts faithfully on the primitive cohomology, it follows that $\calM_m$ is in fact smooth. Furthermore, Naruki \cite{naruki} constructed a smooth normal crossing compactification $\overline{\calN}$ of $\calM_m$, which admits various geometric interpretations (see e.g.,  \cite{HKT09}).  In this section, we use the geometry of the Naruki model $\overline{\calN}$ to get a good hold on the space of interest in our paper,  $\oBG$.

\medskip

By construction, the marked moduli space $\calM_m$ is a Galois cover of $\calM$ with Galois group $W(E_6)$. This cover is compatible with the ball quotient construction of Allcock--Carlson--Toledo \cite{ACTsurf}. Specifically, the monodromy group $\Gamma$ for cubic surfaces contains a normal subgroup $\Gamma_m\unlhd \Gamma$ with
\begin{equation}\label{eqgamma}
\Gamma/\Gamma_m\cong W(E_6) \times \{ \pm 1\}.
\end{equation}
(see \cite[(3.12)]{ACTsurf}).
We observe that $-1 $ acts trivially on the ball $\calB_4$. This leads to a $W(E_6)$ cover $\BG_m\to \BG$. Furthermore, this cover extends to the Baily--Borel compactifications (compare \cite[Thm. 3.17]{ACTsurf}), and then also to the toroidal compactifications --- essentially because in the ball quotient case, the toroidal compactification is canonical, and thus there is an automatic extension.  In summary, the following holds:
\begin{pro}\label{Prop-compatible}
With notation as above, we have the following diagram:
\smallskip
\begin{equation}\label{we6diag}
\xymatrix{
\calM_m \ar@{->>}[d]\ar@{^{(}->}[r]&\BG_m \ar@{->>}[d]\ar@{^{(}->}[r]\ar@/^1.3pc/@{^{(}->}[rr]&(\BG_m)^*\ar@{->>}[d]&\overline{\BG_m}\ar[l]\ar@{->>}[d]\\ \calM\ar@{^{(}->}[r]&\BG\ar@{^{(}->}[r]\ar@{^{(}->}[r]\ar@/^1.3pc/@{^{(}->}[rr]&(\BG)^*&\oBG\ar[l]
}
\end{equation}
where all the spaces in the top row admit a $W(E_6)$ action, and the morphisms are $W(E_6)$-equivariant, while the spaces in the bottom row are the quotients with respect to this $W(E_6)$ action. \qed
\end{pro}

\begin{rem}
The GIT construction does not admit a natural $W(E_6)$ cover (e.g., it involves taking a quotient by $\mathrm{PGL}(4,\CC)$, which has no natural connection to $W(E_6)$).
\end{rem}

\begin{rem}\label{rem:cubics with a line}
We note that there is another very natural moduli space $\calM_\ell$ parameterizing smooth cubics together with a chosen line. It is a degree 27 non-Galois cover $\calM_\ell\to\calM$, which is in turn covered via $\calM_m\to\calM_\ell$. This moduli of cubics with a line is of particular interest as it has a model as a Deligne--Mostow moduli space $DM(2^5,1^2)$ of points on a line. As such, $\calM_\ell$ has both a GIT compactification, and the corresponding Kirwan blowup, as of a configuration of points, and a toroidal compactification covering $\oBG$. It seems likely that our methods from this and previous works would make it possible to determine whether the corresponding Kirwan blowup and toroidal compactification are naturally isomorphic, but we will not pursue it here.
\end{rem}
It was shown recently  that the marked toroidal and Naruki compactifications coincide.
\begin{teo}[{\cite{GaKeSch}}]\label{ThmGallardo}
The Naruki compactification $\overline{\calN}$ is isomorphic to the toroidal compactification $\oBGm$.  More precisely, there is a $W(E_6)$-equivariant commutative diagram
$$\xymatrix{\calM_m \ar@{^{(}->}[r]\ar@{^{(}->}[d]&\BG_m\ar@{^{(}->}[d]\\
{\overline\calN}\ar[r]^{\sim}&\oBGm}
$$
\end{teo}

\begin{notation}
In what follows, we use  freely the identification given by \Cref{ThmGallardo}. Motivated by the
compatibility given by diagram \eqref{we6diag}, we will use for $\BG_m$ and its compactifications the same notation as for $\BG$, simply adding the subscript~$m$. In particular, we will consider the divisors $\widetilde D_{n,m}$ and $T_{3A_2,m}$ on $\oBGm$. We will informally refer to $\widetilde D_n$ (and $\widetilde D_{n,m}$) as the nodal divisor, to $\widetilde R$ (and $\widetilde R_m$ on $\oBGm$) as the Eckardt divisor, and to $T_{3A_2}$ (and $T_{3A_2,m}$) as the (toroidal) boundary divisor. Note that while $\widetilde D_n$ and $\widetilde R$ are irreducible divisors in $\oBG$, the corresponding divisors in the marked case have several irreducible components transitively permuted by the natural $W(E_6)$ action.
\end{notation}

The Naruki compactification has a well-understood structure, mostly due to Naruki \cite{naruki}, with some further later clarifications by other authors.

\begin{teo}[Naruki]\label{P:MarkedTor} The spaces and maps above have the following descriptions:
\begin{enumerate}
\item[(0)] $\oBGm(\cong \overline{\calN})$ is smooth, and the complement of the locus $\calM_m$ of smooth marked cubic surfaces is the simple normal crossing divisor $T_{3A_2,m}\cup \widetilde D_{n,m}$.

\item
\begin{itemize}
\item[a)] The Baily--Borel compactification $\BBGm$  has $40$ cusps, permuted transitively by the $W(E_6)$ action, each of which lies over the unique cusp $c_{3A_2}\in\BBG$.  Near each of the $40$ cusps, $\BBGm$  is locally isomorphic to the cone over $(\PP^1)^{\times3}$ embedded in $\PP^7$ via $\calO(1,1,1)$.
\item[b)] The boundary divisor $T_{3A_2,m}\subseteq\oBGm$ has $40$ disjoint irreducible components, each isomorphic to $(\PP^1)^{\times 3}$. The deck transformation group $W(E_6)$ of the cover $\oBGm\to\oBG$ acts transitively on the set of these irreducible components.
\end{itemize}

\item
\begin{itemize}
\item[a)] The $W(E_6)$ cover $\BBGm\to \BBG$ is branched along the discriminant divisor~$D_{A_1}$ and the Eckardt divisor~$R$, with ramification index $2$ along each.
\item[b)] The $W(E_6)$cover $\oBGm\to\oBG$ is generically \'etale along the toroidal boundary divisors $T_{3A_2,m}$.
\end{itemize}

\item The stabilizer $S_{3A_2,m}\subseteq W(E_6)$ of an irreducible component of the divisor $T_{3A_2,m}$ fits in an extension
$$1\to (S_3)^{\times 3}\to S_{3A_2,m}\to S_3\to 1\,.$$
Under the identification of the irreducible component with $(\PP^1)^{\times 3}$, the stabilizer $S_{3A_2,m}$ acts as follows: the normal subgroup $(S_3)^{\times 3}$ acts diagonally on $(\PP^1)^{\times 3}$, while the residual $S_3$ acts on the quotient $(\PP^1)^{\times 3}/(S_3)^{\times 3}\cong (\PP(2,3))^{\times 3}\cong (\PP^1)^{\times 3}$ by permuting the factors.
\end{enumerate}
\end{teo}

\begin{proof}
Item (0) is the main result of Naruki \cite[Thm.~1]{naruki}. Naruki also proved
that the toroidal boundary $T_{3A_2,m}$ consists of $40$ irreducible components (permuted transitively by $W(E_6)$), each isomorphic to $(\PP^1)^{\times 3}$. The stabilizer of a boundary component and its action are discussed in \cite[p.22]{naruki}. In particular, items (1b), (2b), and (3) follow. As an aside, we note that each of the $40$ components corresponds to a choice of embedding of the $A_2\times A_2\times A_2$ lattice into the $E_6$ lattice, and that intrinsically $S_{3A_2,m}$ is the normalizer $N_{W(E_6)}(3A_2)$ of such an embedding (recall that $N_{W(E_6)}(3A_2)$ is the unique, up to conjugacy, index $40$ subgroup of $W(E_6)$; compare \cite[Table 9]{Carter}).

The ramification statement (2a) is clear for geometric reasons: the branch divisor consists of the nodal locus (the Picard--Lefschetz transformations act as reflections in $W(E_6)$), and the locus of smooth cubics with extra automorphisms, which coincides with the Eckardt divisor. While the occurrence of nodal degenerations is quite general, the presence of the Eckardt component is special to cubic surfaces: it is rare for the locus of objects with extra automorphisms to form a divisor in the moduli space, and secondly it reflects the fact that the automorphism group of a cubic surface  embeds into $W(E_6)$. The ramification statement
is also worked out in detail (from the ball quotient perspective) in \cite[Thms.~2.14 and 7.26] {ACTsurf} for the nodal locus, and \cite[Sec.~11 and Lem.~11.4]{ACTsurf}
for the Eckardt locus. For the interested reader, we point out that these two types of ramification are associated, respectively, to short and long roots in the Eisenstein lattice used to construct the ball quotient model of \cite{ACTsurf}.

Finally, it remains to discuss the structure at the boundary of the  Baily--Borel compactification and the relation to the toroidal compactification. First, the Baily--Borel compactification is a contraction of the toroidal boundary $T_{3A_2,m}$ in $\oBGm(\cong \overline{\calN})$. Thus, (1b) implies that $\BBGm$ has $40$ cusps (see also \cite[\S8.16]{DvGK}, where this fact is proved without reference to the toroidal compactification).
Since $(\BG)^*$ has a unique cusp, and the two Baily--Borel compactifications are compatible as in \eqref{we6diag}, the first part of (1a) follows.   It remains to determine the local structure near the cusps of $\BBGm$.
Naruki \cite[\S11]{naruki} constructed a contraction $ \overline{\calN}\to \calN^*$ of the $40$ irreducible components of $T_{3A_2,m}$ to $40$ singularities of the type described in (1a). It was then noted in \cite[\S2.9]{DvGK} that  $\calN^*$ coincides with $\BBGm$; this completes the proof.
\end{proof}

\begin{rem}\label{rem-gks}
The paper \cite{GaKeSch} proves a stronger statement than \Cref{ThmGallardo}, namely that both $\overline{\mathcal N}$ and $\oBGm$ are 
isomorphic to an appropriate Hassett moduli space of weighted stable rational curves. 
However, we note that in fact, Naruki's results \cite{naruki} regarding the Naruki compactification $\overline {\mathcal N}$ as explained in \Cref{P:MarkedTor} are enough to provide a short proof of the result of \cite{GaKeSch} stated in  \Cref{ThmGallardo}. Indeed, as explained in the proof of \Cref{P:MarkedTor}(0) and (1)b), the union of the discriminant and Naruki boundary in $\overline{\mathcal N}$ is a simple normal crossing divisor, and so the rational map $\overline{\mathcal N}\dashrightarrow \oBGm$ extends to a morphism \cite{AMRT}.  
From \cite[Thm.~3.17]{ACTsurf} the morphism is an  isomorphism over the generic points of the discriminant divisor in $\overline {\mathcal N}$, mapping the discriminant in $\overline{\mathcal N}$  to the discriminant in $\oBGm$.  
By \cite[Prop.~2.1]{AllFrei} we know the Satake compactification  $\BBGm$ has $40$ cusps and thus the toroidal boundary consists of $40$ disjoint irreducible components. 
Since Naruki has shown that the Naruki boundary also consists of $40$ disjoint irreducible components (see \Cref{P:MarkedTor}(1)b)), the morphism $\overline{\mathcal N}\to \oBGm$ cannot be a divisorial contraction.  At the same time, since the target is $\mathbb Q$-factorial, the morphism cannot be a small contraction.  Hence it is an isomorphism.
Note that the reason this type of argument  fails in the unmarked case, i.e., the reason this type of argument does not imply that ${\MK}$ and $\oBG$ are isomorphic,  is that we can not apply the Borel extension theorem since ${\MK}$ does not have a  normal crossing boundary (cf.~\Cref{P:discMK}).
\end{rem}

With these preliminaries, we can extract a series of immediate consequences on the boundary of $\calM\subseteq \oBG$. First, \Cref{Prop-compatible} and \Cref{P:MarkedTor}(0) together show that this is a normal crossing compactification in a stack sense:
\begin{cor}\label{C:oBGNCb}
The boundary $\widetilde D_n\cup T_{3A_2}$ in $\oBG$ is a normal crossings divisor, up to finite quotients. \qed
\end{cor}

For further reference, we also record the structure of the toroidal boundary divisor:
\begin{cor}\label{C:T3A2}
The toroidal boundary divisor $T_{3A_2}\subseteq \oBG$ is isomorphic to $\PP^3$.
\end{cor}
\begin{proof}
The boundary divisor $T_{3A_2}$ is the quotient of a fixed component of $T_{3A_2,m}$ by the relevant stabilizer group. Using \Cref{P:MarkedTor}, we get $T_{3A_2}\cong (\PP^1)^{\times 3}/S_{3A_2,m}\cong \Sym^3\PP^1\cong\PP^3$.
\end{proof}

\subsection{Proof of \Cref{T:mainNonIso}}
At this point, we are able to establish one of our main results, namely that the period map does not extend to an isomorphism between the Kirwan compactification and the toroidal compactification of the ball quotient model. We have seen by \Cref{P:discMK} that the nodal and boundary divisors do not meet transversally, even generically, in the Kirwan model, while an immediate consequence of the above discussion is that they do so in the toroidal model:

\begin{pro}\label{P:discTor}
The discriminant $\widetilde D_n$ and boundary $T_{3A_2}$ divisors in $\oBG$ meet generically transversally along an irreducible surface.
\end{pro}
\begin{proof}

This follows easily from the geometry explicitly described in \Cref{P:MarkedTor} for the marked case.
We first observe that the intersection of $\widetilde D_n$ and $T_{3A_2}$ is irreducible and hence it will be enough to consider a
generic point $P$ of some component of the intersection $ \widetilde D_{n,m}\cap T_{3A_2,m}$.
To see this we  recall that, by \Cref{P:MarkedTor} (1)(b), the Weyl group $W(E_6)$ acts transitively on the cups and hence the toroidal boundary components. We fix a component $T$ of $T_{3A_2,m}$. We claim that the stabilizer $S_T(\cong S_{3A_2,m}$ see  \Cref{P:MarkedTor} (3))
of $T$ in $W(E_6)$ permutes all components of $T \cap \widetilde D_{n,m}$.
For this we recall that $T\cong (\PP^1)^{\times 3}$ and use the description of the
action of the stabilizer $S_T$ given in  \Cref{P:MarkedTor} (3).
This stabilizer contains a normal subgroup isomorphic to $(S_3)^{\times 3}$ acting independently on each $\PP^1$ factor. There is also a residual $S_3$ factor, which acts by permuting the $3$ copies of $\PP^1$ in $T$.
The three copies of $(S_3)^{\times 3}$ each act on one of the copies of $(\PP^1)^{\times 3}$ by
the projectivization of the standard action of $S_3$ on $\CC^2$, or more intrinsically here, the projectivization of the action of $W(A_2)$ on $A_2\otimes_{\ZZ}\CC$, and trivially on the other two factors.  Choosing
suitable coordinates we can assume that $S_3$ permutes the points $0,1,\infty$. By \cite[Prop. 11.2']{naruki} the intersection $T \cap \widetilde D_{n,m}$ consists of $9$ components,
namely  the surfaces $\{0\} \times (\PP^1)^{\times 2}$, $\{1\} \times (\PP^1)^{\times 2}$ and $\{\infty\} \times (\PP^1)^{\times 2}$ and their translates under the group $S_3$ interchanging the three factors of $(\PP^1)^{\times 3}$.
Clearly these components are permuted under $S_{T}$.

We can now work with the component $\{0\} \times (\PP^1)^{\times 2}$ and recall that in $\oBGm$, the divisor $\widetilde D_{n,m}\cup T_{3A_2,m}$ has simple normal crossings. The stabilizer of a point
$P=(0,*_1,*_2)$ with $*_i \neq 0,1,\infty$ and $*_1 \neq *_2$ has order $2$;  its non-trivial element is the involution in the first factor of $(S_3)^{\times 3}$ which fixes $0$ and interchanges $1$ and $\infty$.
This defines a reflection in $W(E_6)$ whose fixed locus is the component of $ \widetilde D_{n,m}$ passing through $P$.
This involution also fixes the component $T$ of $T_{3A_2,m}$ (as a set, not pointwise). Thus we can choose local analytic coordinates $(x_1,x_2,x_3,x_4)$ on $\oBGm$ near~$P$ such that $\widetilde D_{n,m}$ and $T_{3A_2,m}$ are the zero loci of coordinates $x_1$ and $x_2$, respectively, and such that the stabilizer acts by $x_1 \mapsto -x_1$, leaving all other coordinates fixed. On the quotient we can therefore take local analytic coordinates $y_1=x_1^2, y_i=x_i, i= 2,3,4$. Then locally $\widetilde D_n$ and $T_{3A_2}$
are given by $y_1=0$ and $y_2=0$ respectively, and the claim follows.
\end{proof}

\begin{rem}\label{R:irreducibleintersection}
In \Cref{rem:intersectionwithboundarylocal} we shall provide a different proof for the fact that $W(E_6)$ acts transitively on the components of $\widetilde D_{n,m}\cup T_{3A_2,m}$.
\end{rem}

\begin{proof}[{Proof of \Cref{T:mainNonIso}}]
This follows immediately from \Cref{P:discMK} and \Cref{P:discTor}. Indeed, these propositions show {\em a priori} that the period map does not extend to an {\em isomorphism}. However, this is enough to show that the period map does not extend to a morphism in either direction. Indeed,
since $f$ gives an isomorphism $\MK -D_{3A_2} \cong \oBG - T_{3A_2}$, if $f$ extended to a morphism, it would have to send the irreducible divisor $D_{3A_2}$ to the irreducible divisor $T_{3A_2}$.
Thus, if $f$ were to extend to a morphism, which was not an isomorphism, it would have to be a small contraction. For a small contraction of complex varieties $f:Y\to X$, if $Y$ is quasi-projective and $X$ is normal, then $X$ is not $\QQ$-factorial, see eg.~\cite[Cor.~2.63]{kollarmori}. 
This would contradict the fact that $\oBG$ is $\QQ$-factorial, having only finite quotient singularities. A similar argument holds for the rational map $f^{-1}$, since $\MK$ is also $\QQ$-factorial.
\end{proof}

As a corollary of \Cref{T:mainNonIso}, we have the following result, which  contrasts with \Cref{C:oBGNCb} for the toroidal compactification, and strengthens \Cref{P:discMK} (which was itself used in the proof of \Cref{T:mainNonIso}) for the Kirwan compactification.

\begin{cor}\label{cor:quickproofnontransversal}
The boundary $\widetilde D_{A_1}\cup D_{3A_2}$ in $\MK$ is not a normal crossings divisor, even up to finite quotients.
\end{cor}

\begin{proof}
If the boundary $\widetilde D_{A_1}\cup D_{3A_2}$ in $\MK$ were a normal crossings divisor, up to finite quotients, then the standard extension theorems for period maps to toroidal compactifications \cite{AMRT} would imply that the period map $f:\MK \dashrightarrow \oBG$ extended to a morphism, contradicting \Cref{T:mainNonIso}.
\end{proof}

\begin{rem}\label{R:log-depend}
We emphasize that \Cref{cor:quickproofnontransversal} depends on \Cref{T:mainNonIso}, which in turn depends on \Cref{P:discMK}, so that the proof of \Cref{cor:quickproofnontransversal} above does  not provide an alternate proof of \Cref{P:discMK}.  On the other hand,  \Cref{T:mainNonIso} follows from \Cref{T:mainNonK}, whose proof is independent of \Cref{P:discMK}.  In other words, one can give an alternate proof of \Cref{P:discMK} by first proving \Cref{T:mainNonK}, which implies  \Cref{T:mainNonIso}, and then proving \Cref{cor:quickproofnontransversal}.
\end{rem}

\section{The canonical bundles of the ball quotient models}\label{sec:canonicalbundle}
While \Cref{T:mainNonIso}
says that the period map does not extend to an isomorphism between $\MK$ and $\oBG$, {\em a priori} it is possible that $\MK$ and $\oBG$ might be abstractly isomorphic. To distinguish them, we study their canonical classes. We start with a fairly complete discussion of divisor classes (in particular the canonical class) on the ball quotient side. Similar computations occur in \cite[\S 7]{CML} for the ball quotient model for cubic threefolds. However, due to the simpler structure of the moduli of cubic surfaces, we are able to obtain sharper results here.

\subsection{Divisors and relations in $\BBG$}
The ball quotients come equipped with a natural $\QQ$-line bundle $\lambda$, the so-called Hodge line bundle,  that gives the polarization for the Baily--Borel compactification (and pulls back to a big and nef bundle on the toroidal compactification). Since $\BBG\cong \GIT\cong \PP(1,2,3,4,5)$ has Picard rank $1$, the divisors $\lambda$, $ D_n,$ and~$R$ must all be proportional. It turns out that one can express the classes of these divisors in terms of $\lambda$ as  a consequence of the work of Borcherds. Specifically, the following holds:

\begin{teo}[{\cite[Thm. 4.7]{AllFrei}}] \label{THM-AF}
There is an automorphic form $\chi_4$  of weight $4$ whose divisor in $\calB_4$ is the sum of all short mirrors, each with multiplicity $1$. Similarly, there is an automorphic form $\chi_{75}$  of weight $75$ whose divisor in $\calB_4$ is the sum of all long mirrors, each with multiplicity $1$.
\end{teo}

Taking into account the ramification orders $6$ for the discriminant divisor, which corresponds to the sum of all short mirrors, and $2$ for the Eckardt divisor, which corresponds to the sum of all long mirrors (see esp. \cite[\S11]{ACTsurf}), we obtain the following equalities in $\Pic_\QQ(\BBG)$:
\begin{equation}\label{eq:relationslambda}
4\lambda=\frac{1}{6}  D_n;\qquad
75\lambda=\frac{1}{2} R\,.
\end{equation}
Similarly, in the marked case, where the map $\calB_4\to \BBGm$ is only ramified along the nodal locus with index $3$ (e.g., \cite[Thm. 7.26]{ACTsurf}),
the following holds in $\Pic_\QQ((\BG_m)^*)$:
\begin{equation*}
4\lambda_m=\frac{1}{3}  D_{n,m};\qquad
75\lambda_m= R_m\,.
\end{equation*}

Under the identification $\BBG\cong \PP(1,2,3,4,5)$, we have seen \eqref{equation:invarianttheory} that
$D_n=\calO_{\PP(1,2,3,4,5)}(4)$ and $R=\calO_{\PP(1,2,3,4,5)}(25)$
(which gives $R=\frac{25}{4} D_n$, agreeing with \eqref{eq:relationslambda}), either of which combined with \eqref{eq:relationslambda}
give the following relation between the natural polarizations on $\BBG$ and $\PP(1,2,3,4,5)$:
\begin{equation}\label{eq:Olambda}
\calO_{\PP(1,2,3,4,5)}(1)=6\lambda\,.
\end{equation}

We now turn to the canonical divisor. By the general theory of ball quotients, e.g., \cite[Prop.~3.4]{MumHirz77}, this is given by a multiple of the Hodge line bundle $\lambda$, adjusted by a contribution
due to the ramification, see e.g.,  \cite[Thm.~3.4]{Alexeev}. The relevant multiple of $\lambda$ is ``$\dim+1$''. We also note that $\lambda$ is given by the canonical automorphy factor given by the Jacobian,
and that modular forms of weight $k$ are exactly the sections of $\lambda^{\otimes k}$. The ramification divisor in our case is the union of the discriminant divisor $D_n$ and the Eckardt divisor $R$, which have branch orders $6$ and $2$, respectively. Thus we obtain
\begin{equation}\label{eq-KBBG}
K_{\BBG}=5\lambda - \frac{5}{6}  D_n - \frac{1}{2}  R =-90\lambda\,.
\end{equation}
By \eqref{eq:Olambda}, we see that $K_{\BBG}=\calO_{\PP(1,2,3,4,5)}(-15)$, which agrees with the canonical bundle of $\PP(1,2,3,4,5)$ (e.g., \eqref{equation:invarianttheory}). In the marked case, as the map $\calB_4\to \BBGm$ is only ramified along the nodal locus with index $3$, we thus obtain
\begin{equation}\label{eq:can-mark}
K_{\BBGm}=5\lambda - \frac{2}{3}  D_{n,m}\,.
\end{equation}

The same general considerations give the formulas for the canonical bundles of the toroidal compactifications.
\begin{pro}\label{pro:canonicalbundleontoroidals}
The following hold:
\begin{enumerate}
\item $K_{\oBG}=5\lambda - \frac{5}{6}  \widetilde D_n - \frac{1}{2}  \widetilde R - T_{3A_2}$\,,
\item $K_{\oBGm}=5\lambda_m - \frac{2}{3}\widetilde D_{n,m} -T_{3A_2,m}$\,.
\end{enumerate}
\end{pro}
\begin{proof}
This is a standard computation, a consequence of Mumford's Hirzebruch proportionality theorem (see e.g., \cite[Thm.~3.4]{Alexeev}).
\end{proof}

\subsection{Computations of discrepancies}
In order to enable us to compute the top self-intersection of the canonical class, in this section we compute how divisors on the Baily--Borel compactifications compare to divisors on the toroidal compactifications.

\begin{cor}\label{cor:discr}
In the notation above, the canonical class is given by
\begin{equation}\label{E:KTorMkd}
K_{\oBGm}=p_m^*K_{\BBGm}+T_{3A_2,m}\,.
\end{equation}
\end{cor}
\begin{proof}
By \Cref{P:MarkedTor},
locally near a cusp of $\BBGm$, the map $\oBGm\to \BBGm$ is the standard  blowup of the cone over $(\PP^1)^{\times3}\hookrightarrow \PP^7$. The claim then follows by a standard computation for the blowup.
\end{proof}

To descend to the unmarked case, we need to understand the intersection of the boundary divisor $T_{3A_2}$ with the two ramification divisors $\widetilde D_{n,m}$ and $\widetilde R_{m}$.
\begin{pro} \label{P:MarkDivRestrict}

\begin{enumerate}
\item The normal bundle to the marked toroidal boundary divisor $T_{3A_2,m}$, restricted to each irreducible boundary component, is equal to $\calO_{(\PP^1)^{\times3}}(-1,-1,-1)$.

\item   The restriction $\widetilde D_{n,m}|_{T_{3A_2,m}}$ of the marked discriminant divisor to each irreducible component of  the marked toroidal boundary divisor is isomorphic to  $\calO_{(\PP^1)^{\times3}}(3,3,3)$.

\item The restriction $\widetilde R_{m}|_{T_{3A_2,m}}$ of the marked Eckardt  divisor to each irreducible component of  the marked toroidal boundary divisor  is isomorphic to $\calO_{(\PP^1)^{\times3}}(12,12,12)$.
\end{enumerate}
\end{pro}
\begin{proof}
The same argument as in \Cref{cor:discr} gives the first item. As discussed in \Cref{P:MarkedTor}, the divisors $\widetilde D_{n,m}$ and ${T_{3A_2,m}}$  intersect transversely, and  \cite[Thm. 11.2']{naruki} describes this intersection precisely. For a fixed irreducible component $T_0$ of $\widetilde T_{3A_2}$, which we have seen is isomorphic to $(\PP^1)^{\times 3}$, the intersection of $\widetilde D_{n,m}$ with it consists of  $9$ irreducible components of type $\{\textrm{pt}\}\times \PP^1\times \PP^1\subseteq \PP^1\times\PP^1\times \PP^1$, and thus of class $\calO_{(\PP^1)^{\times3}}(1,0,0)$ (up to permuting the coordinates). The claim (2) thus follows.

The stabilizer $S_{3A_2,m}\subseteq W(E_6)$ of $T_0$ was identified in \Cref{P:MarkedTor}. Taking $U_m$ to be a suitable invariant neighborhood of $T_0$, locally near $T_0$, the map $\oBGm\to \oBG$ is simply the quotient $U_m\to U_m/S_{3A_2,m}\cong U$, with $U$  a neighborhood of the toroidal boundary $T_{3A_2}\subseteq \oBG$. Since $S_{3A_2,m}$ contains a normal subgroup $(S_3)^{\times 3}$, we can take the intermediate quotient $U'=U_m/(S_3)^{\times 3}$ and obtain a diagram
$$\xymatrix{
U_m \ar@{->}[r]^{/(S_3)^{\times 3}}_{\alpha}&U'\ar@{->}[r]^{/S_3}_{\beta}&U\\
(\PP^1)^{\times3}\ar[r]\ar@{^{(}->}[u]&(\PP^1)^{\times3}\ar@{^{(}->}[u]\ar[r]&\ \ T_{3A_2}\cong \PP^3\ar@{^{(}->}[u]
}$$
compatible with \Cref{P:MarkedTor}(3) and \Cref{C:T3A2}. Let $T'\subseteq U'$ be the quotient $T_0/((S_3)^{\times 3})\cong (\PP(2,3))^{\times 3}\cong (\PP^1)^{\times 3}$.

To describe $\widetilde R_{m}\cap T_0$, first recall that $U_m\to U$ is ramified along the Eckardt and nodal loci, with order $2$ along each. The factorization $U_m\xrightarrow{\alpha} U'\xrightarrow{\beta} U$ has the property that $\alpha$ is ramified along the nodal locus, while $\beta$ is ramified along the Eckardt locus. Indeed, in the language of \cite{ACTsurf}, the subgroup $(S_3)^{\times 3}\subseteq S_ {3A_2,m}$ is generated by short roots (corresponding to the $9$ nodal components meeting $T_0$), and the residual $S_3=S_ {3A_2,m}/(S_3)^{\times 3}$ is generated by classes of long roots (corresponding to the Eckardt locus, and geometrically to cubics with an extra involution). In this description, it is clear that the restriction of the Eckardt locus to $T'$ is simply the sum of the $3$ small diagonals (each of type $\calO_{(\PP^1)^{\times 3}}(1,1,0)$, up to permutation), and thus of class $\calO_{(\PP^1)^{\times 3}}(2,2,2)$. The pullback via $\alpha_{\mid T_0}$ of this class will be of type $\calO_{(\PP^1)^{\times 3}}(12,12,12)$. Since $\alpha$ is not ramified along the Eckardt locus, this will be also the class of the reduced divisor $\widetilde{R}_{m}|_{T_0}$.
\end{proof}
\begin{rem}\label{rem:intersectionwithboundarylocal}
The last two claims in \Cref{P:MarkDivRestrict} can be obtained also by a purely arithmetic argument. As alluded to in the proofs above, they follow by counting the short and long roots incident to a fixed cusp in $\BBGm$. To explain this, we consider the vector space $\FF_3^5$, equipped with the standard orthogonal form of signature $(4,1)$. It is well known that the orthogonal group can then be identified as
\begin{equation*}
\O(\FF_3^5) \cong W(E_6) \times \{ \pm 1 \}.
\end{equation*}
An element $[w] \in \PP(\FF_3^5)$ is called {\em isotropic}, {\em short} or {\em long}, depending on whether the norm of $w$ equals $0,1$ or $2$ (this does not depend on the chosen representative in $\FF_3^5$).
By \cite[\S2]{AllFrei} (also \cite{ACTsurf}) these elements enumerate the cusps, the components of the discriminant divisor $\widetilde D_{n,m}$, and of the Eckardt divisor $\widetilde R_{m}$,  respectively.
We further know  from \cite[Prop.~2.1]{AllFrei} that there are $40$ isotropic, $36$ short, and
$45$ long elements in $ \PP(\FF_3^5)$, and that the Weyl group $W(E_6)$ acts transitively on each of these three sets.

Counting the number of components of the discriminant divisor $\widetilde D_{n,m}$  and of the Eckardt divisor $\widetilde R_{m}$ intersecting an irreducible component $T_0$ of $T_{3A_2,m}$ as above is then an easy enumeration. Indeed, the choice of the cusp, and of $T_0$, means fixing an isotropic element $h\in \PP(\FF_3^5)$, and a straightforward count shows that there are $9$ short and $18$ long vectors orthogonal to $h$, which counts the number of irreducible components of $\widetilde{D}_{n,m}$ and $\widetilde{R}_m$ that intersect~$T_0$.

To describe the intersection of these components of $\widetilde D_{n,m}$  and $\widetilde R_{m}$ with $T_0$, note that since the stabilizer subgroup of $h$ in $W(E_6)$ acts
transitively on the set of all sort (and also on the set of all long) vectors orthogonal to $h$, by symmetry it is enough to understand the intersection with $T_0$ of only one component of $\widetilde D_{n,m}$  and one component of $\widetilde R_{m}$  with
$T_0$, which can be seen, e.g.,  from the standard local coordinates near a boundary component.
All we need is then to understand the involutions which account for the degree $2$ ramification along $\widetilde D_n$ and $\widetilde R$. In the case of the nodal divisor, this is, up to symmetry, an involution on a factor $\PP^1$ which
fixes some point $p\in\PP^1$, and then the restriction of the component of $\widetilde D_{n,m}$ to $T_0$, fixed under this involution, is $p \times (\PP^1)^{\times 2}$, which gives the divisor class
$ \calO_{(\PP^1)^{\times3}}(1,0,0)$. Adding up all $9$ components of $\widetilde D_{n,m}$ that meet $T_0$ we obtain
$ \calO_{(\PP^1)^{\times3}}(3,3,3)$. In the case of the Eckardt divisor, the involution is given by interchanging two of the factors of $(\PP^1)^{\times 3}$, while fixing the third factor. In the case of interchanging the first two factors, the fixed locus is then equal to $\Delta_{3A_2,12} \times \PP^1$, and has
class $ \calO_{(\PP^1)^{\times3}}(1,1,0)$. Summing over all $18$ components of $\widetilde R_m$ which meet $T_0$ we obtain $\calO_{(\PP^1)^{\times3}}(12,12,12)$.
\end{rem}

We can now compare the discrepancies in the pullbacks of the discriminant and Eckardt divisors for the moduli of marked cubic surfaces:
\begin{cor}\label{C:mu}
Let $p_m^*:\oBGm\to \BBGm$ be the natural map. The following hold:
\begin{enumerate}
\item $p_m^* D_{n,m} = \widetilde D_{n,m}+3T_{3A_2,m}$\,;
\item $p_m^* R_m = \widetilde R_m + 12 T_{3A_2,m}$\,.
\end{enumerate}
\end{cor}
\begin{proof}
This follows immediately from \Cref{P:MarkDivRestrict}, by restricting to $T_{3A_2,m}$. Indeed, {\em a priori} we have $p_m^* D_{n,m} = \widetilde D_{n,m}+aT_{3A_2,m}$ for some $a$, and restricting to the component $T_0$ of $T_{3A_2,m}$ gives $0= \widetilde D_{n,m}|_{T_0}+ aT_{3A_2,m}|_{T_0}$, which gives $a=3$ by parts (1) and (2) of \Cref{P:MarkDivRestrict}. The computation for the Eckardt divisor is identical using parts (1) and (3).
\end{proof}
\begin{rem}\label{Rem-doublecheck}
It is interesting to note that the formulas above are compatible with those of \Cref{pro:canonicalbundleontoroidals} that were obtained by general considerations. Specifically, using \eqref{eq:can-mark}, \Cref{cor:discr} and \Cref{C:mu}, we get
\begin{eqnarray*}
K_{\oBGm}&=&p_m^*K_{\BBGm}+T_{3A_2,m}\\
&=& 5\lambda-\frac{2}{3} (\widetilde D_{n,m}+3T_{3A_2,m})+T_{3A_2,m}\\
&=&5\lambda-\frac{2}{3} \widetilde D_{n,m}-T_{3A_2,m}\,,
\end{eqnarray*}
agreeing indeed with  \Cref{pro:canonicalbundleontoroidals}(2).
\end{rem}

The main result of this Section is the computation of the discrepancy of $p^*:\oBG\to \BBG$ in the unmarked case, as this coefficient will be crucial for computing the top self-intersection number of $K_{\oBG}$.

\begin{pro}\label{P:Tor_can}
The canonical bundle of the toroidal compactification of the ball quotient model of the moduli space of cubic surfaces is given by the formula:
\begin{equation*}
K_{\oBG}= p^*K_{\BBG}+16T_{3A_2}\,.
\end{equation*}
\end{pro}
\begin{proof} Since in \eqref{we6diag} we computed the discrepancy for the finite cover $\oBGm\to \BBGm$, a standard computation (see e.g., \cite[2.3]{kollar_sings_13}) gives
$$
K_{\oBG}= p^*K_{\BBG}+aT_{3A_2}
$$
for
$$
a= \frac{1}{r(T_{3A_2,m})}\left((1+\mu_{\widetilde D_{n,m}}+\mu_{\widetilde R_m})+1)\right)-1\,,
$$
where $r(T_{3A_2,m})=1$ is the ramification index for the cover $\oBGm\to\oBG$ along the toroidal boundary divisor (\Cref{P:MarkedTor}(4)), and $\mu_{\widetilde D_{n,m}}=3$ and $\mu_{\widetilde R_m}=12 $ are defined so that $p_m^* D_{n,m} = \widetilde D_{n,m}+\mu_{\widetilde D_{n,m}}T_{3A_2,m}$
and $p_m^* R_m = \widetilde R_m + \mu_{\widetilde R_m} T_{3A_2,m}$ (\Cref{C:mu}).  We conclude that $a= ((1+3+12) +1)-1=16$ as claimed.
\end{proof}
\begin{rem}
For completeness, we note that the analogue of \Cref{C:mu} in the unmarked case is
\begin{eqnarray*}
p^*D_n&=&\widetilde D_n+6T_{3A_2}\,;\\
p^*R&=&\widetilde R+24T_{3A_2}\,.
\end{eqnarray*}
Similarly to \Cref{Rem-doublecheck}, these formulas are compatible with \Cref{P:Tor_can}, \Cref{pro:canonicalbundleontoroidals}(1), and \eqref{eq-KBBG}, giving a double check of our computations.
\end{rem}

\subsection{Self-intersection numbers for the toroidal compactification}
Using the fact that the Baily--Borel compactification is a weighted projective space $\BBG\cong \GIT\cong \PP(1,2,3,4,5)$, and \eqref{eq:Olambda} we conclude:
\begin{cor} \label{cor-BBG}
On $\BBG$, the following holds:
$$\left(K_{\BBG}\right)^4=\frac{(-15)^4}{5!}=\frac{15^3}{2^3}=\frac{3375}{8}$$
and $\lambda^4=\frac{1}{5!\cdot 6^4}=\frac{1}{2^73^55}=\frac{1}{155520}$.\qed
\end{cor}

Using the description of toroidal boundary given by \Cref{P:MarkedTor}, we obtain:
\begin{lem}\label{self-t3a2}
The self-intersection numbers of the toroidal boundary divisors are
\begin{eqnarray}
(T_{3A_2,m})^4&=&-240\quad(\hbox{on }\oBGm)\\
(T_{3A_2})^4&=& -\frac{1}{6^3}\quad(\hbox{on }\oBG)
\end{eqnarray}
\end{lem}
\begin{proof}
By \Cref{P:MarkedTor}, each component of $T_{3A_2,m}$ is the exceptional divisor of the standard blowup of the cone over the Segre embedding $(\PP^1)^{\times 3}\hookrightarrow \PP^7$. It follows that the self-intersection of each such component in $\oBGm$ is $-6$. Taking into account that there are $40$ disjoint such components, the first item follows. The degree of the map $\oBGm\to \oBG$ is $51,840=|W(E_6)|$, and since this covering map is unramified along $T_{3A_2}$, the second claim follows.
\end{proof}
Finally, we can compute the top degree self-intersection of the canonical class on the toroidal compactification.
\begin{teo}\label{thm:volume-oBG} The top self-intersection number of the canonical class on $\oBG$ is
\begin{equation}
(K_{\oBG})^4=\frac{25,589}{2^33^3}=\frac{25,589}{216}\,.
\end{equation}
\end{teo}
\begin{proof}
From \Cref{P:Tor_can}, we get
$$(K_{\oBG})^4=(K_{\BBG})^4+16^4(T_{3A_2})^4\,.$$
Substituting the numbers from \Cref{cor-BBG} and \Cref{self-t3a2}, the conclusion follows.
\end{proof}

\section{The canonical bundle of Kirwan desingularizations}\label{sec:KKirwan}
For addressing the issue of $K$-equivalence of the compactifications, we now need to perform the computations on the GIT side parallel to the ball quotient computations in the previous section. These turn out to be more involved, and we devote this section to the general setup and results on computing the canonical bundle of the Kirwan resolution of a GIT quotient. This does not seem to be available in the literature, and may be of independent interest. After discussing the general case, we specialize to the case of cubic surfaces in \Cref{SS:KMK}.

\subsection*{Setup}
We start with the general setup of a GIT triple $(X,L,G)$, where $X$ is a scheme of finite type (we will continue to work over $\CC$ to simplify notation, but the argument works over any algebraically closed field of characteristic zero, and, with minor adjustments, in positive characteristic, as well), $L$ is an ample line bundle on $X$, and $G$ is a connected reductive linear algebraic group acting on $X$, with $L$ being $G$-linearized.
We will also make the following assumptions that put us into the setup for Kirwan's work:

\begin{itemize}
\item $X$ is a smooth quasi-projective variety;
\item $X^s \ne \emptyset$, i.e.,~the stable locus is non-empty.

\end{itemize}

Note that from the second condition, and say the Luna Slice Theorem, it follows that there is a Zariski dense open subvariety $U\subseteq X^s$ and a finite group $G_X$ such that for all $x\in U$ the stabilizer $G_x\subseteq G$ is isomorphic (although not necessarily equal) to $G_X$; i.e.,~$G_X$ is the stabilizer of some general point of $X$.

We denote by
$$
q: X^{ss}\longrightarrow Y\coloneqq X/\!\!/_LG
$$
the GIT quotient.

\subsection{Canonical classes for GIT quotients}

Since $X^{ss}$ is normal, so is $Y$ (e.g, \cite[Prop.~3.1, p.45]{D03_LOIT}), and so canonical classes $K_{X^{ss}}$ and $K_Y$ are defined. Note that while $K_{X^{ss}}$ is Cartier, recall that there are elementary examples of quotients of smooth varieties by reductive groups that are not $\QQ$-Gorenstein (e.g., \cite[Exa.7.1]{braun_etal}); in other words, $K_Y$ need not be $\QQ$-Cartier.

\subsubsection{The stable locus}

Now let $Y^s\coloneqq X^{s}/G\subseteq Y= X/\!\!/_LG$ be the stable locus, and to fix notation, we have the map
$$
q^s: X^{s}\longrightarrow Y^s= X^s/G\,.
$$
Note that since $X$ is smooth, the Luna Slice Theorem implies that $Y^s$ is \'etale locally the quotient of a smooth variety $U$ by some finite group~$G_i$. In particular, $Y^s$ is $\QQ$-factorial, and there is a well-defined pullback $q^{s*}$ for $\QQ$-Weil divisor classes.

In this situation, we have the following Riemann--Hurwitz lemma:

\begin{lem}[Riemann--Hurwitz for the stable locus]\label{L:RH-DM}

Let $R^s$ be the divisorial locus in $X^s$ that has stabilizer strictly containing $G_X$ (i.e.,~the union of codimension~$1$ irreducible components of the locus of points in $X^s$ where the stabilizer is not isomorphic to $G_X$), let $R^s=\bigcup R_i^s$ be its decomposition into irreducible components, and let $G_{R_i}$ be the stabilizer of a general point of $R_i^s$.
Then
\begin{align*}
K_{X^s}&= q^{s*}K_{Y^s}+\sum (|G_{R_i}|/|G_X| -1)R^s_i\,.
\end{align*}

\end{lem}

\begin{proof}
It suffices to check \'etale locally. The Luna Slice Theorem implies that, up to a smooth factor, the quotient $q^s:X^s\to Y^s$ is \'etale locally equivalent to the quotient $U\to U/G_i$ for a smooth scheme $U$ and some finite group~$G_i$. Computing the canonical bundles is then a standard computation for the ramified cover $U\to U/G_i$ (see, e.g., \cite[pp.63-4]{kollar_sings_13}).
\end{proof}

\subsubsection{Strictly semi-stable locus of codimension at least~$2$}
The computations for the stable locus carry over immediately to the general case, as long as the strictly semi-stable locus is of codimension at least~$2$.
From now on, we will thus assume that $X^{ss}-X^{s} \subseteq X^{ss}$ and $Y-Y^s\subseteq Y$ are codimension at least $2$.
Under this assumption, one can define a pullback $q^*$ on $\QQ$-Weil divisor classes by restricting to the stable locus $Y^s$ (which, as noted above, is $\QQ$-factorial), pulling back to $X^s$, and then extending over the boundary, which is assumed to be of codimension at least $2$. This immediately yields:

\begin{cor}[Riemann--Hurwitz]\label{C:RH-GIT}
Assume that $X^{ss}-X^{s} \subseteq X^{ss}$ and $Y-Y^s\subseteq Y$ have codimension at least $2$.
Then the same conclusion as in \Cref{L:RH-DM} holds:
\begin{align*}
K_{X^{ss}}&= q^{*}K_{Y}+\sum (|G_{R_i}|/|G_X| -1)R_i\,,
\end{align*}
where $R_i$ is the closure of $R_i^s$ in $X^{ss}$. \qed
\end{cor}
\begin{rem}
The codimension at least~$2$ hypothesis above does rule out some standard GIT constructions. For instance, \Cref{C:RH-GIT} is not applicable in the case of the GIT moduli space of cubic curves $\GIT_\curv$, as the locus of strictly semi-stable cubic curves is the locus of $A_1$ cubic curves, which is a codimension~$1$ locus in the semi-stable locus in the Hilbert scheme $\PP^{9}=\PP H^0(\PP^2,\calO_{\PP^2}(3))$. Of course $\GIT_\curv$ is simply equal to $\PP^1$, so this particular situation is trivial.
\end{rem}

\subsection{Canonical bundle of the Kirwan blowup}
Here we consider a step in the Kirwan blowup process:
$$
\xymatrix{
F\ar@{^(->}[rr] \ar[d]&&\widetilde X^{ss}\ar[rr]^{\tilde \pi} \ar[d]^{\tilde q} &&X^{ss} \ar[d]^q &\\
E\ar@{^(->}[rr]&&\widetilde Y= \widetilde X/\!\!/_{\tilde L}G \ar[rr]^{\pi}&&Y= X/\!\!/_LG&\\
}
$$
We refer the reader to Kirwan's various papers on the topic for the details on Kirwan blowups, or to our paper \cite{cohcubics} for a summary. For the convenience of the reader, recall that this includes the data of $\Stab^\circ\le G$, a maximal dimensional connected component of a stabilizer, and the associated locus
$$
Z_{\Stab^\circ}^{ss}\coloneqq \{x\in X^{ss}: \Stab^\circ \text { fixes } x\}\,,
$$
which is a smooth closed subvariety of $X^{ss}$. Note that in Kirwan's papers and in \cite{cohcubics}, $\Stab^\circ$ is denoted by ``$R$''; this would conflict with our notation here, that $R$ is the Eckardt (ramification) divisor, and so we use the (possibly) more transparent $\Stab^\circ$ in the current paper.

\begin{lem}\label{L:KKirw}
Assume that $X^{ss}-X^{s} \subseteq X^{ss}$ and $Y-Y^s\subseteq Y$ are codimension at least $2$.
Let $x\in Z_{\Stab^\circ}^{ss}$ be a general point, let $\calN_x$ be the fiber of the normal bundle to $G\cdot Z_{\Stab^\circ}^{ss}$ in $X^{ss}$ at $x$, let $\ell_x\subseteq \calN_x$ be a general line through the origin, and let $c=\codim_X (G\cdot Z_{\Stab^\circ}^{ss})$. Denote by $\widetilde R_i$ the strict transform of the ramification divisor $R_i$ in $X^{ss}$, denote by $G_{F}\subseteq G_x$ the stabilizer of the general line $\ell_x$, and denote by $\mu_i$ the coefficient defined by $\tilde \pi ^* R_i=\widetilde R_i+\mu_i F$. In terms of these invariants, the canonical bundles admit the following expressions:
\begin{align}
\label{E:LKKirw1}
K_{\widetilde X^{ss}}&= \tilde q^{*}K_{\widetilde Y}+\sum (|G_{ R_i}|/|G_X| -1)\widetilde R_i + (|G_{F}|/|G_X| -1)F = \tilde \pi^* K_{X^{ss}} + (c-1) F\,, \\
\label{E:LKKirw3}
K_{\widetilde Y}& = \pi ^*K_Y+\left(\frac{c+\sum (|G_{ R_i}|/|G_X| -1)\mu_i }{|G_F|/|G_X|}-1\right)E, \ \ \ \text{ if $Y$ is $\QQ$-Gorenstein.}
\end{align}
\end{lem}

\begin{proof}
The first equality in~\eqref{E:LKKirw1} just follows from \Cref{C:RH-GIT}, using that $F$ is the projectivized normal bundle to the orbit $G\cdot Z^{ss}_{\Stab^\circ}\subseteq X^{ss}$, and the stabilizer of a generic point of $F$ is therefore the stabilizer of the projectivized normal space at a generic point. The second equality in~\eqref{E:LKKirw1} is just the formula for the canonical bundle of the blowup of a smooth variety along a smooth subvariety.

For~\eqref{E:LKKirw3}, we use~\eqref{E:LKKirw1}. Indeed, substituting the expressions $K_{\widetilde Y}= \pi ^*K_Y +a(E,Y,0)E$ and $K_{X^{ss}}= q^* K_Y+ \sum (|G_{ R_I}|/|G_X| -1)R_i$, we obtain
\begin{align*}
K_{\widetilde X^{ss}}&= \tilde q^{*}\pi ^*K_{Y}+ a(E,Y,0)\left( |G_F|/|G_X| \right) F+\sum (|G_{ R_i}|/|G_X| -1)\widetilde R_i + (|G_{F}|/|G_X| -1)F\\
&= \tilde q^{*} \pi ^*K_{Y}+ \sum (|G_{ R_i}|/|G_X| -1)\widetilde R_i + \left((a(E,Y,0)+1)|G_{F}|/|G_X| -1\right)F\\
K_{\widetilde X^{ss}}& = \tilde \pi^* q^*K_{Y} + \sum (|G_{ R_i}|/|G_X| -1)\widetilde R_i +\sum (|G_{ R_i}|/|G_X| -1)\mu_i F+ (c-1) F \\
&= \tilde \pi^* q^*K_{Y} + \sum (|G_{ R_i}|/|G_X| -1)\widetilde R_i + \left(c-1 +\sum (|G_{ R_i}|/|G_X| -1)\mu_i \right)F\,.
\end{align*}
Solving for $a(E,Y,0)$ gives the result.
\end{proof}

\begin{rem}
In order to effectively compute invariants on a Kirwan blowup (e.g., cohomology, canonical bundles, etc.), it is useful to be able to make computations directly on $X$, where one in principle has good control of the geometry and group action (as opposed to on the blowups of $X$). The invariants~$\mu_i$ and~$|G_F|$ in \Cref{L:KKirw} are chosen for this reason; note that these can be computed at general points of the strata in $X$ in Kirwan's setup for the Kirwan blowup.
\end{rem}

\begin{rem}\label{R:S-BGLM}
From \eqref{E:LKKirw3}, one can see that $\QQ$-Gorenstein GIT quotients, in our restricted setup, have klt singularities.  This is a special case of a much more general result due to Schoutens \cite[Thm.~2, p.358]{schoutens2005}.  We also note that one can conclude from this that certain moduli spaces of $K$-stable Fano {\em manifolds} are klt; this is again a special case of much deeper results of Braun et al.~\cite{braun_etal} and \cite{LWX18}.
\end{rem}

\begin{rem}[Kirwan blowups with boundary]\label{R:KBBdry}
Still under the assumption that  $X^{ss}-X^{s} \subseteq X^{ss}$ and $Y-Y^s\subseteq Y$ are codimension at least $2$, if we consider the case of a boundary $(Y,\Delta_Y)$ and assume that $K_Y+\Delta_Y$ is $\QQ$-Cartier, then setting  $\Delta_{X^{ss}}\coloneqq q^*\Delta_Y$, and letting  $\Delta_{\widetilde Y}$ be the strict transform of $\Delta_Y$ in $\widetilde Y$, we have
 \begin{align}
\label{E:LKKirwDelta}
K_{\widetilde Y}+\Delta_{\widetilde Y}& = \pi ^*(K_Y+\Delta_Y)+\left(\frac{a(F,X^{ss},\Delta_{X^{ss}})+1+\sum (|G_{ R_i}|/|G_X| -1)\mu_i }{|G_F|/|G_X|}-1\right)E\,.
\end{align}
 Indeed, setting $\Delta_{\widetilde X^{ss}}$ to be the strict transform of $\Delta_{X^{ss}}$ in $\widetilde X^{ss}$,
 we have $\tilde q^*\Delta_{\widetilde Y}= \Delta_{\widetilde X^{ss}}$;  in our situation, both the strict transform and pullback are defined by restricting to the locus where the morphisms are either isomorphisms or \'etale, and then taking closures, and so the strict transform and pullback commute. Then the same analysis as above using
 $$K_{\widetilde Y}+\Delta_{\widetilde Y}= \pi ^*(K_Y+\Delta_Y) +a(E,Y,\Delta_Y)E$$
 $$K_{\widetilde X^{ss}}+\Delta_{\widetilde X^{ss}}=\tilde \pi^*(K_{X^{ss}}+\Delta_{X^{ss}}) +a(F,X^{ss},\Delta_{X^{ss}})F$$
$$K_{X^{ss}}+\Delta_{X^{ss}}= q^*(K_Y+\Delta_Y)+ \sum (|G_{ R_i}|/|G_X| -1)R_i$$
gives \eqref{E:LKKirwDelta}. Note that  from~\eqref{E:LKKirwDelta}, it follows that if $(X^{ss}, \Delta_{X^{ss}} = q^*\Delta_Y)$ is klt, then so is $(Y,\Delta_Y)$; we emphasize that we started by assuming $K_Y+\Delta_Y$ was $\QQ$-Cartier.
\end{rem}

\subsection{Computation of  $K_{\MK}$}\label{SS:KMK}
We now  specialize the general discussion of the previous section to the particular case of the moduli of cubic surfaces. We want to apply~\Cref{C:RH-GIT} to compute $K_{\GIT}$, and then $K_{\MK}$.
To do this for $\GIT$, recall that the locus of unstable points in $\PP^{19}$ has codimension $\ge 2$, and the locus of strictly semi-stable points has codimension $\ge 2$ in the semi-stable locus, both in $(\PP^{19})^{ss}$ as well as in $\GIT$. Similarly, for $\MK$, the locus of unstable points in the blowup of $(\PP^{19})^{ss}$ has codimension $\ge 2$, and the locus of strictly semi-stable points has codimension $\ge 2$ within the semi-stable locus, both in the blowup of $(\PP^{19})^{ss}$ as well as in $\MK$ (see~\cite{cohcubics}). We are thus in the setup of the previous Section.

Recall from \Cref{sec:calM} that the locus of cubics in $\GIT$ with non-trivial stabilizers is the irreducible Eckardt divisor~$R$, and the cubic surface parameterized by a general point of~$R$ has automorphism group~$\ZZ_2$. Thus, the Riemann--Hurwitz formula (\Cref{C:RH-GIT}) in this case gives
$$
K_{(\PP^{19})^{ss}} = q^* K_{\GIT} + R\,.
$$
The class of the Eckardt divisor is known (see \Cref{rem:Eckardt}):

$R=\calO_{\PP^{19}}(100)$, where $\calO_{\PP^{19}}(1)$ is the restriction to $(\PP^{19})^{ss}$ of the hyperplane section in $\PP^{19}$. Thus we have
$$
-\calO_{\PP^{19}}(20) = q^*K_{\GIT} + \calO_{\PP^{19}}(100)\,,
$$
giving
$$
q^*K_{\GIT} = -\calO_{\PP^{19}}(120)\,.
$$
Recall also (e.g., \Cref{rem:Eckardt}) that
the discriminant has class $\calO_{\PP^{19}}(32)$ in our situation.
In other words, as the discriminant in $(\PP^{19})^{ss}$ descends to give the Weil divisor $D_{A_1}$ on $\GIT$, the divisor $\calO_{\PP^{19}}(1)$ descends to $\GIT$ (as a $\QQ$-divisor) to give the $\QQ$-Cartier divisor $\frac{1}{32}D_{A_1}$.
 This finally gives
\begin{equation}\label{E:KGITram}
K_{\GIT} = -\frac{120}{32}D_{A_1} = -\frac{15}{4}D_{A_1}\,,
\end{equation}
agreeing with the computation in~\eqref{E:rat2}.

We now compute $K_{\MK}$ using the same general machinery. Recall that the Kirwan desingularization for the case of cubic surfaces is obtained by a single blowup, supported at the point $\Delta_{3A_2}\in\GIT$ corresponding to the orbit of the $3A_2$ cubic surface~$S_{3A_2}$.

\begin{cor}\label{E:MK_can}
\begin{equation*}
K_{\MK} = \pi^* K_{\GIT} + 20 D_{3A_2}\,.
\end{equation*}
\end{cor}
\begin{proof}
As discussed above, we are in a situation where we can employ \Cref{L:KKirw} to compute the canonical bundle of $K_{\MK}$ via the blow-up of $(\mathbb P^{19})^{ss}$ along the orbit of the $3A_2$ cubic surface. 
The formula \eqref{E:LKKirw3} of \Cref{L:KKirw} states that $K_{\MK} = \pi^* K_{\GIT} +\left(\frac{c+|G_{R}|/|G_X| -1)\mu}{|G_{3A_2}|/|G_X|}-1\right)D_{3A_2}$, where $G_X=\mu_4$ is the stabilizer of a general point of $\mathbb P^{19}$, $G_R$ is the stabilizer of a general point of the Eckardt divisor, $G_{3A_2}$ is the stabilizer of a general point of the exceptional divisor,  $c$ is the codimension of the $3A_2$ orbit, and $\mu$ is the multiplicity of the Eckardt locus along the $3A_2$ orbit in the sense of \Cref{L:KKirw}.  As a general Eckardt cubic surface has a $\mathbb Z_2$ automorphism group, we have that $|G_{R}|/|G_X|=2$.  \Cref{P:stab--ex}(1) states that $G_{3A2}=G_X=\mu_4$.
In other words,  we have 
$K_{\MK} = \pi^* K_{\GIT} +(c+\mu -1)D_{3A_2}$. 
Since we also have $c=6$ (e.g., \Cref{L:3A2slice}) and $\mu=15$ by  \Cref{L:Eck-no-3A2},  the result  follows.
\end{proof}

\begin{rem}
The coefficient $20$ for $D_{3A_2}$ (or at least the fact that the coefficient is divisible by $5$) is crucial for our computation and proof of non-$K$-equivalence. Note that it follows immediately that while $\pi :\MK\to \GIT$ is a blowup supported at the smooth point $\Delta_{3A_2}\in \GIT$, it is {\em not} the standard blowup of the point, since otherwise $K_{\MK}$ would equal~$\pi ^*K_{\GIT}+3D_{3A_2}$.
\end{rem}

\subsection{Intersection numbers for divisors on DM stacks}
 On DM stacks, the top self-intersection numbers of canonical classes are rational numbers. We need the following statement regarding the denominators that may appear, which is essentially \cite[Prop.~2.1.1]{AGV08}:
\begin{pro}\label{P:AGV}
Let $\calX$ be a smooth proper DM stack over a  field $K$ of characteristic $0$ with coarse moduli space $\Phi:\calX\to X$ of dimension $d$. For each geometric point $p:\Spec \overline K \to \calX$, denote by~$e_p$ the exponent of the automorphism group of~$p$, and denote by~$e$ the least common multiple of the numbers~$e_p$ for all geometric points of~$\calX$. For any divisor classes $D_1,\dots,D_d\in \operatorname{CH}_{d-1}(X)$, we have
$$
D_1\cdots D_d \in \frac{1}{e^d}\ZZ\,.
$$
\end{pro}

\begin{proof}
We first recall that for line bundles $M_1,\dots,M_d$ on $X$, with divisor classes $[M_1], \dots, [M_d]\in \operatorname{CH}_{d-1}(X)$, we have by definition
$$
[M_1]\cdots [M_d] \coloneqq  \int_X c_1(M_1) \cdots c_1(M_d)\cap [X] \in \ZZ \,.
$$
Using the fact that $X$ is $\QQ$-factorial, this allows us to define the intersection number $D_1\cdots D_d$ as a rational number for any divisor classes $D_1,\dots,D_d\in \operatorname{CH}_{d-1}(X)$.

To prove the bound on the denominators, we argue as follows. For any divisor $D\in \operatorname{CH}_{d-1}(X)$, the pullback $\Phi^*D$ is the divisor class associated to some line bundle $\calL$ on $\calX$ (given an \'etale presentation $P:U\to \calX$, the pullback $P^* \Phi^*D$ is the class of a line bundle on $U$, which descends to $\calX$). The fact that $\calL^{\otimes e}$ descends to a line bundle $M$ on $X$ is \cite[Lem.~2.1.2]{AGV08}. Thus we have $(eD_1) \cdots (e D_d) \in \ZZ$, completing the proof.
\end{proof}

\section{Proof of the non-$K$-equivalence} \label{sec:proofnonKequiv}
We can now conclude that $\MK$ and $\oBG$ are not $K$-equivalent, which is one of our main results. All the work for this proof has been already done, and we just gather the pieces here. We recall that the top self-intersection numbers of the canonical class on $K$-equivalent varieties are equal --- this follows from say \cite[Ch.~IV App.~Prop.~2.11, p.296]{kollarRC}, which implies that if $f:Y\to X$ is a morphism of schemes of dimension $d$ over a field, such that $f_*(\calO_Y) = \calO_X$, and $D_i$ are $\QQ$-Cartier divisors on $X$,  then $f^*D_1\cdots f^*D_d = D_1\cdots D_d$.
Consequently, given two $K$-equivalent birational normal projective $\QQ$-Gorenstein varieties $X$ and $Y$ of dimension $d$,  one has from the definition of $K$-equivalence that, in the notation of diagram \eqref{E:Kequiv}, $K_X^d=g^*K_X^d= K_Z^d=h^*K_Y^d=K_Y^d$.

We will thus prove \Cref{T:mainNonK} by showing that the top self-intersection numbers of the canonical classes on the Kirwan and toroidal compactifications are not equal to each other. Showing that the top self-intersection numbers of $K_{\MK}$ and of $K_{\oBG}$ are not equal to each other
is greatly simplified by the fact that both of these spaces admit blowdown maps ($\pi$ and $p$ respectively) to the same space $\BBG=\GIT$, with the exceptional divisors of $\pi$ and $p$ both contracted to a single point~$\Delta_{3A_2}$. Thus, as computed above, the canonical class in each case is the pullback of the canonical bundle $K_{\GIT}$ plus the exceptional divisor of the blowup with a suitable multiplicity. Since the exceptional divisors of $\pi$ and $p$ are both contracted to a point, the top self-intersection numbers of $K_{\MK}$ and $K_{\oBG}$ will each equal to the top self-intersection number of $K_{\GIT}$ plus a suitable multiple of the top self-intersection number of the corresponding exceptional divisor.

\begin{proof}[Proof of \Cref{T:mainNonK}]
By the above discussion, it suffices to show that $K_{\MK}^4\ne K_{\oBG}^4$.
To this end, recall from \Cref{E:MK_can} and \Cref{P:Tor_can} that
\begin{align*}
K_{\MK}& = \pi^* K_{ \GIT}+20D_{3A2}\,,\\
K_{\oBG} & = p^* K_{\GIT}+16T_{3A_2}\,.
\end{align*}
The top self-intersection number for the latter is then
$$
K_{\oBG}^4 =K_{\GIT}^4-\tfrac{16^4}{6^3}
$$
by \Cref{self-t3a2}.
We then compute
\begin{equation}
K_{\MK}^4= K_{\GIT}^4+20^4D_{3A_2}^4=K_{\GIT}^4+5^42^8D_{3A_2}^4\,.
\end{equation}
By \Cref{P:AGV} we have $D_{3A_2}^4\in \frac{1}{e}\ZZ$ for some $e$ that is not divisible by $5$; indeed, $D_{3A_2}^4= (D_{3A_2}|_{D_{3A_2}})^3$, where we are using that $\MK$ is $\QQ$-factorial to define the restriction. By \Cref{P:stab--ex}, no stabilizers along $D_{3A_2}$ have order divisible by $5$. Consequently,
the two intersection numbers $K_{\MK}^4$ and $K_{\oBG}^4$ cannot be the same since if we write $5^42^8D_{3A_2}^4$ as a product of nonzero powers of distinct primes, $5$ will appear with positive exponent, while it does not appear with positive exponent in $-\tfrac{16^4}{6^3}=-\tfrac{2^{13}}{3^3}$.
\end{proof}

\section{Proof of equality in the Grothendieck ring of varieties} \label{sec:proofLequiv}
In this section we show that $\MK$ and $\oBG$ have the same class in the Grothendieck ring of varieties, which is our last main result to be proven.  Due to the fact that moduli spaces of cubic surfaces have no odd degree cohomology, this will turn out to give another, conceptual rather than computational, proof that these two compactifications have the same cohomology (\Cref{C:BettiEq}).

\begin{proof}[Proof of \Cref{T:mainL}]
Since $\MK- D_{3A_2}\cong \oBG- T_{3A_2}$, by additivity of the class of the disjoint union of varieties in the Grothendieck ring, it suffices to show that the exceptional divisors $D_{3A_2}$ and $T_{3A_2}$, as varieties, are equivalent in the Grothendieck ring. Since we have seen that $T_{3A_2}$ is isomorphic to $\PP^3$, it thus suffices to show that $D_{3A_2}$ is equivalent to $\PP^3$ in the Grothendieck ring.  This is accomplished below in \Cref{L:D3A2class} (for an alternative approach and proof, see \Cref{L:D3A2toric} and \Cref{R:D3A2class}).
\end{proof}

We now describe explicitly the geometry of the exceptional divisor $D_{3A_2}\subseteq \MK$, by thinking of it as the GIT quotient of the exceptional divisor $\PP^5$ of the blowup of the Luna slice by the
group $\Aut(S_{3A_2})$. Recall that $\Aut(S_{3A_2})$ is the semidirect product of $\TT^2$ and $S_3$, as described in~\Cref{L:R3A2Norm2}. Furthermore, the~$\CC^6$ Luna slice is
described in~\Cref{L:3A2slice}, and the action of $\Aut(S_{3A_2})$ on it is given by~\eqref{E:LunaAction}. We are interested in the action of $\Aut(S_{3A_2})$ on the exceptional divisor
of the
blowup $\Bl_0\CC^6\to\CC^6$ of the Luna slice at the origin. Recall that the action of~$S_3$ on the exceptional~$\PP^5$ is simply by permuting the homogeneous
coordinates $(T_0:T_1:T_2:T_\hzero:T_\hone:T_\htwo)$ pairwise. Our goal is to describe all semi-stable orbits, up to the action of~$S_3$. Recall that the set of unstable orbits has already been determined in \Cref{L:LunaStability}.

We will describe sets of various orbits of the same type, where by type we will simply mean which of the homogeneous coordinates~$T$ vanish (and we will write $0$ for them), and which of the coordinates~$T$ do not vanish (and we will write $*$ for an arbitrary non-zero complex number). For such a given type of an orbit, we will parameterize the orbits of such a type by a number of copies of $\CC^*$, for the non-zero coordinates, quotiented by a finite group. Indeed, for all orbits of a given type we can set one of the non-zero (labeled $*$) homogeneous coordinates $T$ to be equal to 1 by projectivizing, and then set two more of the non-zero homogeneous coordinates to be
equal to 1 by acting by $\TT^2\subseteq\Aut(S_{3A_2})$. However, a given $\TT^2$ orbit may contain more than one point where these three chosen coordinates are equal to 1, due to finite stabilizers that we will describe in detail in~\Cref{L:D3A2class} below.

To enumerate possible types of orbits, up to the~$S_3$ action, we will always permute the
coordinates $T_0,T_1,T_2$ in such a way that all zeroes precede all non-zero coordinates. If possible to do so while keeping this condition for $T_0,T_1,T_2$, we will further permute $T_\hzero,T_\hone,T_\htwo$ to also put all zeroes before all the non-zero coordinates.
Recalling from~\Cref{L:LunaStability} that semi-stable points on $\PP^5$ are those where $T_i\ne 0$ or $T_{\widehat{i}}\ne 0$, for each $i=0,1,2$, we thus enumerate the types of
semi-stable (in fact all of them stable)~$\TT^2$ orbits on~$\PP^5$, and their contributions to the class $[D_{3A_2}]$ in the Grothendieck ring of varieties in \Cref{F:1} (where we number the types of orbits, for easy reference).

\begin{table}[h]
\begin{tabular}{r|ccc|ccc|l}
\hline
Number&$T_0$&$T_1$&$T_2$&$T_\hzero$&$T_\hone $&$T_\htwo $&Contribution to $[D_{3A_2}]$\\
\hline
(1)&$*$&$*$&$*$   &$*$&$*$&$*$&preserved setwise by~$S_3$, so $[(\CC^*)^3/S_3]$\\
(2)& 0 &$*$&$*$   &$*$&$*$&$*$&preserved setwise by $1\leftrightarrow 2$, so $[(\CC^*)^2/S_2]$ \\
(3)&$*$&$*$&$*$  & 0 &$*$&$*$&preserved setwise by $1\leftrightarrow 2$, so $[(\CC^*)^2/S_2]$ \\
(4)& 0 & 0 &$*$   &$*$&$*$&$*$&$[\CC^*]$ (preserved setwise  by $0\leftrightarrow 1$)\\
(5)& 0 &$*$&$*$   &$*$&$0$&$*$&$[\CC^*]$ \\
(6)&$*$&$*$&$*$   & 0 & 0 &$*$&$[\CC^*]$ (preserved setwise  by $0\leftrightarrow 1$)\\
(7)& 0 & 0 & 0    &$*$&$*$&$*$&$[\CC^0]$\\
(8)& 0 & 0 &$*$  &$*$&$*$& 0 &$[\CC^0]$\\
(9)& 0 & 0 &$*$  &$*$&$*$& 0 &$[\CC^0]$\\
(10)&$*$&$*$&$*$  & 0 & 0 & 0 &$[\CC^0]$\\
\hline
\end{tabular}
\caption{Stable~$\TT^2$ orbits on~$\PP^5$}\label{F:1}
\end{table}

Here we note that the quotient of $\CC^*$ by any finite subgroup is still a $\CC^*$, so that contributions from (4) and (6) are equal to $[\CC^*]$. 
Thus, to compute the
class of $D_{3A_2}$ in the Grothendieck ring of varieties it remains to compute the classes of the higher-dimensional quotients, for which we need to understand the actions of~$S_2$ and~$S_3$ appearing in 
the orbits of types (1),(2),(3) above.
\begin{lem}\label{L:D3A2class}
The divisor $D_{3A_2}\subseteq \MK$ is equivalent to $\PP^3$ in the Grothendieck ring of varieties.
\end{lem}
\begin{proof}
From the table of contributions (\Cref{F:1}), and referring to the types of orbits by those numbers, we compute the class in the Grothendieck ring to be
$$
 [D_{3A_2}]=[(1)]+[(2)]+[(3)]+3[\CC^*]+4[\CC^0]\,,
$$
where the first three summands denote the classes in the Grothendieck ring of varieties of the orbits of the corresponding types from \Cref{F:1}. We start with locus (1), noticing first that any~$\TT^2$ orbit of any such point
contains a point whose homogeneous coordinates are $(*:*:*:1:1:1)$ (i.e.,~such that $T_\hzero=T_\hone=T_\htwo=1$). Moreover, looking at the explicit action of~$\diag(\lambda_0,\lambda_1,\lambda_2)\in D''\simeq\TT^2$  (with $\lambda_0\lambda_1\lambda_2=1$) given by~\eqref{E:Thomogeneous}, if $(T_0:T_1:T_2:1:1:1)$ and $(T'_0:T'_1:T'_2:1:1:1)$ lie on the same~$\TT^2$ orbit, it means that there must
exist $(\lambda_0,\lambda_1,\lambda_2)\in\TT^2$ such that $\lambda_0^2=\lambda_1^2=\lambda_2^2$ and such that $T_i=\lambda_i^3 T'_i$ for $i=0,1,2$. This means that each of the three $\lambda_i$ must
be equal to some $\sigma_i=\pm 1$, subject to the condition that the product of the three signs is equal to $+1$, and the values of $T_i$ and $T_i'$ must then differ by the corresponding signs. Thus the set of
orbits of this form is equal to
$$
 (\CC^*)^3/G\,,
$$
where $G$ is the subgroup of $\mu_2\times\mu_2\times \mu_2$ given by the condition $\sigma_0\sigma_1\sigma_2=1$, and the action is by diagonal multiplication. The action of~$S_3$ on~$(\CC^*)^3/G$ is
induced by permuting the coordinates on $(\CC^*)^3$, and we note that it does {\em not} commute with the action of~$G$. For example, acting by signs $(-1,1,-1)$ maps $T_0:T_1:T_2$ to $-T_0:T_1:-T_2$, and
then permuting $0\leftrightarrow 1$ gives $T_1:-T_0:-T_2$, while first permuting and then acting by signs gives $-T_1:T_0:-T_2$.

By taking the squares of the coordinates, we observe that the quotient $(\CC^*)^3/(\mu_2\times\mu_2\times\mu_2)$, where the action is by multiplication by three independent signs, is isomorphic to $(\CC^*)^3$. Furthermore, recall that the quotient $\CC^3/S_3$, under the action that permutes coordinates, is identified with $\CC^3$ by taking elementary symmetric polynomials, i.e.,~the bijection $\CC^3/S_3\simeq \CC^3$ is given in coordinates by
$$
(x_1,x_2,x_3)\mapsto (x_1x_2x_3\,,\ x_1x_2 + x_2x_3 + x_1x_3\,,\ x_1 + x_2 + x_3)\,.
$$
By inspection, the image of $(\CC^*)^3\subseteq \CC^3$ under this bijection is $\CC^*\times\CC^2\simeq (\CC^*)^3/S_3$. Altogether, this means that the map
$$
 (y_1,y_2,y_3)\mapsto (y_1^2y_2^2y_3^2\,,\ y_1^2y_2^2 + y_2^2y_3^2 + y_1^2y_3^2\,,\ y_1^2 + y_2^2 + y_3^2)
$$
identifies the quotient $(\CC^*)^3/(\mu_2^{\times 3})\rtimes S_3$ with $\CC^*\times\CC^2$. The contribution $[(1)]$ to the class $[D_{3A_2}]$ is $[(\CC^*)^3/G\rtimes S_3]$. We now claim that the double cover
$$
(\CC^*)^3/G\rtimes S_3\to (\CC^*)^3/\mu_2^{\times 3}\rtimes S_3
$$
is \'etale. Indeed, to prove this we need to check that no element of $((\mu_2)^{\times 3}\rtimes S_3)- (G\rtimes S_3)$ stabilizes any point in the domain of this map. Indeed, up to renumbering the coordinates, we need to worry about the permutation being the identity, an involution $0\leftrightarrow 1$ or the cycle $0\mapsto 1\mapsto 2\mapsto 0$, and the signs can either all be minus, or just one sign can be minus. We thus check case by case that there are no fixed points:
$$
\begin{aligned}
 (T_0,T_1,T_2)&=(-T_0,T_1,T_2)&\Rightarrow T_0=0\,\,\\
 (T_0,T_1,T_2)&=(-T_0,-T_1,-T_2)&\Rightarrow T_0=T_1=T_2=0\,\,\\
 (T_0,T_1,T_2)&=(-T_1,T_0,T_2)&\Rightarrow T_0=-T_1=-T_0\Rightarrow T_0=T_1=0\,\,\\
 (T_0,T_1,T_2)&=(T_1,T_0,-T_2)&\Rightarrow T_2=0\,\,\\
 (T_0,T_1,T_2)&=(-T_1,-T_0,-T_2)&\Rightarrow T_2=0\,\,\\
 (T_0,T_1,T_2)&=(-T_1,T_2,T_0)&\Rightarrow T_0=-T_1=-T_2=-T_0\Rightarrow T_0=0\,\,\\
 (T_0,T_1,T_2)&=(-T_1,-T_2,-T_0)&\Rightarrow T_0=-T_1=T_2=-T_0\Rightarrow T_0=0\,,\\
\end{aligned}
$$
so that in each case we deduce that some coordinate must be zero, and thus the fixed point set in $(\CC^*)^3$ is empty. Since the only connected \'etale double cover of $\CC^*\times \CC^2$ is topologically itself (covering along the~$\CC^*$ factor), it follows that $[(1)]=[\CC^*\times\CC^2]$ in the Grothendieck ring.

The contributions to the class in the Grothendieck ring of the orbits of types (2) and (3) are simpler. For (2), similarly to the previous case, we can always find a representative with homogeneous coordinates of the form $(T_0:T_1:0:1:1:1)$, and such a point lies on the same $\TT^2$ orbit as $(T'_0:T'_1:0:1:1:1)$ if and only if $T_0=\pm T'_0$ and $T_1=\pm T'_1$, with the signs chosen independently (as the signs can be compensated by choosing the suitable sign for $\lambda_2$, multiplying by which fixes the zero coordinate~$T_2$ anyway). Thus the set of such orbits is $(\CC^*/\mu_2)^2\simeq(\CC^*)^2$, where the explicit isomorphism is given by squaring each coordinate. Then the action of the coordinate interchange involution $0\leftrightarrow 1$ in these coordinates, as an action on $(\CC^*)^2$, is simply the interchange of the two coordinates, i.e.~the restriction to $(\CC^*)^2\subseteq \CC^2$ of the usual action of $S_2$ interchanging the two coordinates. Under the bijection $\CC^2/S_2\leftrightarrow\CC^2$ given explicitly by
$$
 (x_1,x_2)\mapsto (x_1x_2,x_1+x_2)\,,
$$
the image of $(\CC^*)^2$ is equal to $\CC^*\times\CC$, and thus $[(2)]=[\CC^*\times\CC]$.

Finally, for orbits of type (3), each orbit has a representative of the form $(T_0:1:T_2:1:0:1)$ (we choose this form as it is preserved by the involution $0\leftrightarrow 2$), and such a point lies on the same~$\TT^2$ orbit as $(T'_0:1:T'_2:1:0:1)$ if and only if they are mapped to each other by the action of $\diag(\lambda_0,\lambda_1,\lambda_2)$, which means  we must have $\lambda_0\lambda_1\lambda_2=1$ and $\lambda_0^2=\lambda_2^2=\lambda_1^3=1$. This is to say $\lambda_0=\lambda_2=\sigma=\pm 1$ and $\lambda_1=1$, and thus the set of such orbits is $(\CC^*)^2/\mu_2$, where the action is by multiplying both coordinates by~$-1$ simultaneously. Similarly to orbits of type (2), this action of $\mu_2$ commutes with the action of the coordinate interchange involution, and thus the contribution to the Grothendieck ring of varieties is $[(\CC^*)^2/\mu_2\times S_2]$. To determine this class, we first identify, as above, $(\CC^*)^2/S_2\simeq \CC^*\times \CC$ by using the elementary symmetric functions. Then $\mu_2$ action $(x_1,x_2)\mapsto (-x_1,-x_2)$ acts on elementary symmetric polynomials via $(x_1x_2,x_1+x_2)\mapsto (x_1x_2,-x_1-x_2)$, and thus finally
$$
 [(3)]=[(\CC^*)^2/\mu_2\times S_2]=[\CC^*\times(\CC/\mu_2)]=[\CC^*\times\CC]\,
$$
where~$\mu_2$ acts on~$\CC$ by sign, and the quotient is readily identified with~$\CC$ by taking the square of the coordinate.

Altogether, we thus compute
$$
\begin{aligned}
 [D_{3A_2}]&=[(1)]+[(2)]+[(3)]+3[\CC^*]+4[\CC^0] =[\CC^*\times\CC^2]+2[\CC^*\times\CC]+3[\CC^*]+4[\CC^0]\\
 &=[\CC^*]\cdot \left([\CC^*]^2+2[\CC^*]+[\CC^0]\right)+2[\CC^*]\cdot\left([\CC^*]+[\CC^0]\right)+3[\CC^*]+4[\CC^0]\\ &=[\CC^*]^3+4[\CC^*]^2+6[\CC^*]+4[\CC^0]\,,
\end{aligned}
$$
which is equal to the class $[\PP^3]$, as can be seen by decomposing $\PP^3$ in the usual toric way, as it is the toric variety associated to a tetrahedron, which has $1$~highest-dimensional cell, $4$~faces, $6$~edges, and $4$~vertices.
\end{proof}

One can interpret the computations above as giving a description of  the exceptional divisor $D_{3A_2}$ as the quotient of a toric threefold by the action of $S_3$. Rather than going into the geometry of this action in detail, we sketch an alternative direct toric approach.
Using this alternative approach, we can in fact identify the polytope of this toric threefold and the action of $S_3$ explicitly.
This gives us the added information that the toric threefold is simplicial, and provides an alternate proof of \Cref{T:mainL} (see \Cref{R:D3A2class}).

We first note that by general theory the GIT quotient $\PP^5/\!\!/\TT^2$ is itself a toric variety, see \cite[Ch.~14]{CLS}.
We now describe the polytope giving us this toric variety, and the~$S_3$ action on it.

\begin{lem}[Toric polytope]\label{L:D3A2toric}
Let $P_a$ be the polytope in $\ZZ^3\otimes_\ZZ\RR$ defined by the columns of the following matrix:
\begin{equation}\label{E:Pa}
P_a \longleftrightarrow \left(
\begin{array}{rrrrrrrr}
-\frac{3}{8}&0&\frac{3}{8}&0&\frac{3}{7}&0&0&-\frac{3}{7}\\
1&\frac{1}{7}&-\frac{7}{8}&0&-\frac{8}{7}&-\frac{1}{8}&0&1\\
2&\frac{20}{7}&\frac{11}{4}&2&\frac{20}{7}&2&3&2\\
\end{array}
\right)\
\end{equation}
and consider the action of $S_3$ on $\ZZ^3$ given by
the transposition $\tau$ and the $3$-cycle $\sigma$:
\begin{equation}\label{E:tau-sig}
\tau=\left(\begin{array}{rrr}
-1&0&0\\
5&1&0\\
-2&0&1\\
\end{array}
\right), \ \ \sigma =
\left(
\begin{array}{rrr}
-8&-3&0\\
19&7&0\\
-16&-6&1\\
\end{array}
\right).
\end{equation}
The polytope $P_a$ is a combinatorial cube, the associated toric variety $X_{P_a}$ is simplicial, and the quotient of $X_{P_a}$ by the induced action of $S_3$ is isomorphic to $D_{3A_2}$; i.e.,
$$
D_{3A_2}\cong X_{P_a}/S_3\,.
$$
\end{lem}

\begin{rem}
Strictly speaking, to define $X_{P_a}$  we must first clear denominators to obtain a convex lattice polytope, but scaling the polytope does not affect the abstract variety, it only affects the polarization.
\end{rem}

\begin{proof}
From the Kirwan construction, and the Luna Slice Theorem, identifying the exceptional divisor in the blowup of the Luna slice at the origin with $\PP^5$ gives $D_{3A_2}\cong \PP^5/\!\!/_{\calO_{\PP^5}(1)}\Stab(S_{3A_2})$ (see \Cref{L:3A2slice} and \Cref{L:R3A2Norm2}).
Note that in principle, from the Kirwan construction, one only has that the action of the stabilizer lifts to give a linearization on some positive tensor power of $\calO_{\PP^5}(1)$, but the lift of the action to $\calO_{\PP^5}(1)$ is evident from the given explicit formula for the action, and of course which positive tensor power one uses will not change the GIT quotient.
From the description of the stabilizer in \Cref{L:R3A2Norm2}, we can conclude that $D_{3A_2}\cong (\PP^5/\!\!/_{\calO_{\PP^5}(1)} (\CC^*)^2)/S_3$. In order to keep track of the $S_3$ action, we prefer to use the linearization on $\calO_{\PP^5}(3)$.

Our first goal therefore is to describe the toric variety $\PP^5/\!\!/_{\calO_{\PP^5}(3)} (\CC^*)^2$. To start, we claim that the action of $(\CC^*)^2$ on $\PP^5$ is induced by the natural action on the Luna slice $\CC^6$, so that fixing the inclusion of tori $\CC^*\times (\CC^*)^2\to (\CC^*)^6$ given by the matrix
$$
\gamma=
\left(
\begin{array}{rrrrrr}
1&1&1&1&1&1\\
-3&3&0&-2&2&0\\
-3&0&3&-2&0&2\\
\end{array}
\right)\,,
$$
and the character $\chi:\CC^*\times (\CC^*)^2\to \CC^*$ defined by $\chi(t,\lambda_1,\lambda_2)=t^3$, one has an identification of GIT quotients $\CC^6/\!\!/_\chi (\CC^*\times (\CC^*)^2) = \PP^5/\!\!/_{\calO_{\PP^5}(3)} (\CC^*)^2$ (invariant sections of tensor powers of the trivial line bundle $\CC^6\times \CC$ over $\CC^6$ with respect to the character $\chi$ are canonically identified with the invariant sections of tensor powers of $\calO_{\PP^5}(3)$; see e.g., \cite[Lem.~14.1.1(b)]{CLS}). Using the technique in \cite[Ch.~14]{CLS}, one sees that $\CC^6/\!\!/_\chi (\CC^*\times (\CC^*)^2)$ is the toric variety associated to a $3$-dimensional polytope obtained in the following way.

One defines a lattice $M$ by the exact sequence
$$
\xymatrix{
0\ar[r]& M\ar[r]& \ZZ^6 \ar[r]^\gamma& \ZZ^3\,.
}
$$
Then considering the composition $\CC^*\times (\CC^*)^2\to (\CC^*)^6 \to \CC^*$, where the first map is the inclusion of tori determined by $\gamma$, and the second map of tori is determined by the matrix
$$
a=\left(
\begin{array}{rrrrrr}
-2&-2&-2&3&3&3\\
\end{array}
\right),
$$
one sees that the composition is the character $\chi$ given above. From \cite[Ch.~14]{CLS}, one sees that the quotient $\CC^6/\!\!/_\chi (\CC^*\times (\CC^*)^2)$ is the toric variety associated to the lattice $M$ and the polytope
$$
P_a\coloneqq \{m\in M\otimes_{\ZZ}\RR: e_i(m)\ge -a_i, \ i=1,\dots,6 \}\,,
$$
where $e_i$ is the standard dual coordinate and the $a_i$ are the entries of the vector $a$.
The $S_3$ action on $\CC^6$ (see \Cref{L:3A2slice}) induces the $S_3$ action on $(\CC^*)^6$, and therefore on $\ZZ^6$, and in turn determines an $S_3$ action on $M$.

 The columns of the following matrix give an integral basis of $M$:
 $$
 \left(
 \begin{array}{rrr}
0&0&1\\
-2&0&1\\
-16& -6&1\\
-3&-1&-1\\
0&-1&-1\\
21&8&-1
\end{array}
 \right)
 $$
 and therefore,  identifying $M$ with $\ZZ^3$ using this basis, we may identify $P_a$ as the set of $(a,b,c)\in \RR^3$ such that
 $$
 \begin{array}{rcrcrcr}
&&&&c&\ge&2\\
-2a&&&+&c&\ge&2\\
-16a&-&6b&+&c&\ge&2\\
-3a&-&b&-&c&\ge&-3\\
&&-b&-&c&\ge&-3\\
21a&+&8b&-&c&\ge&-3\\
\end{array}
 $$
From this one can identify $P_a$ with the convex hull of eight vertices, given by the columns of the matrix given in \eqref{E:Pa}.\footnote{We thank Mathieu Dutour Sikiri\'c for computing this for us.}

The polytope $P_a$ is  combinatorially a cube; for instance the first four columns and last four columns give ``top'' and ``bottom'' faces of the combinatorial cube, respectively, and the vectors with last coordinate equal to $2$ give a ``side'' face of the cube.  Considering the dictionary between polytopes and fans (e.g., \cite[p.75]{CLS}), it is elementary to check that the toric variety associated to $P_a$, being a combinatorial cube,  is simplicial.

Finally, in these coordinates, the $S_3$ action on $M$, identified with  $\ZZ^3$, is given by the transposition $\tau$ and  $3$-cycle $\sigma$ given in \eqref{E:tau-sig}.
\end{proof}

\begin{rem}[Class in the Grothendieck ring]\label{R:D3A2class}
In the terminology of \Cref{L:D3A2toric}, we note that the vectors $(0,1,-1)$, $(3,-8,-1)$, $(-3,7,-7)$ define a rank $3$, index $27$ sublattice of $\ZZ^3$ on which $S_3$ acts via $\tau$ and $\sigma$ by the standard permutation of vectors.
Moreover, one can check directly that the action of $S_3$ on the vertices of the combinatorial cube defining $P_a$ corresponds to the standard action of $S_3$ on a cube, fixing two antipodal vertices, and in particular acts by toric automorphisms. In this situation, it is the antipodal vertices $(0,0,3)$ and $(0,0,2)$ of $P_a$ that are  fixed.  At the same time, taking the basis $(0,0,1)$, $(-1,3,0)$, $(1,-2,2)$ for $\ZZ^3$, we have that the action of $S_3$ in these coordinates is given by the matrices
$$
\tau=\left(\begin{array}{ccc}
0&1&0\\
1&0&0\\
0&0&1\\
\end{array}
\right), \ \ \sigma =
\left(
\begin{array}{rrr}
0&1&0\\
-1&-1&0\\
0&0&1\\
\end{array}
\right).
$$
From this, using the dictionary between faces of polytopes and torus orbits,  one can easily work out the action of $S_3$ on all the torus orbits of $X_{P_a}$, as well as their quotients, and consequently, re-derive the class of $X_{P_a}/S_3\cong D_{3A_2}$ in the Grothendieck ring.  For instance, it is immediate from the action of the matrices above that the quotient of the maximal torus $(\CC^*)^3/S_3$ is $\CC^2\times \CC^*$, and similarly, that the two-dimensional tori contribute  $(\CC^*)^2/S_2$ with quotient $\CC\times \CC^*$, and that the one-dimensional tori contribute quotients of $\CC^*$ by subgroups, each giving $\CC^*$. Adding up all of the contributions in the Grothendieck ring gives the same class as $\PP^3$. This gives an alternate proof of \Cref{L:D3A2class}, and therefore of \Cref{T:mainL}.
\end{rem}

\smallskip
It turns out that one can use \Cref{T:mainL} and \Cref{L:D3A2toric} to recover the fact that the Betti numbers of $\MK$ and $\oBG$ are equal. Of course, the Betti numbers of $\MK$ were already computed in \cite{kirwanhyp} and~\cite{ZhangCubic} (see also \cite[(C.2)]{cohcubics} to reconcile the numbers in those two papers), and the Betti numbers of $\oBG$ were computed in our paper \cite[Thm.~C.1]{cohcubics}, but the proof of the following corollary helps to give a more intuitive reason for the agreement of the Betti numbers.

\begin{cor}\label{C:BettiEq}
The Kirwan compactification $\MK$ and the toroidal compactification $\oBG$ of the moduli space of smooth cubic surfaces $\calM$ have the same Betti numbers.
\end{cor}
\begin{proof}
First, we claim that $\MK$ and $\oBG$ have no odd cohomology. This follows from the fact that $\GIT$ has no odd cohomology, together with the decomposition theorem and the fact that the exceptional divisors for $\pi:\MK\to \GIT$ and $p:\oBG\to \GIT$ are quotients of simplicial toric varieties by finite groups  \cite[Thm.~12.3.11]{CLS} (for $\MK$ \Cref{L:D3A2toric} shows that the exceptional divisor $D_{3A_2}$ is a finite quotient of a simplicial toric variety, while for $\oBG$ the exceptional divisor $T_{3A_2}$ is simply equal to $\PP^3$, by \Cref{C:T3A2}). The decomposition theorem also gives $\dim H^2(\MK)=\dim H^2(\GIT)+1= \dim H^2(\oBG)$. Finally, it remains to determine the cohomology in the middle degree 4, and the agreement $\dim H^4(\MK)=\dim H^4(\oBG)$ then follows from the fact that the topological Euler characteristic for cohomology with compact supports is well-defined on the Grothendieck ring.
\end{proof}

\begin{rem}
Note that this also gives a short method of computing the rational cohomology of $\oBG$ and $\MK$. Indeed, $\GIT$ is a weighted projective space, and so has the rational cohomology of $\PP^4$. From the decomposition theorem, and the fact that $\oBG\to \GIT$ has exceptional divisor equal to  $\PP^3$,
it follows that $\dim H^0(\oBG)=\dim H^8(\oBG)= 1$, $\dim H^2(\oBG)=\dim H^6(\oBG) = 2$, and $\dim H^4(\oBG)=2$. This determines the cohomology of $\MK$ via \Cref{C:BettiEq}.
\end{rem}

\appendix

\section{Luna slice computations for the GIT model for cubic surfaces} \label{sec:app}
In this Appendix, we give the detailed proofs of some of the statements from \Cref{sec:calM}. While the results of these computations are crucial for our argument, the method is by explicit computations in local charts on the exceptional divisors, and we have put the calculations here in order not to interrupt the line of thought of our arguments in \Cref{sec:calM}.

\subsection{Proof of~\Cref{L:R3A2Norm2}}\label{S:R3A2Norm2}
While this proof is parallel to the case of the $3D_4$ cubic threefold, which was treated in~\cite[Prop.~B.6]{cohcubics},
we still give the complete details, as the careful identification of the finite groups involved is essential in our argument.

We first determine the stabilizer group $\GL(S_{3A_2})\subseteq\GL(4,\CC)$.
To begin, it is clear that the group
\begin{equation}\label{E:App-GF3A_2-1}
\left\{\left(
\begin{array}{c|c}
\SSS_3&\\ \hline
&\CC^*\\
\end{array}
\right): \lambda_0\lambda_1\lambda_2=\lambda_3^3\right\}\subseteq \GL(4,\CC)
\end{equation}
stabilizes $S_{3A_2}$. We wish to show that the stabilizer is equal to this group. For this, we observe that any symmetry must permute the $3$ singularities of the cubic $S_{3A_2}$, i.e.,~the points $(1:0:0:0)$, $(0:1:0:0)$ and $(0:0:1:0)$. This forces a matrix stabilizing $S_{3A_2}$ to be of the form
$$
\left(
\begin{array}{c|c}
\SSS_3&*\\ \hline
0&\lambda_3\\
\end{array}
\right)\,.
$$
Such a transformation sends the monomial $x_0x_1x_2$ to $(\lambda_0 x_0+*x_3)\cdot (\lambda_1 x_1+*x_3)\cdot (\lambda_2 x_2+*x_3)$, where all the $\lambda$'s are non-zero, and $*$ are the entries of the unknown $1\times 3$ block of the matrix. Furthermore, $x_3$ is sent to $\lambda_3 x_3$. Thus all entries $*$ must be equal to zero, or otherwise applying this transformation to $S_{3A_2}$ would give a cubic with non-zero coefficient of some monomial $x_ax_b x_3$ with $0\le a<b\le 2$. Thus we have deduced that the matrix stabilizing $S_{3A_2}$ must actually be of the form
$$
\left(
\begin{array}{c|c}
\SSS_3&0\\ \hline
0&\lambda_3\\
\end{array}
\right)\,.
$$
Finally it is obvious that any element of the stabilizer satisfies the condition $\lambda_0\lambda_1\lambda_2=\lambda_3^3$. This completes the proof that the stabilizer group $\GL(S_{3A_2}) \subseteq \GL(4,\CC)$ is as claimed. The description of $\Stab(S_{3A_2}) \subseteq \SL(4,\CC)$ follows immediately.

(2) This is immediate since $\Aut(S_{3A_2})$ is naturally a subgroup of $\PGL(4,\CC)$.

(3) We now want to describe the structure of the stabilizer group $\GL(S_{3A_2})\subseteq \GL(4,\CC)$ more precisely. There is clearly a short exact sequence
$$
1\to D\to \GL(S_{3A_2})\to S_3\to 1\,,
$$
where $D$ is the subgroup of diagonal matrices in $\GL(S_{3A_2})$, and the map to $S_3$ is the one taking a generalized permutation matrix to the associated permutation. There is an obvious section $S_3\to \GL(S_{3A_2})$, viewing $S_3$ as block diagonal permutation matrices. This gives the identification
$$
\GL(S_{3A_2})\cong D\rtimes S_3\,,
$$
where the action of $S_3$ on $D$ is to permute the first three entries.
The surjection of $\Stab(S_{3A_2})$ onto $S_3$ can be seen by the matrices
$$
\left(
\begin{array}{cccc}
0&1&0&0\\
1& 0&0&0\\
0&0&\zeta_8^{3}&0\\
0&0&0&\zeta_8
\end{array}
 \right) \qquad
 \left(
\begin{array}{cccc}
0& 0&1&0\\
1& 0&0&0\\
0&1&0&0\\
0&0&0&1
\end{array}
 \right)\,,
$$
where $\zeta_8$ is a primitive $8$-th root of unity.

(4) The identification $\GL(S_{3A_2})^\circ= D$ comes from~\eqref{E:GLStDs3s2-1} and the identification $D\cong \TT^3$. We then obtain
the isomorphism $\GL(S_{3A_2})/\GL(S_{3A_2})^\circ \cong S_3$ from~\eqref{E:GLStDs3s2-1}, as well.

The isomorphism $\Aut(S_{3A_2})/\Aut(S_{3A_2})^\circ \cong S_3$ then comes from~\eqref{E:GLSt3D4Ct}. Indeed, taking connected components of the identity we have
 \begin{equation*}
\xymatrix{
1 \ar[r]& \CC^* \ar[r] \ar@{=}[d]& \GL(S_{3A_2})^\circ \ar@{^(->}[d] \ar[r] & \Aut(S_{3A_2})^\circ \ar[r] \ar@{^(->}[d]& 1\\
1 \ar[r]& \CC^* \ar[r]& \GL(S_{3A_2}) \ar[r] & \Aut(S_{3A_2})\ar[r]& 1\\
}
\end{equation*}
The surjection on the right in the top row is standard (say coming from looking at the dimensions of the connected components), and the identification of the kernels of the two rows is elementary in this case since $\CC^*$ is connected. Then one applies the Snake Lemma.

The short exact sequence~\eqref{E:G3A2L} follows from~\eqref{E:GLStDs3s2-1} and the description $D'\cong \TT^2\times \mu_4$.
\qed

\subsection{Proof of~\Cref{L:Eck-no-3A2}}\label{S:Eck-no-3A2}
The Luna Slice Theorem implies that we can compute in the Luna slice. More precisely, we mean the following. First, in order to have shorter, more standard notation (e.g., to match the discussion of the Luna Slice Theorem in \cite[p.198]{GIT}), let us set $X\coloneqq\PP^{19}$, $G\coloneqq\SL(4,\CC)$, $x\in X$ the point corresponding to the $3A_2$ cubic $S_{3A_2}$, $G_x\coloneqq\Stab(S_{3A_2})$, and $W\subseteq X$ the Luna slice. Then there is an open affine neighborhood $U\subseteq X$ of $x$ such that $U$ is \'etale equivalent to $(G\times W)/G_x$, and since $G\times W\to (G\times W)/G_x$ is a principal $G_x$-bundle (e.g., \cite[Cor.~p.199]{GIT}), it follows that, up to a smooth factor, $U$ is \'etale equivalent to $G\times W$. In other words, we have a diagram
$$
\xymatrix{
G\times W \ar[d]_{\text{$G_x$-bundle}} \ar[rd]^{\text{smooth}}&\\
(G\times W)/G_x \ar[r]^<>(0.5){\text{\'et}}& U\subseteq X
}
$$
Since the Eckardt divisor $R\subseteq X$ is an irreducible effective divisor, preserved by the action of $G$, it is elementary to check that $R|_U$ is the image of $G\times (R\cap W)$ under the map $G\times W\to U$, where we take the reduced induced scheme structure on $R\cap W$. In other words, up to a smooth factor, $R|_U$ is \'etale equivalent to $G\times (R\cap W)$. Then, since $G\times (R\cap W)$ is, up to a smooth factor, \'etale equivalent to $R\cap W$, we conclude that $\mu$ is the multiplicity of $R\cap W$ at the origin.

We shall now apply this to the Luna slice in our case. 
We recall from our discussion in Subsection \ref{subsection:prelimmoduli} that the Eckardt divisor is the divisorial locus containing {\em all} smooth cubics with an extra autiomorphism.
smooth cubic surfaces with extra automorphisms, so it suffices to describe the locus of cubic surfaces in the Luna slice that have extra automorphisms.

Translating into the action of $\SL(4,\CC)$, it suffices to describe the locus of cubic surfaces in the Luna slice with stabilizer group strictly containing the diagonal $\mu_4$. Since we are only interested in this locus in a neighborhood of the $3A_2$ cubic surface, we can use the fact from the Luna Slice Theorem that in a neighborhood of the origin in the Luna slice, the stabilizer groups must be subgroups of $\Stab(S_{3A_2})$.

Since the general point of the Eckardt divisor parameterizes cubic surfaces with a $\ZZ_2$ automorphism group,
we will first describe all subgroups of $\Stab(S_{3A_2})$ whose image in $\PGL(4,\CC)$ is of order $2$. Given such a subgroup, there is a matrix $A$ in the group such that the image of $A$ in $\PGL(4,\CC)$ has order $2$. Let us classify these matrices. First, considering the sequence
$$
0\to \mu_4\to \Stab(S_{3A_2})\to \PGL(4,\CC)\,,
$$
we have $\langle A\rangle \cap \mu_4$ is one of $\{Id\}$, $\{Id,-Id\}$, or $\mu_4$, so that
$$
|A|= 2,4,8\,.
$$
Then, considering the diagram
$$
\xymatrix{
&\langle A\rangle \ar@{^(->}[r]&\Stab(S_{3A_2}) \ar@{^(->}[d] \ar[rd]&&\\
1\ar[r]&D \ar[r]&\GL(S_{3A_2})\ar[r]&S_3\ar[r]&1
}
$$
we see that $A$ must map to an element of $S_3$ of order a power of $2$; i.e.,~either the identity or a transposition.

Thus there are two cases. We have
$$
\text{Case I: }\ \ \ \ A=\left(
\begin{array}{cccc}
\lambda_0& 0&0&0\\
0&\lambda_1&0&0\\
0&0&\lambda_2&0\\
0&0&0&\lambda_3
\end{array}
 \right)\ \ \ \
 \lambda_0\lambda_1\lambda_2=\lambda_3^3, \ \ \lambda_0\lambda_1\lambda_2\lambda_3=1
$$
and, up to changing indices for the transposition,
$$
\text{Case II: }\ \ \ \ A=\left(
\begin{array}{cccc}
0&\lambda_1&0&0\\
\lambda_0& 0&0&0\\
0&0&\lambda_2&0\\
0&0&0&\lambda_3
\end{array}
 \right)\ \ \ \
 \lambda_0\lambda_1\lambda_2=\lambda_3^3, \ \ \lambda_0\lambda_1\lambda_2\lambda_3=-1\,.
$$

We will first consider Case I. In this case, combining the two equations for the $\lambda_i$, we see that $\lambda_3^4=1$. Now let us further subdivide Case I by the order
$|A|$ of $A$; we will denote by $\zeta_n$ a primitive $n$-th root of unity. If $|A|=2$, we have (recalling that we are always assuming that $A$ generates a subgroup of $\PGL(4,\CC)$ of order $2$)
$$
\begin{aligned}
&|A|=2 \implies
A=\left(
\begin{array}{cccc}
\pm 1& 0&0&0\\
0&\pm 1&0&0\\
0&0&\pm 1&0\\
0&0&0&\pm 1
\end{array}
 \right)&\text{two entries $1$ and two entries $-1$}\\
&|A|=4 \implies
A=\left(
\begin{array}{cccc}
i^{i_0}& 0&0&0\\
0&i^{i_1}&0&0\\
0&0&i^{3i_3-i_0-i_1}&0\\
0&0&0&i^{i_3}
\end{array}
 \right) &\text{not all entries $i$ or all $-i$, not all indices even}\\
&|A|=8 \implies
A=\left(
\begin{array}{cccc}
\zeta_8^{i_0}& 0&0&0\\
0&\zeta_8^{i_1}&0&0\\
0&0&\zeta_8^{6i_3-i_0-i_1}&0\\
0&0&0&\zeta_8^{2i_3}
\end{array}
 \right)&\text{$i_0$ and $i_1$ not both even}
\end{aligned}
$$

In Case II, combining the two equations for the $\lambda_i$, we see that $\lambda_3^4=-1$, so that $\lambda_3$ must be a primitive $8$-th root of unity. Consequently, we can rule out the cases $|A|=2$ and $|A|=4$, since then $\lambda_3^4=1$. Thus we must have $|A|=8$, and we start by observing
$$
\left(
\begin{array}{cccc}
0&\lambda_1&0&0\\
\lambda_0& 0&0&0\\
0&0&\lambda_2&0\\
0&0&0&\lambda_3
\end{array}
 \right)\cdot\left(
\begin{array}{cccc}
0&\lambda_1&0&0\\
\lambda_0& 0&0&0\\
0&0&\lambda_2&0\\
0&0&0&\lambda_3
\end{array}
 \right)=\left(
\begin{array}{cccc}
\lambda_0\lambda_1& 0&0&0\\
0&\lambda_0\lambda_1&0&0\\
0&0&\lambda_2^2&0\\
0&0&0&\lambda_3^2
\end{array}
 \right)\,.
$$
We know that $\lambda_3=\zeta_8$ is a primitive $8$-th root of unity. The above shows that $\lambda_0\lambda_1$ is a $4$-th root of unity, so we must have $\lambda_0\lambda_1=\zeta_8^{2n}$ for some $n$. Combining this with $\lambda_0\lambda_1\lambda_2=\lambda_3^3$, we see that $A$ must be of the form
$$
A=\left(
\begin{array}{cccc}
0&\lambda_0&0&0\\
\zeta_8^{2n}/\lambda_0& 0&0&0\\
0&0&\zeta_8^{3-2n}&0\\
0&0&0&\zeta_8
\end{array}
 \right)\,.
$$
Concretely, we have four options (recalling that we are always assuming that $A$ generates a subgroup of $\PGL(4,\CC)$ of order $2$):
\begin{enumerate}
\item[II(i)]
$$
A=\left(
\begin{array}{cccc}
0&\lambda_0&0&0\\
1/\lambda_0& 0&0&0\\
0&0&\zeta_8^{3}&0\\
0&0&0&\zeta_8
\end{array}
 \right)
$$

\item[II(ii)]
$$
A=\left(
\begin{array}{cccc}
0&\lambda_0&0&0\\
\zeta_8^2/\lambda_0& 0&0&0\\
0&0&\zeta_8&0\\
0&0&0&\zeta_8
\end{array}
 \right)
$$
\item[II(iii)]
$$
A=\left(
\begin{array}{cccc}
0&\lambda_0&0&0\\
\zeta_8^4/\lambda_0& 0&0&0\\
0&0&\zeta_8^{-1}&0\\
0&0&0&\zeta_8
\end{array}
 \right)
$$

\item[II(iv)]
$$
A=\left(
\begin{array}{cccc}
0&\lambda_0&0&0\\
\zeta_8^6/\lambda_0& 0&0&0\\
0&0&\zeta_8^{-3}&0\\
0&0&0&\zeta_8
\end{array}
 \right)
$$
\end{enumerate}

We now consider when there can be a divisor with generic point fixed by any family of the matrices above. In Case I, the matrices $A$ form discrete families and so we must just show that each such $A$ has fixed locus of codimension~$2$ or more. This is a case by case analysis. We recall that the action is given by~\eqref{E:LunaAction}.

For the case $|A|=2$, in the notation of \eqref{E:LunaAction} where the action of $A$ is described, the last three coordinates are fixed for the action of $A$, 
and because exactly two of the $\lambda_i/\lambda_3$ are equal to $-1$, we see that the fixed locus is the intersection of the two coordinate hyperplanes given by those $\alpha_i$ being set to zero (where the $\alpha_i$ are the coordinates used in  \eqref{E:LunaAction}), 
thus of codimension at least two. For the case $|A|=4$, we have a matrix
$$
A=\left(
\begin{array}{cccc}
i^{i_0}& 0&0&0\\
0&i^{i_1}&0&0\\
0&0&i^{3i_3-i_0-i_1}&0\\
0&0&0&i^{i_3}
\end{array}
 \right)\,,
 $$
where exactly two of the entries are real, and two are imaginary. Again, looking at the action, we see that the fixed locus is contained in intersections of more than two coordinate hyperplanes and we are done. The case $|A|=8$ is similar.
Indeed, suppose without loss of generality that $i_0$ is odd. Then $2i_0\not\equiv 0\mod 4$, and $3i_0\not\equiv 0 \mod 2$, and thus since $\lambda_3$ is a power of the fourth (not eighths) root of unity, both $(\lambda_0/\lambda_3)^2$ and $(\lambda_0/\lambda_3)^3$ are not equal to identity, so this means the fixed locus must have $\alpha_0=\alpha_\hzero =0$. In other words, there are no divisors in the Luna slice that have a general point fixed by a matrix $A$ in the form of Case I.

We now consider Case II.
Recall that in this case the action (similar to~\eqref{E:LunaAction}) is given by
$$
A\cdot (\alpha_0,\alpha_1,\alpha_2,\alpha_\hzero ,\alpha_\hone ,\alpha_\htwo )=\left(\left(\frac{\lambda_1}{\lambda_3}\right)^3\alpha_1,\left(\frac{\lambda_0}{\lambda_3}\right)^3\alpha_0, \left(\frac{\lambda_2}{\lambda_3}\right)^3\alpha_2,
\left(\frac{\lambda_1}{\lambda_3}\right)^2\alpha_\hone , \left(\frac{\lambda_0}{\lambda_3}\right)^2\alpha_\hzero , \left(\frac{\lambda_2}{\lambda_3}\right)^2\alpha_\htwo \right)\,.
$$

Now assume that we have a general point $(\alpha_0,\dots,\alpha_\htwo )$ of some divisor, fixed by a matrix $A$ in one of the four subcases of Case II. We see that $\alpha_0=0$ if and only if $\alpha_1=0$. Since this is codimension~$2$ (and could not sweep out a divisor while moving $\lambda_0$), we can assume that $\alpha_0$ and $\alpha_1$ are nonzero. Then looking in the first coordinate we have $\alpha_0/\alpha_1=(\lambda_1/\lambda_3)^3$, and in the second coordinate, $\alpha_1/\alpha_0=(\lambda_0/\lambda_3)^3$. Multiplying gives us $\left(\frac{\lambda_1}{\lambda_3}\cdot \frac{\lambda_0}{\lambda_3}\right)^3=1$, so that $(\lambda_0\lambda_1)^3=\lambda_3^6$.
Looking at the list of cases, Case II(i)--(iv), the only option is Case II(ii). Now focusing on Case II(ii), we see that $\lambda_2/\lambda_3=1$. Note that one also sees that $\alpha_\hzero /\alpha_\hone =(\lambda_1/\lambda_3)^2$, and $\alpha_\hone /\alpha_\hzero =(\lambda_0/\lambda_3)^2$, which would imply that $(\lambda_0\lambda_1)^2=\lambda_3^4$, which also holds for Case II(ii).

In other words, if we have a general point $(\alpha_0,\dots,\alpha_\htwo )$ of some divisor that is fixed by a matrix $A$ in the form of Case II, then we must be in Case II(ii), and we must have
$$
\alpha_0/\alpha_1=(\lambda_1/\lambda_3)^3, \ \ \alpha_1/\alpha_0=(\lambda_0/\lambda_3)^3, \text { and }
\alpha_\hzero /\alpha_\hone =(\lambda_1/\lambda_3)^2, \ \ \alpha_\hone /\alpha_\hzero =(\lambda_0/\lambda_3)^2\,.
$$
Since in Case II(ii) we have $(\lambda_1/\lambda_3)=(\lambda_0/\lambda_3)^{-1}$, we see that the above only constitutes two conditions, and cuts out a codimension~$2$ locus. As we vary $\lambda_0$, this indeed sweeps out a divisor.

More precisely, consider the following divisor:
$$
\{\alpha_1^2\alpha_\hzero ^3 - \alpha_0^2\alpha_\hone ^3=0\}\,,
$$
i.e.,~$\left(\frac{\alpha_1}{\alpha_0}\right)^2 = \left(\frac{\alpha_\hone }{\alpha_\hzero }\right)^3$. Then the general point $(\alpha_0,\alpha_1,\alpha_2,\alpha_\hzero , \alpha_\hone , \alpha_\htwo )$ of this divisor is stabilized by a matrix of the form
$$
A_{\lambda_0}=\left(
\begin{array}{cccc}
0 & \lambda_0& 0& 0\\
\frac{\zeta_8^2}{\lambda_0}& 0 & 0& 0\\
0&0&\zeta_8 & 0\\
0&0&0&\zeta_8
\end{array}
\right)\in \Stab(S_{3A_2})\,,
$$
where $\zeta_8$ is a primitive $8$-th root of unity, and $\lambda_0$ is chosen so that $(\frac{\lambda_0}{\zeta_8})^2= \frac{\alpha_\hone }{\alpha_\hzero }$. There are two choices of $\lambda_0$, and one is also free to choose another primitive $8$-th root of unity, but these matrices generate the same subgroup of $\Stab(S_{3A_2})$; note that since the cyclic group $\langle A_{\lambda_0}\rangle $ contains the diagonal $\mu_4$, the image of $\langle A_{\lambda_0}\rangle $ in $\PGL(4,\CC)$ is isomorphic to $\ZZ_2$. In other words, the generic point of the divisor above corresponds to a cubic surface with automorphism group isomorphic to $\ZZ_2$.

Permuting the indices, this gives $3$ divisors of the same type. Each has degree $5$, and is a cone through the origin, and so has multiplicity $5$ at the origin. There are no other divisors in a neighborhood of $0$ in the Luna slice parameterizing cubic surfaces with extra automorphisms. This implies that the multiplicity of the restriction of the Eckardt divisor to the Luna slice, as a reduced variety, is $15$. This completes the proof.
\qed

\medskip

Some of the computations below will use related notation to the above. Specifically, we will perform computations in charts on the blowup $\Bl_0\CC^6\to\CC^6$ of a Luna slice as above. Recall that the blowup is embedded in $\CC^6\times\PP^5$, with coordinates $\alpha$ on the~$\CC^6$ and homogeneous coordinates $T$ on the $\PP^5$. By using the $S_3$ action on $\CC^6$ and its extension to the blowup, it will be enough to work only with two charts on $\Bl_0\CC^6$. The first chart is the chart $U_0$ where $T_0\ne 0$, and the local coordinates in this chart are then
\begin{equation}\label{E:ChartU0}
 (\alpha_0,t_1,t_2,t_\hzero,t_\hone,t_\htwo)\quad\hbox{where}\quad t_i=T_i/T_0,\, t_{\widehat i}=T_{\widehat i}/T_0,\, \alpha_i=\alpha_0t_i,\,
 \alpha_{\widehat i}=\alpha_0t_{\widehat i}\,,
\end{equation}
and similarly for the chart $U_\hzero$ where $T_\hzero\ne 0$.

\subsection{Proof of~\Cref{P:stab--ex}}\label{S:stab--ex}
Let $x\in D_{3A_2}\subseteq \MK$ be a point in the exceptional divisor.  We want to compute the stabilizer $S_x\subseteq \Stab(S_{3A_2})\subseteq \SL(4,\CC)$; i.e., the stabilizer of a point in the exceptional divisor of  $\Bl_{\SL(4,\CC) \cdot [S_{3A_2}]}(\PP^{19})^{ss}$ with orbit corresponding to $x$.  By construction, it suffices to compute the stabilizers of points in the exceptional divisor of the blowup of the Luna slice, with respect to the action of $\Stab(S_{3A_2})$.  Indeed, the $\SL(4,\CC)$ orbit of each point in the exceptional divisor of $\Bl_{\SL(4,\CC) \cdot [S_{3A_2}]}(\PP^{19})^{ss}$ intersects the blowup of the Luna slice, and since the blowups are equivariant, one can check that the stabilizer for the $\SL(4,\CC)$ action on $\Bl_{\SL(4,\CC) \cdot [S_{3A_2}]}(\PP^{19})^{ss}$ is just the stabilizer for a corresponding point on the Luna slice, with respect to the action of $\Stab(S_{3A_2})$.

Therefore, we work with the $6$-dimensional Luna slice from \Cref{L:3A2slice}. In order to obtain the Kirwan blowup we have to blow up the Luna slice at the origin, and we denote the exceptional $\PP^5$ by $E$.
The group $\Stab(S_{3A_2})$ acts on the blowup and we have to analyze the stabilizers of the action of $\Stab(S_{3A_2})$ on the semi-stable locus $E^{ss}$.
For this we recall that the connected component of $\Stab(S_{3A_2})$
is the torus given by $\diag(\lambda_1, \lambda_2, \lambda_3,1)$ with $\lambda_1 \lambda_2 \lambda_3=1$. From \Cref{L:R3A2Norm2}, we have an exact sequence
\begin{equation}\label{equ:stabconncomponent1}
1 \to \TT^2 \to \Stab(S_{3A_2}) \to G(3A_2) \to 1\,,
\end{equation}
where $G(3A_2)$ is a finite group which is an extension of the form
\begin{equation}\label{equ:stabconncomponent2}
1 \to \mu_4 \to G(3A_2) \to S_3 \to 1\,,
\end{equation}
where $S_3$ acts by permuting the coordinates $x_0,x_1,x_2$. We shall now analyze the stabilizers of the action of $\Stab(S_{3A_2})$ on~$E^{ss}$.

(1) This claim follows from \Cref{L:Eck-no-3A2}. Indeed, $E$ is the projectivization of the Luna slice, and in \Cref{L:Eck-no-3A2} it is shown that the general point of the Luna slice has stabilizer $\mu_4$ (more precisely, it is trivial that every point of the Luna slice has stabilizer containing the diagonal $\mu_4$, and in \Cref{L:Eck-no-3A2} it is shown that there is a neighborhood of the origin so that the points with stabilizer group strictly containing $\mu_4$ form a divisor in this neighborhood).

(2) Since the order of the group $G(3A_2)$ is $2^3\cdot 3$, which is not divisible by $5$, it is enough to analyze the stabilizers of the connected component $D''\cong \TT^2$ acting on $E$, with the action as given in~\eqref{E:LunaAction}. There the action is on the Luna slice, but of course this gives the action on the projectivized Luna slice.

Since it will be convenient in the next proof in this Appendix, we prefer to describe the action of the connected component $\Stab(S_{3A_2})^\circ =D'' \cong \TT^2$ on affine charts for the blowup of $\CC^6$ at the origin. Using the $S_3$ symmetry we can assume that either $T_0 \neq 0$ or $T_\hzero \neq 0$, and thus only deal with the two charts $U_0$ and $U_\hzero$ on the blowup. We shall start with the chart $U_0$, with coordinates~\eqref{E:ChartU0}, and focus on the exceptional divisor $\alpha_0=0$ within it. In the chart $U_0$ the action~\eqref{E:LunaAction} is given by
\begin{equation}\label{E:ActU0}
\begin{aligned}
(\lambda_0,\lambda_1,\lambda_2,\lambda_3)&\cdot (\alpha_0,t_1,t_2,t_\hzero , t_\hone , t_\htwo )=\\ &=\left(\left(\frac{\lambda_0}{\lambda_3}\right)^3\alpha_0,\left(\frac{\lambda_1}{\lambda_0}\right)^3t_1,\left(\frac{\lambda_2}{\lambda_0}\right)^3t_2,
\left(\frac{\lambda_3}{\lambda_0}\right)t_\hzero , \left(\frac{\lambda_1^2\lambda_3}{\lambda_0^3}\right) t_\hone , \left(\frac{\lambda_2^2\lambda_3}{\lambda_0^3}\right)t_\htwo \right)&\quad (\hbox{for $D'$})\,\,\\
&=\left(\lambda_0^3\alpha_0,\lambda_0^{-3}\lambda_1^{3}t_1, \lambda_0^{-6}\lambda_1^{-3}t_2,\lambda_0^{-1}t_\hzero ,\lambda_0^{-3}\lambda_1^{2}t_\hone , \lambda_0^{-5}\lambda_1^{-2}t_\htwo \right)& (\hbox{for $D''$})\,,
\end{aligned}
\end{equation}
where in the last equality we are using that in $\Stab(S_{3A_2})^\circ =D'' \cong \TT^2$ we have $\lambda_3=1$, and $\lambda_2=\lambda_0^{-1}\lambda_1^{-1}$. We note in passing that this chart for the blowup is not equivariant with respect to the full stabilizer group $\Stab(S_{3A_2})$, as any element of $S_3$ that does not fix $0$ would not preserve it.

The proof of (1) now becomes a case by case check, determining the group of $(\lambda_0,\lambda_1)\subseteq \TT^2$ that stabilizes a given point on the exceptional divisor, i.e.,~with coordinate $\alpha_0=0$. Recall that \Cref{L:LunaStability} described the unstable locus on this exceptional divisor, and we are only interested in semi-stable (which are in fact all stable, as this is the Kirwan blowup) orbits.

First, we consider the locus where $t_\hzero\ne 0$. Such a point can only be stabilized if $\lambda_0^{-1}=1$. If all four coordinates $t_1,t_2,t_\hone,t_\htwo$ are non-zero, then the stabilizer is trivial. If $t_1=t_2=0$, then stabilizer consists of $\lambda_1$ such that $\lambda_1^2=\lambda_1^{-2}=1$, i.e.,~is the group $\ZZ_2$. If $t_\hone=t_\htwo=0$, then the stabilizer is $\ZZ_3$.

We now deal with the points where the coordinate $t_\hzero=0$. For such a point to be stable,
\Cref{L:LunaStability}
 implies that
 one of the
following pairs of coordinates must both be non-zero: $(t_1,t_2)$, $(t_1,t_\htwo)$, $(t_\hone,t_2)$ or $(t_\hone,t_\htwo)$. For any pair of non-zero coordinates, the
action of $\TT^2$ on this pair of coordinates is given by multiplying them by $(\lambda_0^a\lambda_1^b,\lambda_0^c\lambda_1^d)$.  If $t_1t_2\ne 0=t_\hone=t_\htwo$, then
the stabilizer must satisfy $\lambda_0^{-3}\lambda_1^3=1$ and $\lambda_0^{-6}\lambda_1^{-3}=1$. Thus $\lambda_0=\lambda_1\rho$ from the first equation, for $\rho$ a third root of unity, and
then the second equation gives $\lambda_1^{-6}\lambda_1^{-3}=1$, so that $\lambda_1$ can then be an arbitrary ninth root of unity.
Altogether the stabilizer group is then $\ZZ_3\times \ZZ_9$.

For the other three cases, the powers $(a,b)$ and $(c,d)$ are linearly independent, and at least one of them is a primitive integral vector. Thus in each of these cases the stabilizer of
the pair where these are the only two coordinates is a cyclic group of order $|ad-bc|$. For $t_1t_\htwo \neq 0=t_2=t_\hone$, the powers are $(-3,3)$ and $(-5,-2)$, and thus we
obtain $\ZZ_{21}$. For $t_\hone t_2\ne 0=t_1=t_\htwo$, the powers are $(-6,-3)$ and $(-3,2)$, and thus we again obtain $\ZZ_{21}$ (as we should, since the coordinate interchange $(t_1,t_\hone)\leftrightarrow (t_\hone,t_\htwo)$ interchanges this with the previous case). Finally, for $t_\hone t_\htwo\ne 0=t_1=t_2$, the
powers are $(-3,2)$ and $(-5,-2)$, and thus we obtain $\ZZ_{16}$. The stabilizers for any case where more than two of the coordinates are non-zero is a subgroup of one of these listed groups, and
thus also has order not divisible by 5.

It remains to consider the chart $U_\hzero$ where $T_\hzero \ne 0$, so that the coordinates on this chart are $(t_0,t_1,t_2,\alpha_\hzero,t_\hone,t_\htwo)$ with $\alpha_i=t_i\alpha_\hzero$. Writing down the action~\eqref{E:LunaAction} in these coordinates gives
\begin{equation}\label{E:ActUhzero}
\begin{aligned}
(\lambda_0,\lambda_1,\lambda_2,\lambda_3)&\cdot (t_0,t_1,t_2,\alpha_\hzero ,t_\hone ,t_\htwo )=\\
&=\left(\tfrac{\lambda_0}{\lambda_3}t_0,\tfrac{\lambda_1^3}{\lambda_0^2\lambda_3}t_1,
\tfrac{\lambda_2^3}{\lambda_0^2\lambda_3}t_2,\tfrac{\lambda_0^2}{\lambda_3^2}\alpha_\hzero,
\tfrac{\lambda_1^2}{\lambda_0^2}t_\hone , \tfrac{\lambda_2^2}{\lambda_0^2}t_\htwo \right)&\quad (\hbox{for $D'$})\,\,\\
&=\left(\lambda_0 t_0,\lambda_0^{-2}\lambda_1^{3}t_1,\lambda_0^{-5}\lambda_1^{-3}t_2,\lambda_0^2\alpha_\hzero ,\lambda_0^{-2}\lambda_1^{2}t_\hone , \lambda_0^{-4}\lambda_1^{-2}t_\htwo \right)&\quad(\hbox{for $D''$})\,.
\end{aligned}
\end{equation}
The only points in the exceptional divisor in the chart $U_\hzero$ whose $S_3$ orbit is disjoint from the chart $U_0$ are those where $t_0=t_1=t_2=0$.
For such a point to be stable, we must then have $t_\hone t_\htwo\ne 0$. In this case the action is by $\lambda_0^{-2}\lambda_1^2$ and $\lambda_0^{-4}\lambda_1^{-2}$. From $\lambda_0^{-2}\lambda_1^2=1$ it then follows that $\lambda_0=\lambda_1\sigma$ for some $\sigma\in\pm 1$, and then from the second equation $\lambda_1$ is a 6th root of unity, so we obtain the stabilizer $\ZZ_2\times \ZZ_6$.
\qed

\subsection{Proof of~\Cref{P:discMK}}\label{S:discMK}
We will work in the Luna slice identified in \Cref{L:3A2slice}, with the action of the toric part of the stabilizer on the affine space $\CC^6$ with coordinates $\alpha_0,\alpha_1,\alpha_2, \alpha_\hzero ,\alpha_\hone ,\alpha_\htwo $ given by~\eqref{E:LunaAction}, and $S_3$ permuting pairs of coordinates with the same index.

We first claim that in this Luna slice the discriminant divisor $D_{A_1}$ locally near the origin is given by the equation
\begin{equation}\label{E:discLunaApp}
 (27\alpha_0^2+4\alpha_\hzero ^3)\cdot(27\alpha_1^2+4\alpha_\hone ^3)\cdot (27\alpha_2^2+4\alpha_\htwo ^3)=0\,.
\end{equation}
More precisely, it is known \cite{duPlessisWall} that the global deformations of the $3A_2$ cubic surfaces versally (and independently) unfold the three $A_2$ singularities (i.e., the global-to-local restriction of deformations is surjective).
Since the Luna slice is smooth of dimension $6$, it follows that locally analytically  (or \'etale locally) the Luna slice is just a product of three copies of the standard deformation space for the $A_2$ singularity (see e.g., \cite[\S 3.4, esp.~Fact~3.13]{cubics}). In particular, it follows that the divisor  $D_{A_1}$  in the Luna slice is locally the union of three (divisorial) components, and we denote it $H\coloneqq H_0\cup H_1\cup H_2$, with $S_3$ permuting the components. This shows that locally analytically there are {\em some} coordinates for the Luna slice so that \eqref{E:discLunaApp} is the equation of the discriminant. We claim that \eqref{E:discLunaApp} is in fact the local equation of the discriminant in {\em our given} coordinates from  \Cref{L:3A2slice}.  By the discussion above, it suffices to show that the hypersurface in the Luna slice given by the equation $27\alpha_0^2+4\alpha_\hzero ^3=0$ in our coordinates is contained in the discriminant.  Taking partial derivatives of our family of cubics parameterized by  the Luna slice in  \Cref{L:3A2slice}
$$
\alpha_0x_0^3 +\alpha_1x_1^3 +\alpha_2x_2^3+\alpha_\hzero x_0^2x_3+\alpha_\hone x_1^2x_3 +\alpha_\htwo x_2^2x_3 +(x_0x_1x_2+x_3^3)
$$
one can check that if $27\alpha_0^2+4\alpha_\hzero ^3=0$, then the associated cubic has a singularity at the point $(1:0:0:-\tfrac{3\alpha_0}{2\alpha_\hzero})$, unless $\alpha_\hzero=0$, in which case $\alpha_0=0$, and the associated cubic has a singularity at $(1:0:0:0)$, establishing the claim.

The Kirwan desingularization proceeds by blowing up the origin of this Luna slice~$\CC^6$, and then taking the GIT quotient of the blowup by the stabilizer $\Stab(S_{3A_2})$ given by~\eqref{E:App-GF3D4}. We will determine the strict transform of $H$ in the coordinate charts on the blowup~$\Bl_0\CC^6\subseteq \CC^6\times\PP^5$. As in the previous proof, by symmetry it is enough to work in the chart $U_0$ given by~\eqref{E:ChartU0}, and in the chart $U_\hzero $ where $t_\hzero \ne 0$.

We start with the chart $U_0$, and use~\eqref{E:ChartU0} to express the proper transform of the divisor~$H$ on it as
$$
 \alpha_0^6\cdot (27+4\alpha_0 t_\hzero ^3)\cdot(27t_1^2+4\alpha_0 t_\hone ^3) \cdot (27 t_2^2+4\alpha_0 t_\htwo ^3)=0\,.
$$
In this chart the exceptional divisor of the blowup is given by $\alpha_0=0$, and thus the strict transform of the divisor~$H$ just omits the $\alpha_0^6$ factor above. To compute the local structure of the Kirwan blowup, we now need the Luna slice for the action of the torus $\TT^2\subseteq \Stab(S_{3A_2})$ given by diagonal matrices with $\lambda_0\lambda_1\lambda_2=1=\lambda_3$ in the chart $U_0$, which is given  in~\eqref{E:ActU0}.
One can check directly that the $\CC^4$ given by the two equations $t_1=t_\hone =1$ is the Luna slice for the action of $\TT^2$ in this chart. Thus the intersection of the strict transform of the divisor~$H$ with this Luna slice is given by
$$
 (27+4\alpha_0 t_\hzero ^3)\cdot (27+4\alpha_0) \cdot (27 t_2^2+4\alpha_0 t_\htwo ^3)=0\,,
$$
which intersects the exceptional divisor $\alpha_0=0$ non-transversally. Indeed, the last factor in this strict transform intersects the exceptional divisor as the intersection of the loci $\alpha_0=0$ and $27 t_2^2+4\alpha_0 t_\htwo ^3=0$, which is non-transversal: they intersect along the codimension 2 space $\alpha_0=t_2=0$, but with multiplicity 2. To ascertain the non-transversality in the moduli space, one needs to further check that taking the quotient by $\TT^2$ and by $S_3$ does not cause the intersection to become transverse.
For completeness, and as a cross-check, we will perform this computation in full detail in the chart $U_\hzero $.

We now work in the chart $U_\hzero $ with coordinates $t_0,t_1,t_2,\alpha_\hzero ,t_\hone ,t_\htwo $, and express $\alpha_i=\alpha_\hzero t_i$. Thus the proper transform of the discriminant divisor $H$ becomes
$$
 \alpha_\hzero ^6\cdot (27t_0^2+4\alpha_\hzero )\cdot(27t_1^2+4\alpha_\hzero t_\hone ^3) \cdot (27 t_2^2+4\alpha_\hzero t_\htwo ^3)\,.
$$
In this chart the exceptional divisor of the blowup is given by $\alpha_\hzero =0$, and thus the strict transform of the divisor~$H$ again just omits the $\alpha_\hzero ^6$ factor. To
compute the local structure of the Kirwan blowup, we need the Luna slice for the action of the torus $\TT^2\subseteq \Stab(S_{3A_2})$ given by diagonal matrices with $\lambda_0\lambda_1\lambda_2=1=\lambda_3$
in the chart $U_\hzero $, where the action is given by~\eqref{E:ActUhzero}.
We claim that the $\CC^4$ given by the two equations $t_\hone =t_\htwo =1$ is a Luna slice for this action of $\TT^2$ on $U_\hzero =\CC^6$. Indeed, for any point $t_0,t_1,t_2,\alpha_\hzero ,t_\hone ,t_\htwo \in U_\hzero $ with $t_\hone t_\htwo \ne 0$, there exist $\lambda_0,\lambda_1$ such that $\lambda_0^{-2}\lambda_1^2=t_\hone^{-1}$ and  $\lambda_0^{-4}\lambda_1^{-2}=t_\htwo^{-1}$, and thus acting by these
$(\lambda_0,\lambda_1)$ shows that the orbit contains a point with $t_\hone=t_\htwo=1$.
The same computation also shows that the stabilizer within $\TT^2$ of any point on this $\CC^4$ slice is at most finite.

Restricting the strict transform of the divisor $H$ to this Luna slice for the~$\TT^2$ gives
$$
 (27t_0^2+4\alpha_\hzero )\cdot(27t_1^2+4\alpha_\hzero ) \cdot (27 t_2^2+4\alpha_\hzero )=0\,.
$$
The intersection of the first factor with the exceptional locus $\alpha_\hzero=0$ is along the codimension two locus $t_0=\alpha_\hzero =0$, along which the first factor intersects the exceptional divisor non-transversally (we observe that by~\Cref{L:LunaStability} a generic point in the chart $U_\hzero$ with coordinates $t_0=\alpha_\hzero=0$ is stable, as indeed $T_\hzero\ne 0$).

What remains to check non-transversality is to handle the finite part of the stabilizer $\Stab(S_{3A_2})$, as in principle the quotient of a non-transversal intersection by a finite group may be transversal.
We will thus verify that the subgroup of $\Stab(S_{3A_2})$ that fixes a generic point of the intersection of the strict transform of~$H$ in chart $U_\hzero $ with the exceptional divisor $\alpha_\hzero =0$ is trivial.

Indeed, such a generic point of intersection has coordinates $(t_0=0,t_1,t_2,\alpha_\hzero=0,t_\hone ,t_\htwo )$, as discussed above.
If an element of the group $\Stab(S_{3A_2})$ as described in~\eqref{E:App-GF3D4} fixes this point, then we claim that the image of this element in $S_3$ must be either the identity or
the involution $1\leftrightarrow 2$, which permutes coordinates $(t_1,t_\hone )\leftrightarrow (t_2,t_\htwo )$ (this statement also appears in the proof of~\Cref{L:Eck-no-3A2}).
To see this, note that at a generic point the only zero coordinate in $\PP^5$ is $t_0$, and thus the
image in $S_3$ of an element of a stabilizer fixing a generic point must fix $0\in\lbrace 0,1,2\rbrace$.
For an element of $D'\subseteq \Stab(S_{3A_2})$ (so that its image in~$S_3$ is the identity), the action~\eqref{E:ActUhzero} on the locus $\alpha_\hzero=t_0=0$ restricts to
$$
(t_1,t_2,t_\hone ,t_\htwo )\mapsto (\lambda_1^3\lambda_0^{-2}\lambda_3^{-1}t_1,\lambda_2^3\lambda_0^{-2}\lambda_3^{-1}t_2,\lambda_1^2\lambda_0^{-2}t_\hone , \lambda_2^2\lambda_0^{-2}t_\htwo )\,.
$$
If a general point $t_1,t_2,t_\hone ,t_\htwo $ is mapped to itself, then from $t_\hone $ coordinate being preserved it follows that $\lambda_1=\sigma_1\lambda_0$ for some $\sigma_1\in\{\pm 1\}$, and
then from the preservation of $t_1$ coordinate it follows that $\lambda_1=\lambda_3$. Similarly from the preservation of the $t_\htwo $ coordinate it follows that $\lambda_2=\sigma_2\lambda_0$
with $\sigma_2\in\{\pm 1\}$ , and then from the preservation of the $t_2$ coordinate it follows that $\lambda_2=\lambda_3$. Thus finally $\lambda_1=\lambda_2=\lambda_3=\sigma_1 \lambda_0$.
Furthermore, equation $\lambda_0\lambda_1\lambda_2=\lambda_3^3$ in the description of $\Stab(S_{3A_2})$ in~\eqref{E:App-GF3D4} gives $\sigma_1=1$, so that
finally $\lambda_0=\lambda_1=\lambda_2=\lambda_3$, and since the matrix is in $\SL(4,\CC)$, they must all be equal to the same fourth root of unity, so that as an
element of $\PGL(4,\CC)$ the matrix is equal to the identity.

Finally, for an element of $\Stab(S_{3A_2})$ whose image in~$S_3$ is the involution $1\leftrightarrow 2$, the action is
$$
(t_1,t_2,t_\hone ,t_\htwo )\mapsto (\lambda_2^3\lambda_0^{-2}\lambda_3^{-1}t_2,\lambda_1^3\lambda_0^{-2}\lambda_3^{-1}t_1,\lambda_2^2\lambda_0^{-2}t_\htwo ,\lambda_1^2\lambda_0^{-2}t_\hone )\,.
$$
If this action preserves a point, then
from the coordinate $t_\hone $ being preserved we see that $\lambda_2^2t_\htwo =\lambda_0^2 t_\hone $.
Fix a square root $x=(t_\hone /t_\htwo )^{1/2}$, so that then $\lambda_2=\sigma_2\lambda_0 x$ for some $\sigma_2\in\{\pm 1\}$. From the coordinate $t_\htwo $ being preserved
we obtain $\lambda_1^2\lambda_0^{-2}t_\hone =t_\htwo $ and thus we see that $\lambda_1=\sigma_1\lambda_0 x^{-1}$ for some $\sigma_1\in\{\pm 1\}$. From the product of
coordinates $t_1t_2$ being preserved we see that $\lambda_2^3\lambda_1^3\lambda_0^{-4}\lambda_3^{-2}=1$. Substituting here our expressions for $\lambda_1$ and $\lambda_2$
yields $\sigma_1\sigma_2\lambda_0^2=\lambda_3^2$, so that $\lambda_3=\gamma_3\lambda_0$ with $\gamma_3^2=\sigma_1\sigma_2$. Furthermore substituting the expressions for all $\lambda$'s in
terms of $\lambda_0$ in the condition $\lambda_0\lambda_1\lambda_2=\lambda_3^3$ yields then $\sigma_1\sigma_2=\gamma_3^3=\sigma_1\sigma_2\gamma_3$, so that $\gamma_3=1$ and
thus $\sigma_1\sigma_2=1$. Finally computing the determinant, the condition for the matrix to lie in $\SL(4,\CC)$ gives $\lambda_0\lambda_1\lambda_2\lambda_3=\lambda_0^4(\sigma_1\sigma_2)=-1$
(Note that this is indeed a minus, as the matrix is not diagonal, but includes a transposition! Recall furthermore that we are thinking about the action of $S_3$, and the involution $1\leftrightarrow 2$ interchanges two pairs of coordinates, $t_1 \leftrightarrow t_2$ and  $t_\hone \leftrightarrow t_\htwo$).
Thus finally $\lambda_0=\lambda_3$ must be some eighth roots of unity, but in this case $t_1$ is mapped to
some eighth root of unity times $x^3$ times $t_2$. Since $t_1,t_2,t_\hone ,t_\htwo $, and thus also $x$, were general, $t_1$ cannot be equal to such a product, and thus there is no stabilizer of this form.
\qed

\bibliographystyle{amsalpha}
\bibliography{bib}

\providecommand{\bysame}{\leavevmode\hbox to3em{\hrulefill}\thinspace}
\providecommand{\MR}{\relax\ifhmode\unskip\space\fi MR }
\providecommand{\MRhref}[2]{%
  \href{http://www.ams.org/mathscinet-getitem?mr=#1}{#2}
}
\providecommand{\href}[2]{#2}
\begin{thebibliography}{CMGHL21}

\bibitem[ACT02]{ACTsurf}
D.~Allcock, J.~A. Carlson, and D.~Toledo, \emph{The complex hyperbolic geometry
  of the moduli space of cubic surfaces}, J. Algebraic Geom. \textbf{11}
  (2002), no.~4, 659--724.

\bibitem[AF02]{AllFrei}
D.. Allcock and E.~Freitag, \emph{Cubic surfaces and {B}orcherds products},
  Comment. Math. Helv. \textbf{77} (2002), no.~2, 270--296.

\bibitem[AGV08]{AGV08}
D.~Abramovich, T.~Graber, and A.~Vistoli, \emph{Gromov-{W}itten theory of
  {D}eligne-{M}umford stacks}, Amer. J. Math. \textbf{130} (2008), no.~5,
  1337--1398.

\bibitem[Ale96]{Alexeev}
V.~Alexeev, \emph{Log canonical singularities and complete moduli of stable
  pairs}, arXiv:alg-geom/9608013, 1996.

\bibitem[Ale02]{Alexeev-A}
\bysame, \emph{Complete moduli in the presence of semiabelian group action},
  Ann. of Math. (2) \textbf{155} (2002), no.~3, 611--708.

\bibitem[All03]{allcock}
D.~Allcock, \emph{The moduli space of cubic threefolds}, J. Algebraic Geom.
  \textbf{12} (2003), no.~2, 201--223.

\bibitem[AMRT10]{AMRT}
A.~Ash, D.~Mumford, M.~Rapoport, and Y.S. Tai, \emph{Smooth compactifications
  of locally symmetric varieties}, second ed., Cambridge Mathematical Library,
  Cambridge University Press, Cambridge, 2010, With the collaboration of Peter
  Scholze.

\bibitem[Bat99]{batyrev96}
V.~Batyrev, \emph{Birational {C}alabi-{Y}au {$n$}-folds have equal {B}etti
  numbers}, New trends in algebraic geometry ({W}arwick, 1996), London Math.
  Soc. Lecture Note Ser., vol. 264, Cambridge Univ. Press, Cambridge, 1999,
  pp.~1--11.

\bibitem[BGLM21]{braun_etal}
L.~Braun, D.~Greb, K.~Langlois, and J.~Moraga, \emph{Reductive quotients of klt
  singularities}, arXiv:2111.02812, 2021.

\bibitem[Car72]{Carter}
R.~W. Carter, \emph{Conjugacy classes in the {W}eyl group}, Compositio Math.
  \textbf{25} (1972), 1--59.

\bibitem[CLS11]{CLS}
D.~Cox, J.~Little, and H.~Schenck, \emph{Toric varieties}, Graduate Studies in
  Mathematics, vol. 124, American Mathematical Society, Providence, RI, 2011.

\bibitem[CMGHL19]{cohcubics}
S.~Casalaina-Martin, S.~Grushevsky, K.~Hulek, and R.~Laza, \emph{Cohomology of
  the moduli space of cubic threefolds and its smooth models},
  arXiv:1904.08728, to appear in Mem.~Amer.~Math.~Soc., 2019.

\bibitem[CMGHL21]{cubics}
\bysame, \emph{Complete moduli of cubic threefolds and their intermediate
  {J}acobians}, Proc. Lond. Math. Soc. (3) \textbf{122} (2021), no.~2,
  259--316.

\bibitem[CML09]{CML}
S.~Casalaina-Martin and R.~Laza, \emph{The moduli space of cubic threefolds via
  degenerations of the intermediate {J}acobian}, J. Reine Angew. Math.
  \textbf{633} (2009), 29--65.

\bibitem[CML13]{cml2}
\bysame, \emph{Simultaneous semi-stable reduction for curves with {ADE}
  singularities}, Trans. Amer. Math. Soc. \textbf{365} (2013), no.~5,
  2271--2295.

\bibitem[CPS15]{CPS15}
I.~Coskun and A.~Prendergast-Smith, \emph{Eckardt loci on hypersurfaces}, Comm.
  Algebra \textbf{43} (2015), no.~8, 3083--3101.

\bibitem[DD19]{DoDu}
I.~Dolgachev and A.~Duncan, \emph{Automorphisms of cubic surfaces in positive
  characteristic}, Izv. Ross. Akad. Nauk Ser. Mat. \textbf{83} (2019), no.~3,
  15--92.

\bibitem[Dol82]{dolgachevWP}
I.~Dolgachev, \emph{Weighted projective varieties}, Group actions and vector
  fields ({V}ancouver, {B}.{C}., 1981), Lecture Notes in Math., vol. 956,
  Springer, Berlin, 1982, pp.~34--71.

\bibitem[Dol03]{D03_LOIT}
\bysame, \emph{Lectures on invariant theory}, London Mathematical Society
  Lecture Note Series, vol. 296, Cambridge University Press, Cambridge, 2003.

\bibitem[Dol12]{DoCAG}
\bysame, \emph{Classical algebraic geometry}, Cambridge University Press,
  Cambridge, 2012, A modern view.

\bibitem[dPW00]{duPlessisWall}
A.~A. du~Plessis and C.~T.~C. Wall, \emph{Singular hypersurfaces, versality,
  and {G}orenstein algebras}, J. Algebraic Geom. \textbf{9} (2000), no.~2,
  309--322.

\bibitem[DvG07]{DvG}
E.~Dardanelli and B.~van Geemen, \emph{Hessians and the moduli space of cubic
  surfaces}, Algebraic geometry, Contemp. Math., vol. 422, Amer. Math. Soc.,
  Providence, RI, 2007, pp.~17--36.

\bibitem[DvGK05]{DvGK}
I.~Dolgachev, B.~van Geemen, and S.~Kondo, \emph{A complex ball uniformization
  of the moduli space of cubic surfaces via periods of {$K3$} surfaces}, J.
  Reine Angew. Math. \textbf{588} (2005), 99--148.

\bibitem[Ful93]{fultonToric}
W.~Fulton, \emph{Introduction to toric varieties}, Annals of Mathematics
  Studies, vol. 131, Princeton University Press, Princeton, NJ, 1993, The
  William H. Roever Lectures in Geometry.

\bibitem[GKS21]{GaKeSch}
P.~Gallardo, M.~Kerr, and L.~Schaffler, \emph{Geometric interpretation of
  toroidal compactifications of moduli of points in the line and cubic
  surfaces}, Adv. Math. \textbf{381} (2021), Paper No. 107632, 48. \MR{4214398}

\bibitem[HKT09]{HKT09}
P.~Hacking, S.~Keel, and J.~Tevelev, \emph{Stable pair, tropical, and log
  canonical compactifications of moduli spaces of del {P}ezzo surfaces},
  Invent. Math. \textbf{178} (2009), no.~1, 173--227.

\bibitem[Kaw02]{kawamata_02}
Yu. Kawamata, \emph{{$D$}-equivalence and {$K$}-equivalence}, J. Differential
  Geom. \textbf{61} (2002), no.~1, 147--171.

\bibitem[Kir85]{kirwanblowup}
F.~Kirwan, \emph{Partial desingularisations of quotients of nonsingular
  varieties and their {B}etti numbers}, Ann. of Math. (2) \textbf{122} (1985),
  no.~1, 41--85.

\bibitem[Kir89]{kirwanhyp}
\bysame, \emph{Moduli spaces of degree {$d$} hypersurfaces in {${\mathbb
  P}_n$}}, Duke Math. J. \textbf{58} (1989), no.~1, 39--78.

\bibitem[KM98]{kollarmori}
J.~Koll\'{a}r and Sh. Mori, \emph{Birational geometry of algebraic varieties},
  Cambridge Tracts in Mathematics, vol. 134, Cambridge University Press,
  Cambridge, 1998, With the collaboration of C. H. Clemens and A. Corti,
  Translated from the 1998 Japanese original. \MR{1658959}

\bibitem[Kol96]{kollarRC}
J.~Koll\'{a}r, \emph{Rational curves on algebraic varieties}, Ergebnisse der
  Mathematik und ihrer Grenzgebiete. 3. Folge. A Series of Modern Surveys in
  Mathematics [Results in Mathematics and Related Areas. 3rd Series. A Series
  of Modern Surveys in Mathematics], vol.~32, Springer-Verlag, Berlin, 1996.
  \MR{1440180}

\bibitem[Kol13]{kollar_sings_13}
\bysame, \emph{Singularities of the minimal model program}, Cambridge Tracts in
  Mathematics, vol. 200, Cambridge University Press, Cambridge, 2013, With a
  collaboration of S\'{a}ndor Kov\'{a}cs.

\bibitem[KR12]{KudlaRapoport}
S.~Kudla and M.~Rapoport, \emph{On occult period maps}, Pacific J. Math.
  \textbf{260} (2012), no.~2, 565--581.

\bibitem[LWX18]{LWX18}
C.~Li, X.~Wang, and C.~Xu, \emph{Quasi-projectivity of the moduli space of
  smooth {K}\"{a}hler-{E}instein {F}ano manifolds}, Ann. Sci. \'{E}c. Norm.
  Sup\'{e}r. (4) \textbf{51} (2018), no.~3, 739--772.

\bibitem[MFK94]{GIT}
D.~Mumford, J.~Fogarty, and F.~Kirwan, \emph{Geometric invariant theory}, third
  ed., Ergebnisse der Mathematik und ihrer Grenzgebiete (2), vol.~34,
  Springer-Verlag, Berlin, 1994.

\bibitem[Muk03]{mukai}
S.~Mukai, \emph{An introduction to invariants and moduli}, Cambridge Studies in
  Advanced Mathematics, vol.~81, Cambridge University Press, Cambridge, 2003,
  Translated from the 1998 and 2000 Japanese editions by W. M. Oxbury.

\bibitem[Mum77]{MumHirz77}
D.~Mumford, \emph{Hirzebruch's proportionality theorem in the noncompact case},
  Invent. Math. \textbf{42} (1977), 239--272. \MR{471627}

\bibitem[Nar82]{naruki}
I.~Naruki, \emph{Cross ratio variety as a moduli space of cubic surfaces},
  Proc. London Math. Soc. (3) \textbf{45} (1982), no.~1, 1--30, With an
  appendix by Eduard Looijenga.

\bibitem[Ols04]{olsson}
M.~C. Olsson, \emph{Semistable degenerations and period spaces for polarized
  {$K3$} surfaces}, Duke Math. J. \textbf{125} (2004), no.~1, 121--203.

\bibitem[OSS16]{OSS}
Y.~Odaka, C.~Spotti, and S.~Sun, \emph{Compact moduli spaces of del {P}ezzo
  surfaces and {K}\"{a}hler-{E}instein metrics}, J. Differential Geom.
  \textbf{102} (2016), no.~1, 127--172.

\bibitem[Pop11]{popaNote}
M.~Popa, \emph{Modern aspects of the cohomological study of varieties},
  available at https://people.math.harvard.edu/~mpopa/571/, 2011.

\bibitem[Sch05]{schoutens2005}
H.~Schoutens, \emph{Log-terminal singularities and vanishing theorems via
  non-standard tight closure}, J. Algebraic Geom. \textbf{14} (2005), no.~2,
  357--390.

\bibitem[Zha05]{ZhangCubic}
J.~Zhang, \emph{Geometric compactification of moduli space of cubic surfaces
  and {K}irwan blowup}, ProQuest LLC, Ann Arbor, MI, 2005, Thesis (Ph.D.)--Rice
  University.

\bibitem[Zhe21]{ZhiweiZheng}
Z.~Zheng, \emph{Orbifold aspects of certain occult period maps}, Nagoya Math.
  J. \textbf{243} (2021), 137--156.

\end{thebibliography}

\end{document}